%% file: Arxiv_final.tex
\documentclass[11pt]{amsart}

\usepackage[utf8]{inputenc}
\usepackage{graphicx}
\usepackage{float}
\usepackage{amsthm}
\usepackage{amsmath}
\usepackage{amssymb}
\usepackage[abs]{overpic}
\usepackage{enumerate}
\usepackage{verbatim}
\usepackage{caption}
\usepackage{url}
\usepackage{textcomp}
\usepackage{color}
\usepackage{hyperref}
\usepackage{multicol}
\usepackage{gensymb}
\usepackage{subcaption}
\usepackage{bm}
\usepackage{CJKutf8}
\usepackage{multirow}
\usepackage{tabularx}
\usepackage{transparent}
\usepackage{mathtools}
\usepackage{enumitem}
\usepackage{doi}
\usepackage{appendix}
\usepackage{dsfont}
\usepackage{soul}
\usepackage[abbrev,nobysame]{amsrefs}
\DeclareMathOperator\erf{erf}
\allowdisplaybreaks
\newcounter{descriptcount}
\newlist{enumdescript}{description}{1}
\setlist[enumdescript,1]{
  before={\setcounter{descriptcount}{0}
          \renewcommand*\thedescriptcount{\arabic{descriptcount}}},
        font={\bfseries\stepcounter{descriptcount}Case \thedescriptcount~}
}

\newtheorem{alg}{Algorithm}
\newtheorem{lemma}{Lemma}[section]
\newtheorem{theorem}[lemma]{Theorem}
\theoremstyle{definition}

\newtheorem{Remark}{Remark}[section]
\theoremstyle{plain}
\newtheorem{proposition}[lemma]{Proposition}

\numberwithin{equation}{section}

\title[Anticipating bifurcations through tails of stationary densities]{Anticipating bifurcations of random dynamical systems through tails of stationary densities}

\begin{document}

\author{Wei Hao Tey}

\address{Department of Mathematics, Imperial College London, 180 Queen's Gate, London, SW7 2AZ, United Kingdom.}

\address{International Research Center for Neurointelligence (IRCN), The University of Tokyo, 7-3-1 Hongo Bunkyo-ku, Tokyo, 
113-0033 Japan.}
\email{w.tey18@imperial.ac.uk}

\author{Guillermo Olic\'{o}n-M\'{e}ndez}
\address{Universit\"{a}t Potsdam, Institut f\"{u}r Mathematik, Karl-Liebknecht-Straße 24, 14476 Potsdam.}
\email{guillermo.olicon@uni-potsdam.de}

\author{Jeroen S.W. Lamb}

\address{Department of Mathematics, Imperial College London, 180 Queen's Gate, London, SW7 2AZ, United Kingdom.}

\address{International Research Center for Neurointelligence (IRCN), The University of Tokyo, 7-3-1 Hongo Bunkyo-ku, Tokyo, 
113-0033 Japan.}

\email{jsw.lamb@imperial.ac.uk}

\author{Kazuyuki Aihara}
\address{International Research Center for Neurointelligence (IRCN), The University of Tokyo, 7-3-1 Hongo Bunkyo-ku, Tokyo, 
113-0033 Japan.}
\email{kaihara@g.ecc.u-tokyo.ac.jp}



\begin{abstract}
    We develop an early-warning signal for bifurcations of one-dimensional random difference equations with additive bounded noise, based on the asymptotic behaviour of the stationary density near a boundary of its support. We demonstrate the practical use in numerical examples. 
\end{abstract}

\maketitle

\section{Introduction}

In recent years, there has been a growing interest in developing early warning signals for critical transitions in complex systems \cite{chen2012detecting,scheffer2009early,feng2024early}, referring to abrupt shifts from one dynamical regime to another at some critical threshold. Such transitions are also known as \textit{tipping points}~\cite{scheffer2001catastrophic,scheffer2009early} and are often locally irreversible. Early warning signals serve as precursors, offering advance notice allowing for the implementation of mitigation or adaptation strategies. Examples include the early treatment of pre-symptomatic diseases in medicine \cite{aihara2022dynamical}.

Much of the research in the theory of critical transitions and their early warning signals relies on simple low-dimensional stochastic differential equations with small noise \cite{Ashwinetal12,kuehn13}. Well-known early warning signals obtained in this context include increases in variance \cite{biggs2009turning,carpenter2006rising}, autocorrelation \cite{dakos2008slowing,held2004detection}, crosscorrelation\cite{aihara2022dynamical,chen2012detecting}, and skewness \cite{guttal2008changing}
of the observed data.

It is well-known that in practice, these early warning signals do not always apply \cite{scheffer2009early}. 
Indeed, different mechanisms for critical transitions (other than those in the prototypical SDEs) may arise, with different characteristics. This leads us to address the question which other types of early warning signals can be found in different modelling contexts.

In this paper, we consider random dynamical systems with bounded noise. These are
quite naturally motivated from the viewpoint of applications \cite{d2013bounded}, but less popular in the literature, due to a lack of readily available tools and methodologies.   
Nevertheless, in recent years more attention has been given to systems with bounded noise \cite{homburg2013bifurcations,olicon2021critical,tey2022minimal} and they are known to feature bounded attractors that trap orbits and support stationary distributions, see eg \cite{zmarrou2007bifurcations,lambhenon}. 

When system parameters are varied, attractors may undergo \textit{topological bifurcations}, going hand-in-hand with discontinuous changes of the support (and possibly numbers) of stationary distributions \cite{lamb2015topological}.
Such changes of invariant regions (and their associated statistics) enlarge or restrict the reach of long-term behaviours of orbits. In particular, they allow (or prevent) noise-induced excursions from certain regions of the state space to others.
As such, it is natural to view topological bifurcations of attractors in systems with bounded noise as fundamental to tipping point behaviour. 

The (topological) persistence and bifurcations of bounded noise attractors can be reformulated as a bifurcation problem of an associated set-valued dynamical system. The latter are notoriously hard to analyse, although some recent breakthrough was made by reformulating this problem in terms of a so-called finite-dimensional \textit{boundary map} that allows for the use of conventional bifurcation theoretical tools to approach this problem \cite{kourliouros2023persistence,lambhenon}.  

\begin{figure}[!b]

    \centering

    \begin{overpic}[width=.54\textwidth]{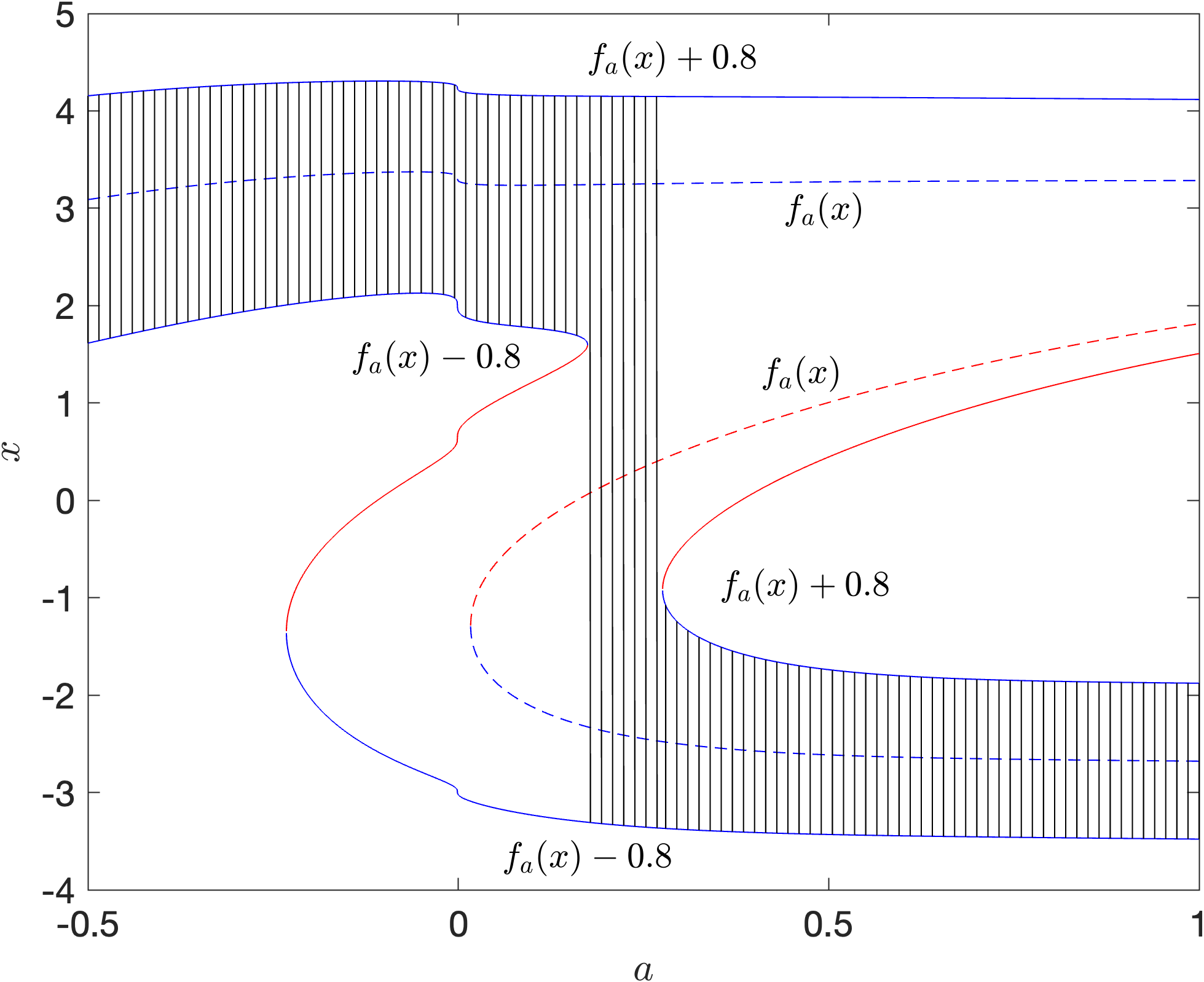}
    \put(0,0){
    (a)
    }\end{overpic}
    \begin{overpic}[width=.45\textwidth]{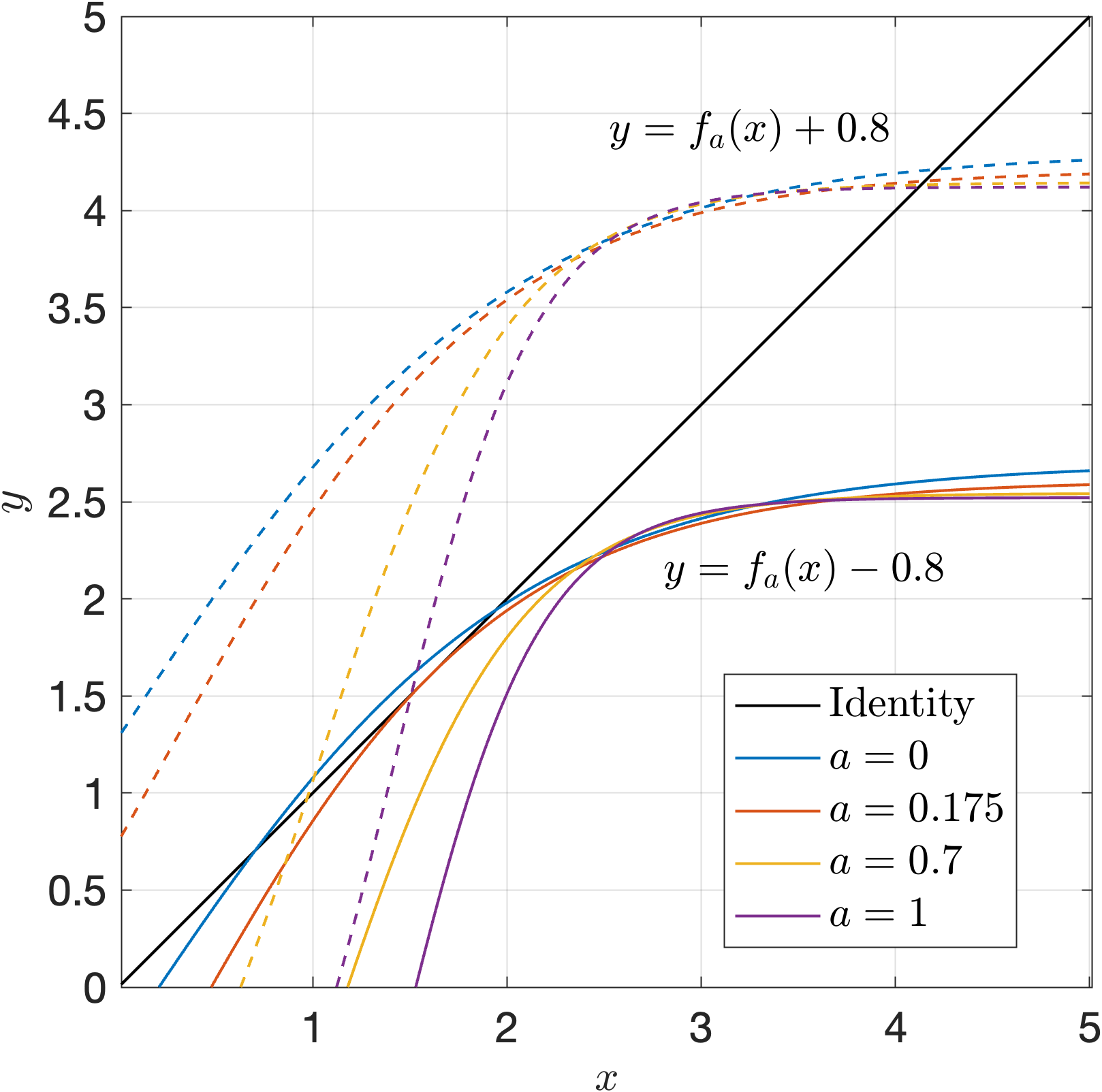}
    \put(0,0){
    (b)
    }
    \end{overpic}
    \captionsetup{width=\linewidth}\caption{Illustration of the one-parameter family of maps $f_a$, cf.~\eqref{eq:variance} in Appendix~\ref{appen:variance_decrease}, together with the upper and lower extremal maps $f_a(x) \pm 0.8$. In (a), the bifurcation diagram of $f_a(x)$ for $-0.5<a<1$ is shown as a dotted line, while those of the extremal maps are plotted as solid lines. The blue lines signify stable fixed points, the red lines represent unstable fixed points, and the black vertical lines are minimal invariant intervals for selected values of $a$.
    In (b), the extremal maps for parameter values $a = 0, \;0.175, \;0.7, \;1$ are portrayed in dotted and solid lines, respectively. There is a topological bifurcation of the minimal invariant interval around $a\approx 0.175$, represented by a fold bifurcation of the lower extremal map, that is when $f_a(x) - 0.8$ is tangential to the diagonal, portrayed in black in (b). The bifurcation can be observed in (a), where the minimal invariant interval changes discontinuously to a larger interval. After the bifurcation, random trajectories can alternate between the two metastable states inside the enlarged minimal invariant interval, in a so-called flickering phenomenon.}
    \label{fig:modified_tanh}
    \end{figure}

In this paper, we consider the simplest setting of one-dimensional diffeomorphisms with bounded additive noise, where the analysis of topological bifurcations is relatively straightforward via the evaluation of so-called \textit{extremal maps}. 
In such systems, the tail behaviour of the stationary distribution near the boundary of an attractor has recently been characterised \cite{olicon2021critical,olicon24tail} and in this paper, we employ these results to formulate new early warning signals for topological bifurcations.

    \begin{figure}[!b]
    \centering
    \includegraphics[width=.5\textwidth,trim=.8 .8 .8 .8,clip]{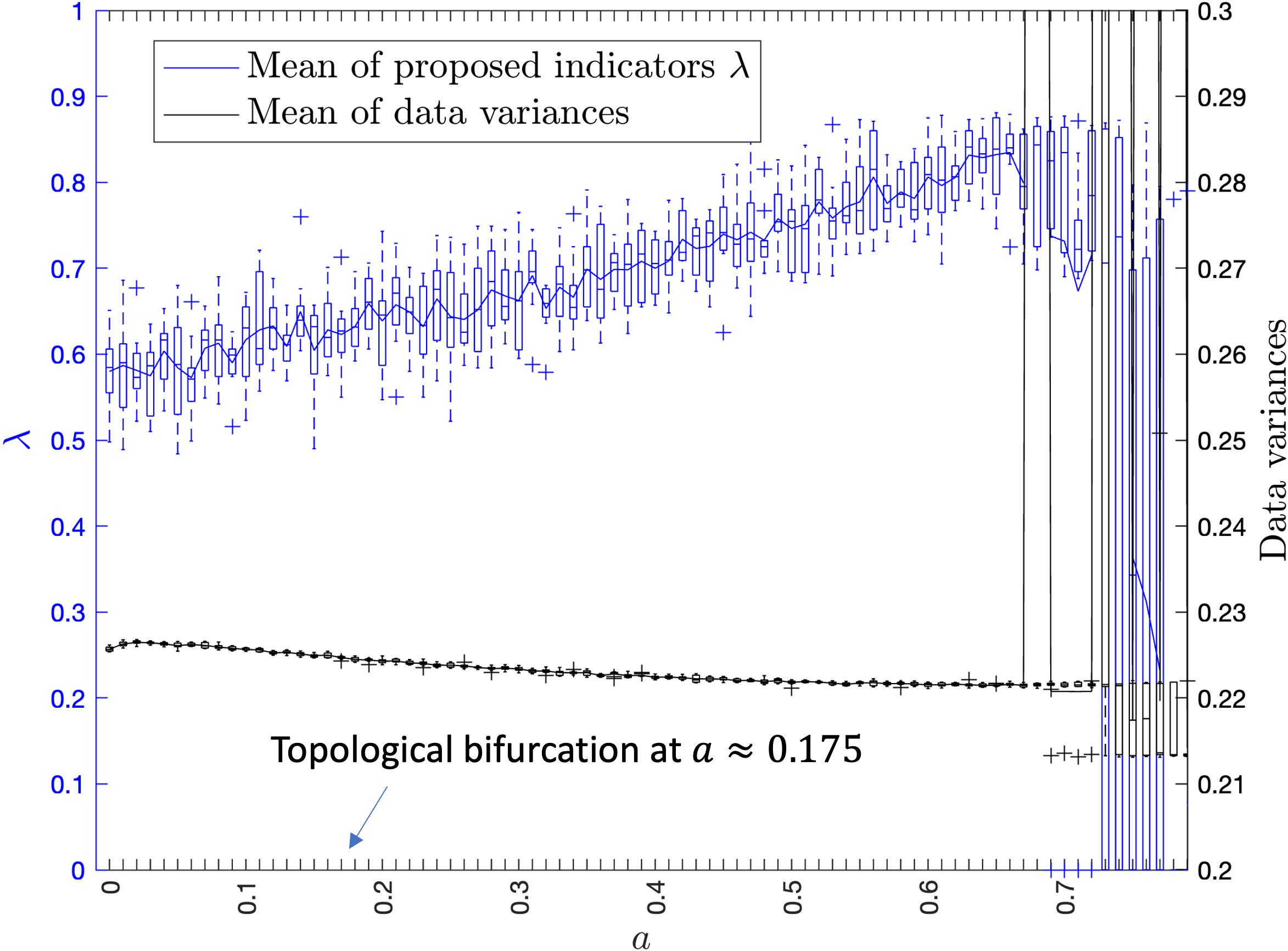}
    \captionsetup{width=\linewidth}\caption{ 
    Results of numerical simulations of the sample variance and our proposed new early warning signal $\hat{\lambda}$ (where $\hat{\lambda}$ is an estimator of the bifurcation parameter $\lambda$, which in turn approaches $1$ as $a\rightarrow a_*$).
    We simulate the random map $y_{t+1} = f_a(y_t) + \xi_t$, given in \eqref{eq:variance} in Appendix~\ref{appen:variance_decrease}, with noise amplitude $|\xi_t|\leq 0.8$ (as in Figure~\ref{fig:modified_tanh}) while varying $a$. This system has a topological bifurcation at $a_*\approx 0.175$, where an attractor explodes when $a$ passes through $a_*$ from below. 
    Starting from $a=0$, a time series of length $10^6$ is generated inside the attractor, after which $a$ is increased by $0.01$, and the time series is continued for another $10^6$ iterates, after which $a$ is increased again, and so on, until $a=0.8$. 
    Ten such time series are generated and we present the measured variances and their mean at each considered fixed value of $a$, above, in black, showing a slowly decreasing trend until the system tips (escapes from the remnant of the attractor). 
    In the same graph, we plot our new indicator $\lambda$, in blue. 
    Boxes represent the interquartile range (25$^{\text{th}}$ to 75$^{\text{th}}$ percentile), with whiskers extending to all data points, except for outliers marked with crosses.} 
    \label{fig:variance}
\end{figure}

We illustrate our result in the context of a concrete example.
Consider the one-parameter family of one-dimensional mappings with additive bounded noise, $y_{t+1} = f_a(y_t) + \xi_t, \; \vert\xi_t\vert \leq 0.8$ for $t \in \mathbb{N}$ and $\{\xi_t\}_{t \in \mathbb{N}}$ are random samples from a uniform distribution in $[-0.8,0.8]$,
with $f_a$ as given in Appendix~\ref{appen:variance_decrease}. The extremal maps $f_a(x) \pm 0.8$ and the bifurcation diagrams are depicted in Figure~\ref{fig:modified_tanh}. When $a \approx 0.175$, the lower extremal map $f_a(x) - 0.8$ develops a tangency with the diagonal (graph of the identity map), corresponding to a fold-bifurcation of fixed points of this lower extremal map, see Figure~\ref{fig:modified_tanh}(b).
This causes the random system to undergo a topological bifurcation where an attractor (the invariant interval bounded from above by the $x$-value where the graph of the upper extremal map intersects the diagonal and from below by the $x$-value where the lower extremal map does so) explodes to a larger region, see Figure~\ref{fig:modified_tanh}(a).

This example is chosen such that the variance of random variables distributed according to the stationary distribution 
does not increase when the parameter $a$ approaches the topological bifurcation point $a_* \approx 0.175$. Indeed, in Figure~\ref{fig:variance}, 
the sample variance of a time series of length $n = 10^6$ is measured for each distinct value of $a$, where the variance is shown to be almost constant (and in fact slightly declining) as $a$ approaches $a_*$ from below. For larger values $a$ arbitrarily close to $a_*$, the attractor explodes and with it also the associated stationary distribution. However, the escape time from the remnant of the attractor is still very long as it is metastable, so one does not typically see the collapse of the attractor from a time series until for much higher values of $a$ (around $0.68$ in this example). We note the similarity of this scenario with the recent suggestion \cite{drijfhout2025shutdown} that the Atlantic Meridional Overturning Circulation (AMOC) is expected to collapse well after passing a tipping point. More slowly changing $a$ would bring the observed tipping point closer to the topological bifurcation point. 
We note that asymptotics of the mean escape time from the remnant of the attractor, close to the bifurcation point, have been obtained in \cite[Chapter 5]{olicon2021critical}. 
This example thus illustrates that variance is not always a reliable early warning signal for tipping.

In this paper, we propose a new early warning signal, estimating the derivative of the extremal map at the fixed point, denoted as $\lambda\in(0,1)$. This is based on recent results on the asymptotic statistical behaviour of orbits near the boundary of an attractor of one-dimensional random diffeomorphisms \cite{olicon2021critical,olicon24tail}.
The derivative $\lambda$ approaches $1$ from below as the system approaches to the topological bifurcation.

To estimate $\lambda$, we fit a $\lambda$-dependent theoretical prediction of the associated asymptotic behaviour of the tail of the stationary density to the tail learned from data, according to \cite{olicon24tail} and Theorem~\ref{THM:main} below. 
In Figure~\ref{fig:variance}, we show that this early warning signal can be successful where the variance fails. In this example, the numerical estimation of $\lambda$ tends to be lower than the theoretical value, close to the bifurcation point. Rather than the precise value of $\lambda$, its increase is an early warning signal for the topological bifurcation. See also Figure~\ref{fig:variance_decrease_ex_result}(b) in Appendix~\ref{APPENDIX:error_discussion}, which shows the indicators obtained from the stationary density approximations using Ulam's method. Further details and discussions of the numerical algorithm are provided in Section~\ref{sec:tail}.

\subsection{Set-up and main results}
Let $f:\mathbb R \to \mathbb R$ be an order-preserving $C^2$-diffeomorphism.
We consider a one-dimensional random dynamical system defined by the recurrence relation:
\begin{equation}\label{eq:rand_diff}
    y_{t+1} = f(y_t) + \xi_t,\; t \in \mathbb{N}
\end{equation}

\noindent where $(\xi_t)_{t \in \mathbb N}$ are i.i.d. random variables 
with a continuously differentiable density $p:[-\varepsilon,\varepsilon]\rightarrow\mathbb{R}_+$ such that $p(\pm\varepsilon)>0$. In a bounded noise model, trajectories of random system \eqref{eq:rand_diff} may be attracted and confined to a 
\textit{minimal invariant interval}\footnote{An interval $M$ is invariant for \eqref{eq:rand_diff}, if it is compact and $M=\overline{B_{\varepsilon}(f(M))}:=\{ f(x)+\xi : x\in M, \ \xi\in [-\varepsilon,\varepsilon] \}$. A compact interval is a minimal invariant set, if it is invariant for \eqref{eq:rand_diff}, and does not contain a proper invariant subset.} $M=[x_-,x_+]$, from which escape is not possible. Such an interval constitutes the support of a stationary distribution $\mu$. 
We refer the reader to \cite{Kifer86} for a general discussion of stationary distributions of random dynamical systems in discrete-time, and \cite{araujo2000attractors, olicon2021critical,zmarrou2007bifurcations} for our specific setting. Moreover, $\mu$ is absolutely continuous with respect to the Lebesgue measure $\textup{Leb}$, and we denote $\phi(x)$ as its density.

A topological bifurcation of a minimal invariant interval is characterised by discontinuous changes of its boundary points, when either the magnitude $\varepsilon$ or the parameters in $f$ are altered \cite{lamb2015topological}. We denote by $f_{\pm}(x):=f(x)\pm \varepsilon$ the extremal maps, representing one iterate of the random system (\ref{eq:rand_diff}) under extremal noise. The boundary points $x_{\pm}$ of a minimal invariant interval are fixed points of the extremal maps, i.e. $f_{\pm}(x_\pm) = x_\pm$, since $f$ is an order-preserving diffeomorphism.
Hence, for a topological bifurcation of a minimal invariant interval to occur, a fold bifurcation of the extremal map must take place at the corresponding boundary point~\cite{kuehn2018early}.
Without loss of generality, we assume this occurs at the left boundary $x_-$. At this bifurcation point, the derivative of the extremal map at the boundary equals one.
In this case, the parameter $\lambda=f_-'(x_-)$ serves as an indicator of proximity to a topological bifurcation: the system approaches the bifurcation as $\lambda$ tends to 1.

Our main objective is to infer $\lambda$ from a time series produced by \eqref{eq:rand_diff}. Its proximity to $1$ or its increasing values serve as a warning signal of the topological bifurcation. 
Assuming $x_-$ is a hyperbolic fixed point, i.e. $\lambda < 1$, our method utilises an asymptotic expansion of the tail of the stationary density 
as $x\rightarrow x_-$, given by
\begin{equation}
\label{eq:higher_order}
\log\phi(x)=\frac{1}{2\log\lambda}\log^2(x-x_-)-\frac{2}{2\log\lambda} \log(x-x_-)\cdot \log\log\left( \frac{1}{x-x_-}\right) + \mathcal{O}(\log(x-x_-)),
\end{equation}
as derived in Theorem~\ref{THM:main} below. The leading order term of the asymptotic expansion was first obtained in \cite{olicon24tail}, and is given by

\begin{equation}
\label{eq:first_order}
    \log\phi(x)= \frac{1}{2\log\lambda}\log^2(x-x_-) + o\left(\log^2(x-x_-)\right).
\end{equation}
Therefore, given the boundary point $x_-$ and $\phi(x)$, for several values of $x$ close to $x_-$, one can aim to infer $\lambda$.

The paper is structured as follows: In Section~\ref{sec:prelim}, we derive the
asymptotic expansion near the tail of the stationary density, as given in \eqref{eq:higher_order}. Section~\ref{sec:tail} introduces fitting methods motivated by the asymptotic expansions for estimating $\lambda$ from time series of the random difference equation (\ref{eq:rand_diff}). In Section~\ref{sec:results}, we apply these methods to data generated from both linear and nonlinear maps with additive bounded noise. The results are presented in two settings: i) when the true boundary $x_-$ is assumed to be known (Section~\ref{sec:known}), and ii) a more realistic scenario where the boundary $x_-$ is unknown and must be estimated from the time series (Section~\ref{sec:unknown}).

\section{Asymptotic expansion near tail of stationary density}\label{sec:prelim}

In this section, we derive the asymptotic expansion of the stationary density near the boundary of its support. The construction is based on the transfer operator associated with the random system (\ref{eq:rand_diff}). For further background on transfer operators and their role in random dynamical systems, we refer the reader to \cite{zmarrou2007bifurcations,olicon2021critical}.

Consider a stationary distribution $\mu$ of \eqref{eq:rand_diff} supported on a compact interval $M=[x_-,x_+]$, and let $\phi$ be its corresponding density, so that $\phi$ is a fixed point of the so-called \textit{transfer operator} $L:L^1(M) \to L^1(M)$, where $L^1(M)=\{ g:M\rightarrow \mathbf{R}_+ : \int_M\vert g(x)\vert dx<\infty\}$. This operator defines the evolution of densities induced by the dynamics of \eqref{eq:rand_diff}. 
Let $k:M\times M \to \mathbb{R}_+$ be the density of the transition probability of \eqref{eq:rand_diff} in the interval $M$, then the transfer operator is given by
\cite{zmarrou2007bifurcations}:
\[
    Lg(x):=\int_M{k(y,x)g(y)dy}.
\]
Recall that $p:[-\varepsilon,\varepsilon] \to \mathbb{R}_+$ denotes the probability density of the random noise. The density $k$ reads $k(y,x) = \tfrac{1}{\varepsilon}p(x-f(y))$. For a fixed $x_0$ sufficiently close to the boundary point $x_-$ and any $x\in[x_-,x_0]$, the stationary density satisfies \cite[Chapter 2.4]{olicon2021critical}
\begin{equation}
    \label{eq:phi_near_0}
    \phi(x)=\frac{1}{\varepsilon}\int_{x_-}^{f_-^{-1}(x)}{p(x-f(y))\phi(y)dy}.
\end{equation}
Since the stationary density appears again in the integrand, one can iterate this process $n$ times, as long as $f_-^{-n+1}(x)\in (x_-,x_0)$. Equivalently, we can repeat this process $n_{x_0}^x- 1$ times, where $n_{x_0}^x$ is the first time $f^{n_{x_0}^x}_-(x_0)$ crosses $x$, i.e. 
\begin{equation}
    \label{eq:hitting_time_det}
    n_{x_0}^x:=\min\{n\geq 0 : f_-^n(x_0)<x\}.
\end{equation}
By means of this iterative process, one obtains the following bounds of $\phi$.
\begin{lemma}
\label{LEMMA:main}
    Let $M=[x_-,x_+]$ be the support of a stationary distribution of \eqref{eq:rand_diff}, and assume that $f'_-$ is monotone near $x_-$, $\lambda = f'_-(x_-)<1$ and $p(-\varepsilon)>0$.
    Then, there exists $x_0$ sufficiently close to $x_-$ and constants $C_1,C_2>0$ such that 
    \begin{itemize}
        \item the stationary density $\phi:M\rightarrow\mathbb{R}_+$ is increasing in $[x_-,x_0]$, and
    \item for each $x\in(0,x_0)$, the stationary density satisfies
    \begin{equation}
        \label{eq:main_ineqs}
        \begin{split}
        \frac{C_1^n}{(n-1)!}\left[ \prod_{m=2}^n \prod_{j=2}^m \frac{1}{f_-'(f_-^{-j}(x))} \right]\int_{x_-}^{f_-^{-1}(x)}(f_-^{-1}(x)-y)^{n-1}\phi\left( f_-^{-n+1}(y)\right)dy \leq \\
        \phi(x) \leq 
        \frac{C_2^n}{n!}\left( 
 \frac{1}{\lambda}\right)^{\frac{n(n-1)}{2}} \vert f_-^{-1}(x)\vert^n \phi(x_0),
        \end{split}
    \end{equation}
where $n\equiv n_{x_0}^x$.
    \end{itemize}
\end{lemma}
Lemma~\ref{LEMMA:main}   is a slight improvement of \cite[Lemma 3.1(a)]{olicon24tail}. 
We refer to the proof in Appendix~\ref{APPENDIX:proof}.
It establishes the asymptotic expansion of the stationary density in (\ref{eq:higher_order}), which we restate here formally:
\begin{theorem}
    \label{THM:main}
    Let $M=[x_-,x_+]$ be the support of a stationary distribution of the random difference equation with additive bounded noise in \eqref{eq:rand_diff}, and assume that $f_-'$ is monotone near $x_-$, $\lambda = f'_-(x_-)<1$ and the noise density $p$ is such that $p(-\varepsilon) > 0$. Then, the asymptotic expansion of the stationary density $\phi$ satisfies (\ref{eq:higher_order}) as $x\rightarrow x^+_-$.
    
\end{theorem}
The proof of Theorem~\ref{THM:main} is deferred to Appendix~\ref{APPENDIX:proof}.

\section{Tail fitting and bifurcation parameter estimation}\label{sec:tail}

We propose a numerical algorithm to approximate $\lambda=f_-'(x_-)$ from time series of the random system (\ref{eq:rand_diff}), by fitting the data to the asymptotic expansion of the stationary density \eqref{eq:higher_order}.

\subsection*{Least squares curve fitting}

Consider a time series dataset $(y_t)_{1 \leq t \leq n}$ obtained from the random difference equation (\ref{eq:rand_diff}).
We first approximate the stationary density $\phi$ with a piecewise constant function $\hat{\phi}$ using the normalised histogram of the time series with equally sized bins. Let $b\in\mathbb{N}$ be a hyperparameter denoting the number of bins, with edges $\{\min_{1\leq t \leq n}(y_t) = z_0<z_1<\cdots<z_b = \max_{1\leq t \leq n}(y_t)\}$, and $\Delta z:=\frac{z_b-z_0}{b}$ 
the bin width, where $z_{i+1}=z_i+\Delta z$ for $i\leq b-1$. We consider the midpoints\footnote{One may also take $x_i$ to be either the left or right edge of the bin, i.e. $x_i = z_i$ or $x_i = z_{i+1}$. The choice can affect the estimation of $\lambda$, particularly when the bin width $\Delta z$ is relatively large.} given by $x_i=\frac{z_{i-1}+z_{i}}{2}$ for $1\leq i\leq b$ and let $(h_i)_{1\leq i \leq b}$ denote the heights of the histogram.

The asymptotic expansion of the stationary density $\phi(x)$ in (\ref{eq:higher_order}) depends on both the bifurcation parameter $\lambda$ and the boundary $x_-$. When the boundaries $x_-,x_+$ are unknown, we adopt boundary estimates $\hat{x}_-,\hat{x}_+$, such that
\begin{equation}\label{eq:boundary_estimates}
    \hat{x}_- = x_1 - \Delta z, \; \hat{x}_+ = x_b + \Delta z. 
\end{equation}
More refined methods for boundary estimation, which are crucial for determining the bifurcation parameter $\lambda$, are beyond the scope of this paper.

Formally, the histogram approximation of the stationary density can be expressed as a step function $\hat{\phi}$, defined as
\begin{equation}
    \label{eq:approx_density}
    \begin{split}
    \hat{\phi}: [\hat{x}_-,\hat{x}_+]\rightarrow\mathbb{R}_+, & \quad \hat{\phi}(x)= \sum_{i=1}^{b} h_i\mathds{1}_{[z_{i-1},z_i]}(x), \\
    h_i  := & \frac{1}{\Delta z\cdot n}\sum_{j=1}^{n}\mathds{1}_{[z_{i-1},z_i]}(y_j).
    \end{split}
\end{equation}

Let $q\in (0,1)$ be the quantile level (considered in the following to be a hyperparameter), and let $b_l$ be the q-quantile of $\hat{\phi}$, i.e
\begin{equation}
    \label{eq:b_l}
    b_l:= \max \left\{ 1\leq j\leq b : \sum_{i=1}^j h_i < q \right\}.
\end{equation}
The left tail of the stationary density is thus approximated by $\hat{\phi}(x)$ for $\hat{x}_- \leq x\leq x_{b_l}$.

The asymptotic expansion of $\phi(x)$ as $x \to x_-$, in (\ref{eq:first_order}) with its leading order term and (\ref{eq:higher_order}) with an additional higher order term, shows that the bifurcation parameter $\lambda$ can, in principle, be recovered from sample points of $x$ and $\phi(x)$ near the boundary $x_-$. 

In particular, the leading order expansion in (\ref{eq:first_order}) indicates that $\log\phi(x_i)\approx\log\hat{\phi}(x_i) = \log(h_i)$ admits quadratic behaviour
as a function of $\log(x_i-x_-)$.
Our first approach for estimating $\lambda$ is to fit a quadratic curve, consisting only of the linear and quadratic terms\footnote{One can include a constant term $a_0$, or set $a_1 = 0$ in \eqref{eq:quad_method}, but these can cause the $\hat{\lambda}$ approximation to be unstable. Hence, we restrict ourselves to $a_0=0$.}, by minimising the least squares error:
\begin{equation}\label{eq:quad_method}
    \min_{a_1 \in \mathbb{R}, a_2<0}\left(\sum_{i=1}^{b_l}(a_2l_i^2 + a_1l_i- \log(h_i))^2\right),\; l_i = \log(x_i-\hat{x}_-)
\end{equation}
where $\lambda$ can be estimated by 
\begin{equation}
\label{eq:estimator}
    \hat{\lambda} = \exp(1/(2a_2)).
\end{equation}
There are other fitting methods, e.g. least absolute residuals and bisquare, but we will focus on least squares fitting in this paper. From the higher order asymptotic relationship in (\ref{eq:higher_order}), an analogous optimisation problem can be formulated as
\begin{equation}\label{eq:quad_method_strict}
    \min_{a_1 \in \mathbb{R}, a_2<0}\left(\sum_{i=1}^{b_l}(a_2(l_i^2-2l_i\log(-l_i)) + a_1l_i- \log(h_i))^2\right),\; l_i = \log(x_i-\hat{x}_-)
\end{equation}
where again the estimation of $\lambda$ is obtained as in \eqref{eq:estimator}. We refer to (\ref{eq:quad_method}) as the \textit{leading order method} and (\ref{eq:quad_method_strict}) as the \textit{higher order method}.

We summarise the methodology presented above in the following algorithm.

\vspace{.1em}
\begin{alg}\label{alg:tailfit}
\text{ }
\begin{enumerate}
    \item Record $n$ iterations $(y_t)_{1 \leq t \leq n}$ of the random difference equation (\ref{eq:rand_diff}) with some chosen initial condition $y_0$ and construct the normalised histogram $\hat{\phi}$ \eqref{eq:approx_density}, with a chosen number of bins $b$ and boundary estimate $\hat{x}_{-}$.
    \item Choose $0<q<1$, and extract the information of the tail by obtaining $b_l$ as given in \eqref{eq:b_l}. 
    \item The closest values of $x_i$ to $\hat{x}_-$ typically take extremely small values $h_i$ and do not represent the tail distribution well. Hence, we only consider the histogram data where $h_j > h_{\text{min}}$ for a small threshold hyperparameter\footnote{In this paper, we always set $h_{\text{min}} = \max_{1 \leq i \leq b}(h_i)/100$.} $h_{\text{min}} > 0$.
    \item Solve the optimisation problem either (\ref{eq:quad_method}) or (\ref{eq:quad_method_strict}) with the information at the tail obtained in the previous step and obtain the estimation $\hat{\lambda}$ from (\ref{eq:estimator}).
\end{enumerate}
\end{alg}

Algorithm~\ref{alg:tailfit} is implemented in MATLAB, and our results can be reproduced using the source codes provided in \cite{Tey2024}. Figure~\ref{fig:log_hist}(a) illustrates the fitting of the histogram tail in the logarithm scale described by the leading-order method (\ref{eq:quad_method}), for the example where $f(x)=\lambda x$, on which we elaborate in the following.

\subsection*{Linear example}

Our first example consists of the linear map $f(x) = \lambda x$ with $\lambda \in (0,1)$. The support of the unique stationary distribution of the random system (\ref{eq:rand_diff}) is the interval $[-\varepsilon/(1-\lambda),\varepsilon/(1-\lambda)]$. For simplicity, we scale the maximum absolute value of $\xi_t$ to $(1-\lambda) \varepsilon$, ensuring that the interval $[-\varepsilon,\varepsilon]$ is the support for all $\lambda \in (0,1)$, see Figure~\ref{fig:log_hist}(b). The system (\ref{eq:rand_diff}) in this case reads as
\begin{equation}\label{eq:linear_noise}
    y_{t+1} = \lambda y_t + \xi_t,
\end{equation}

\begin{figure}[h]
    \centering
    \begin{overpic}[width=.49\textwidth]{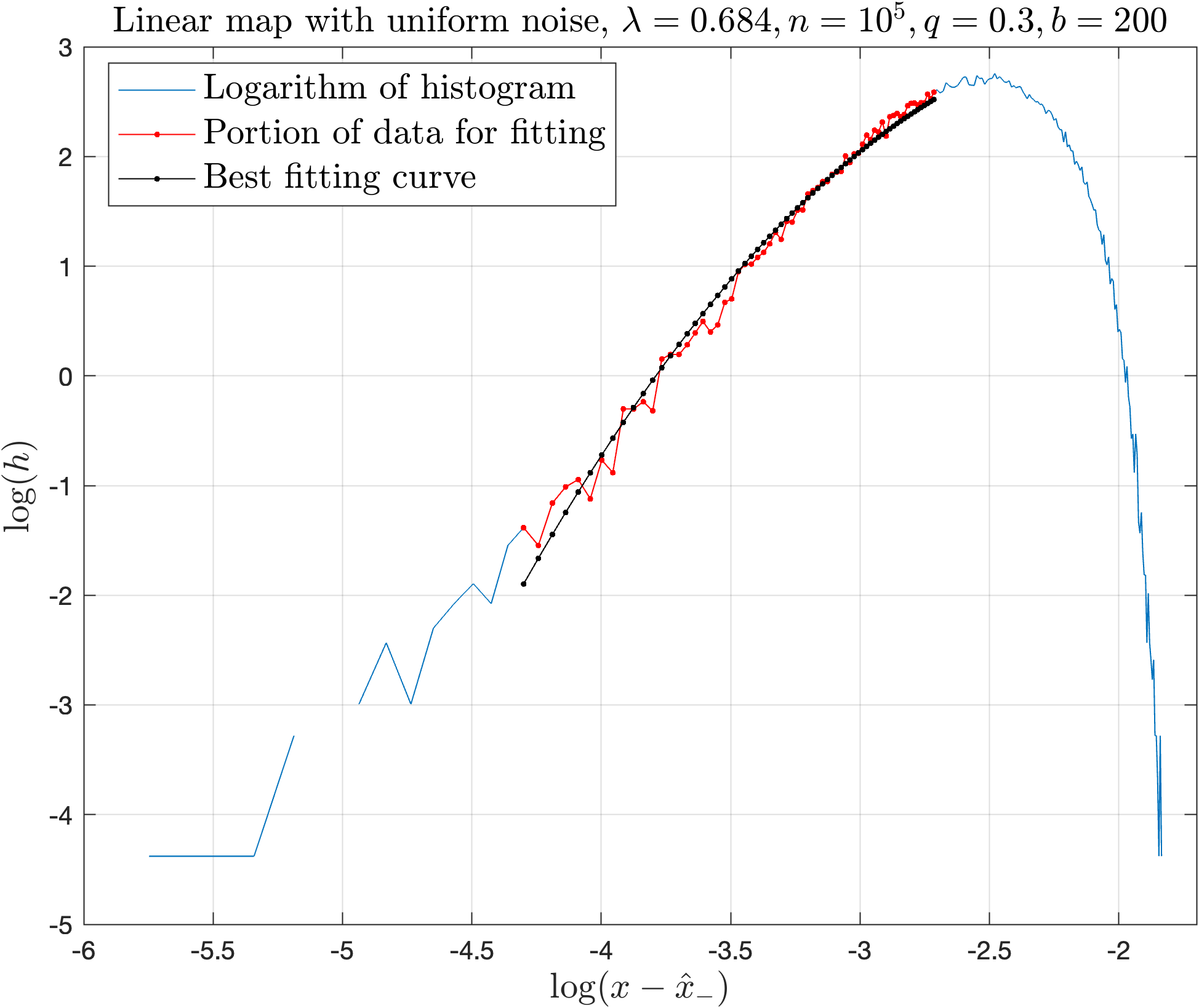}
    \put(0,0){
    (a)
    }
    \end{overpic}
    \begin{overpic}[width=.49\textwidth]{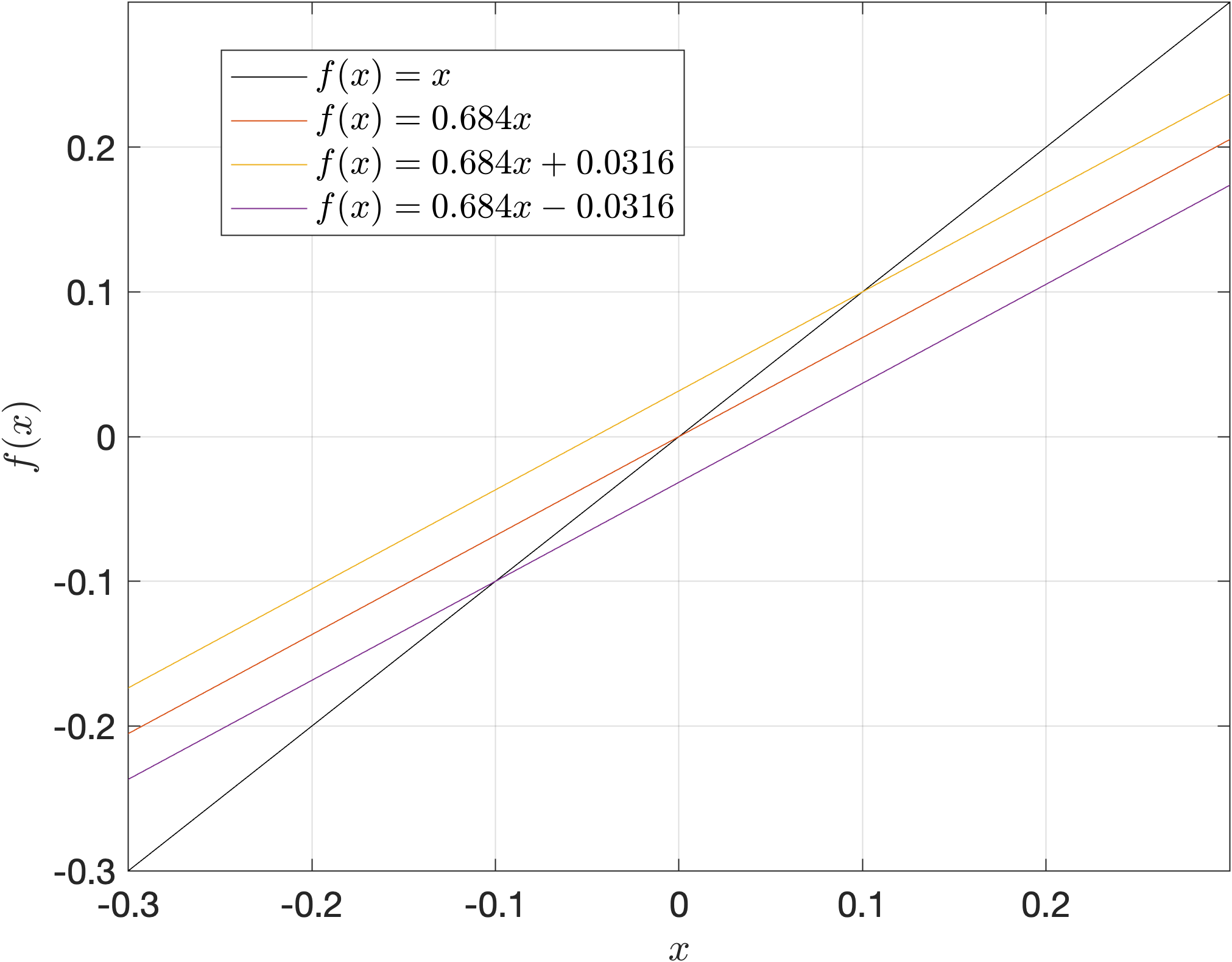}
    \put(0,0){
    (b)
    }
    \end{overpic}
    \captionsetup{width=\linewidth}\caption{Illustration of the optimisation scheme (\ref{eq:quad_method}) following Algorithm~\ref{alg:tailfit} in Figure (a), for the linear map $f(x) = \lambda x$ with uniform noise in (\ref{eq:linear_noise}), and the graphs of the extremal maps $f_{\pm}(x)=\lambda x \pm (1-\lambda)\varepsilon$ for $\varepsilon = 0.1$ is shown in Figure (b). Here we take $n = 10^5$ iterations with $\lambda = 0.684$ and construct the normalised histogram using $b = 200$ bins, shown in blue in Figure (a) on a logarithmic scale. The least squares quadratic curve for $q = 0.3$ quantile of the histogram data is shown in black according to (\ref{eq:quad_method}), giving an estimate $\hat{\lambda} = 0.419$, using boundary estimate $\hat{x}_- = x_1 - \Delta z$. From the figure, it becomes apparent the necessity of discarding the nearest data points to $\hat{x}_-$.}
    \label{fig:log_hist}
\end{figure}

\noindent with $\xi_t \sim U(-(1-\lambda)\varepsilon,(1-\lambda)\varepsilon)$ and we choose $\varepsilon = 0.1$. The noise $\xi_t$ has the probability density function
\begin{displaymath}
    p(z) = \begin{cases}
        \frac{1}{2(1-\lambda)\varepsilon}, \; \text{if }-(1-\lambda)\varepsilon\leq z\leq (1-\lambda)\varepsilon,\\
        0, \qquad \quad  \text{otherwise}.
    \end{cases}
\end{displaymath}

Next, we explore a specific nonlinear example which exhibits a topological bifurcation of a minimal invariant set as its parameter varies.

\subsection*{Nonlinear example}

Consider the random system (\ref{eq:rand_diff}) with invertible nonlinear map $f_a: \mathbb{R} \to \mathbb{R}$
\begin{equation}\label{eq:non-linear_eq}
\begin{split}
    y_{t+1}  = f_a(y_t) + \xi_t, \\
    f_a(y_t)  = 3\tanh{(y_t/2)} - a & \equiv \frac{3(1-\exp(-y_t))}{1+\exp(-y_t)} - a,
\end{split}
\end{equation}
\noindent for some parameter $a \in \mathbb{R}$. The map $f_a$ is plotted in Figure~\ref{fig:deter:tanh}(a) for various values of the parameter $a$. The fixed points $x^*$ of the deterministic map $f_a$ undergo fold bifurcations at two critical values $a_1^*<a_2^*$, forming an S-shaped curve in the bifurcation diagram, see the dotted graph in Figure \ref{fig:deter:tanh}(b).

\begin{figure}[!b]
    \begin{overpic}[width=.49\textwidth]{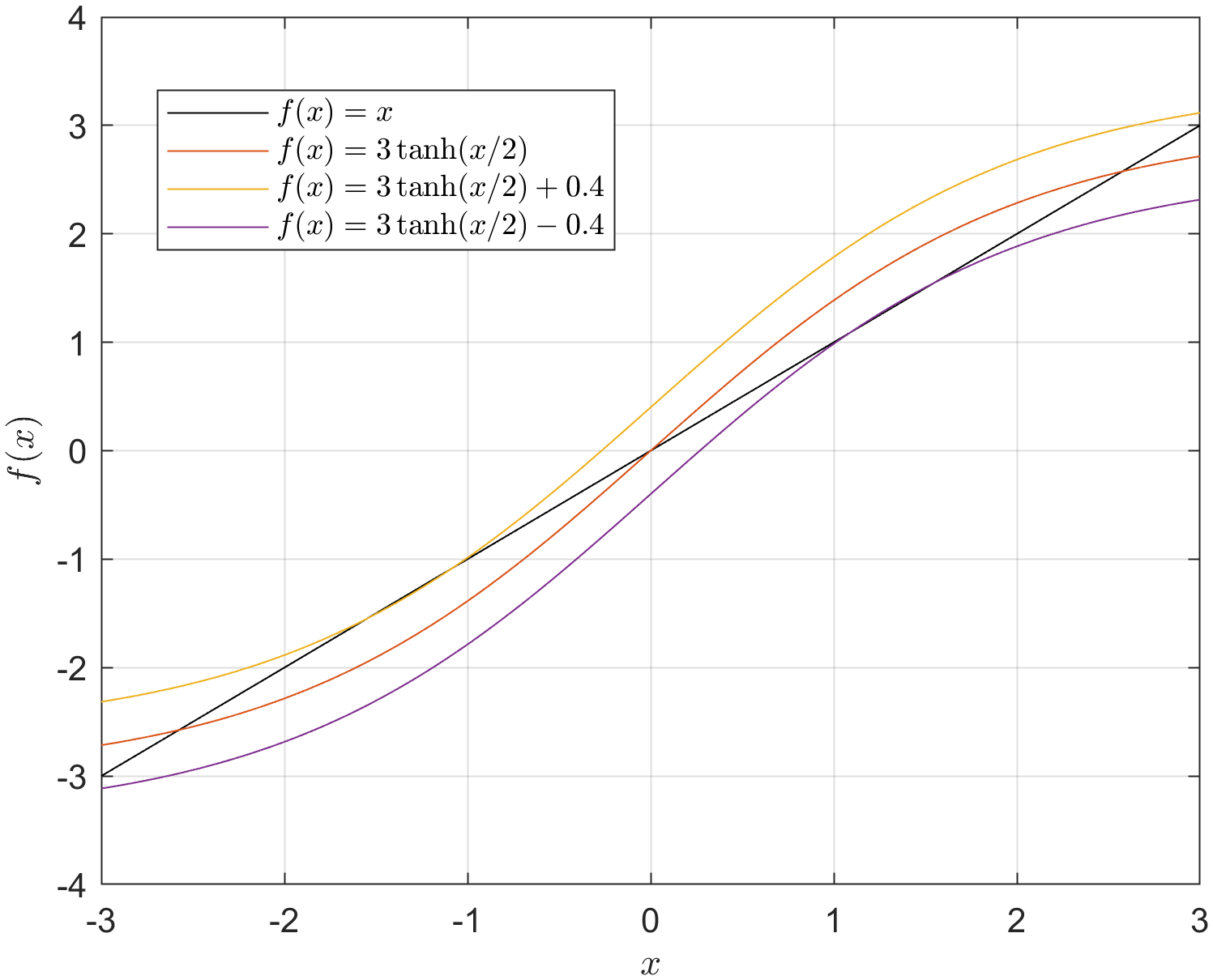}
    \put(0,0){
    (a)
    }
    \end{overpic}
    \begin{overpic}[width=.49\textwidth]{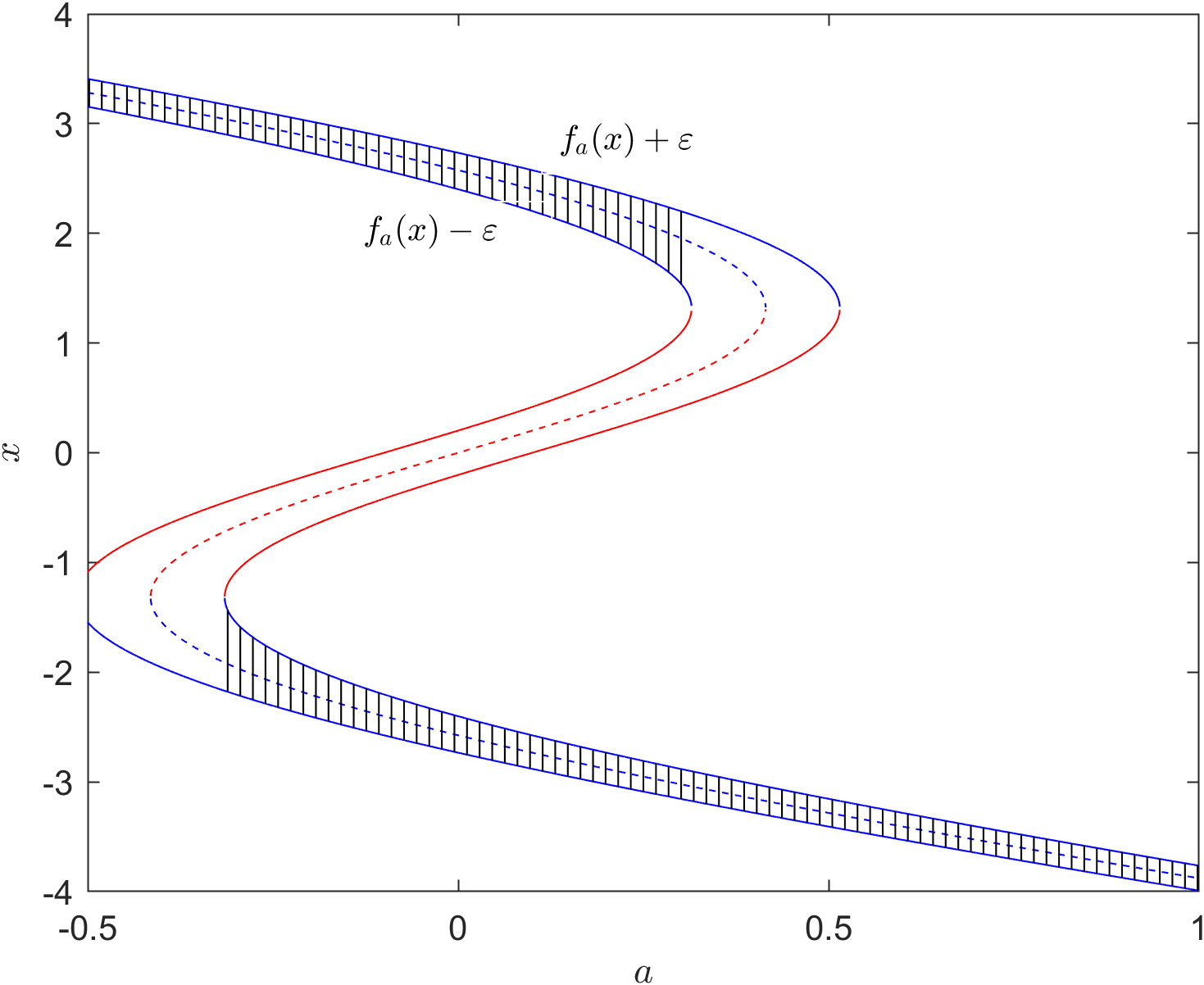}
    \put(0,0){
    (b)
    }
    \end{overpic}
\captionsetup{width=\linewidth}\caption{Illustration of the deterministic map $f_a(x) = 3 \tanh(x/2) - a$, the extremal maps $f_a(x) \pm \varepsilon$ for $\varepsilon = 0.1$ and their bifurcation diagrams. In (a), the graphs of the map $f$ for $a\in \{-0.4, 0, 0.4\}$ are plotted alongside the diagonal identity line. In (b), the bifurcation diagram for $f_a(x)$ across parameters $-0.5\leq a\leq 1$ is shown as the dotted line while those of the extremal maps are plotted as solid lines. The blue lines signify stable fixed points while red lines represent unstable fixed points. The black vertical lines are examples of minimal invariant intervals of (\ref{eq:non-linear_eq}) for different $a$ values.}
\label{fig:deter:tanh}
\end{figure}

The bifurcation points $(x^*,a^*)$ can be determined by solving the system of equations
\begin{equation}\label{eq:analytic_bif}
\begin{split}
    f_{a^*}(x^*)=x^*,\\
f_{a^*}'(x^*) = 1.
\end{split}
\end{equation}
In the deterministic setting, i.e. $\varepsilon=0$, this map mirrors a system with a critical transition from one stable state to another as the parameter of the system changes, akin to a fold catastrophe model~\cite{scheffer2009early}. This transition is evident in the bifurcation diagram where the stability of the upper fixed point weakens as the parameter $a$ increases and undergoes a fold bifurcation at a critical point, after which the orbits of the system converge to the fixed point on the lower branch. Following this transition, it becomes challenging for the system to return to the upper branch without significantly reducing parameter $a$.

The bifurcation diagrams for the extremal maps $f_a(x) + \varepsilon$ and $f_a(x) - \varepsilon$ are shifted right and left by $\varepsilon$ respectively compared to that one for the original map $f_a$, as plotted as solid lines in Figure~\ref{fig:deter:tanh}(b). Any minimal invariant interval $[x_-,x_+]$ of \eqref{eq:non-linear_eq} is such that $f(x_\pm)\pm \varepsilon=x_{\pm}$ \cite{olicon2021critical}, see Figure \ref{fig:deter:tanh}(b) where they are plotted as black lines.

The minimal invariant interval (attractor) in the positive region undergoes a topological bifurcation when $a$ increases near $a \approx 0.31$, where it explodes and disappears. Beyond this point, random trajectories are attracted to the other attractor in the negative region. This transition is locally non-invertible. Namely when $a$ is subsequently lowered, the negative attractor branch is followed until the attractor disappears and the trajectory jumps back to the erstwhile followed attractor branch. This behaviour is called \textit{hysteresis} and contrasts with the example considered in the introduction, where the attractor expands to a larger invariant interval and trajectories display flickering behaviour between two metastable states within it (see Figure~\ref{fig:modified_tanh}(a)).

In the numerical experiments, we consider uniformly distributed noise $\xi_t$ with probability density function
\begin{displaymath}
    p(z) = \begin{cases}
        \frac{1}{2\varepsilon}, \; \text{if }-\varepsilon\leq z\leq \varepsilon,\\
        0, \;\;\;  \text{otherwise}.
    \end{cases}
\end{displaymath}
We also consider noise with a truncated normal distribution, where the probability density function of the noise $\xi_t$ satisfies
\begin{equation}\label{eq:truncated_density}
    p(z;\mu,\sigma,-\varepsilon,\varepsilon) = \begin{cases}
        \frac{\psi(\frac{z-\mu}{\sigma})}{\sigma(\Psi(\frac{\varepsilon-\mu}{\sigma})-\Psi(\frac{-\varepsilon-\mu}{\sigma}))} \;\;, \text{if }-\varepsilon\leq z\leq \varepsilon,\\
        0 \qquad \qquad \qquad \quad \;\;\;, \text{otherwise},
    \end{cases}
\end{equation}

\noindent where we fix $\mu = 0$ and $\sigma = \varepsilon/2$. The functions $\psi(z) = e^{-z^2/2}/\sqrt{2\pi}$ and $\Psi(z) = (1+\erf(z/\sqrt{2}))/2$ are the probability density function and cumulative distribution function of the standard normal distribution respectively, with error function denoted by $\erf(z) = \frac{2}{\sqrt{\pi}}\int_0^z \exp(-t^2)dt$.

In the following, we investigate the efficacy of both the leading-order method (\ref{eq:quad_method}) and the higher-order method (\ref{eq:quad_method_strict}) as implemented through Algorithm~\ref{alg:tailfit}. We explore both cases where the deterministic map $f$ is either a linear map as in (\ref{eq:linear_noise}), or a nonlinear map as given in (\ref{eq:non-linear_eq}). In both cases, we test our method assuming that the noise is uniformly distributed. In addition, for the nonlinear example, we compare the results using a truncated normal distribution for the noise.

\section{Results}\label{sec:results}
Recall that the optimisation schemes in both the leading-order method (\ref{eq:quad_method}) and the higher-order method (\ref{eq:quad_method_strict}) depend on the boundary point $x_-$. 
When implementing Algorithm~\ref{alg:tailfit}, we consider two scenarios: i) the true boundary is assumed to be known, i.e. $\hat{x}_-=x_-$; ii) the boundary is approximated by $\hat{x}_- = x_1 - \Delta z$, cf. \eqref{eq:boundary_estimates}, which corresponds to the midpoint of an empty bin at the left end of the constructed histogram.

\subsection{Prior knowledge of the boundary point}\label{sec:known}

In this subsection, we assume the true boundary is known, i.e. $\hat{x}_- =x_-$. Numerically, the boundary point $x_-$ can be obtained by solving the fixed point equation $f_-(y)=y$,  using Newton's method, for instance. For a broad range of system parameters, we implement Algorithm~\ref{alg:tailfit} by taking $n=10^5$ iterations, and fixing the hyperparameters $b=200$ and $q=0.3$. We refer the reader to Appendix~\ref{sec:optimum} for a discussion of the choice of $b$ and $q$, and Appendix~\ref{sec:low_number} for results in smaller sample sizes $n$, showing that satisfactory early warning signals can already be obtained for $n=10^3$. The initial condition $y_0$ is chosen inside the support of the stationary distribution. This procedure is repeated 100 times, with different noise realisations. Let $\hat{\lambda}_{x_-}$ denote the estimation of $\lambda$ when the true boundary $x_-$ is used.

\begin{figure}[!b]
\begin{overpic}[width=.49\textwidth]{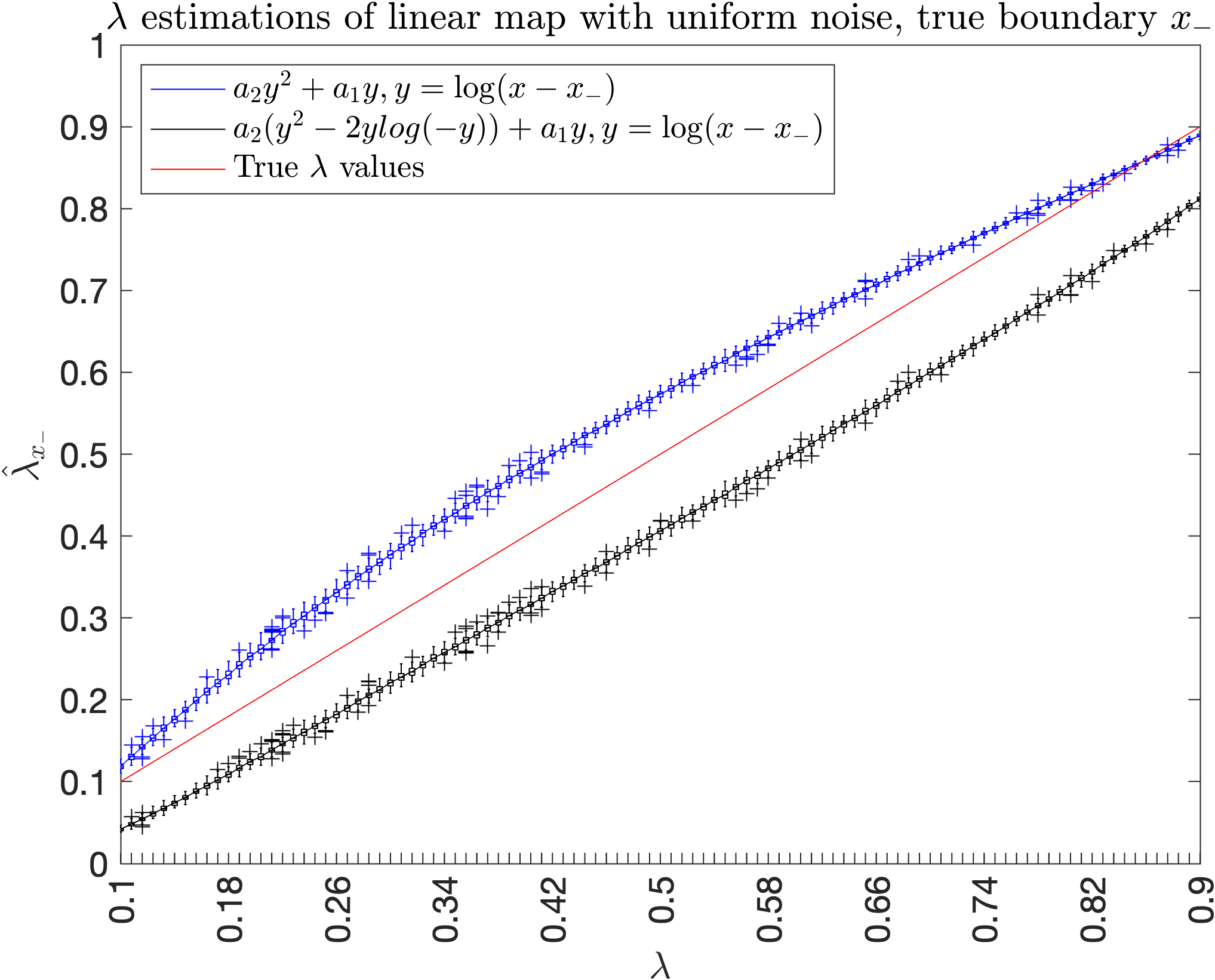}
    \put(0,0){
    (a)
    }
    \end{overpic}
    \begin{overpic}[width=.49\textwidth]{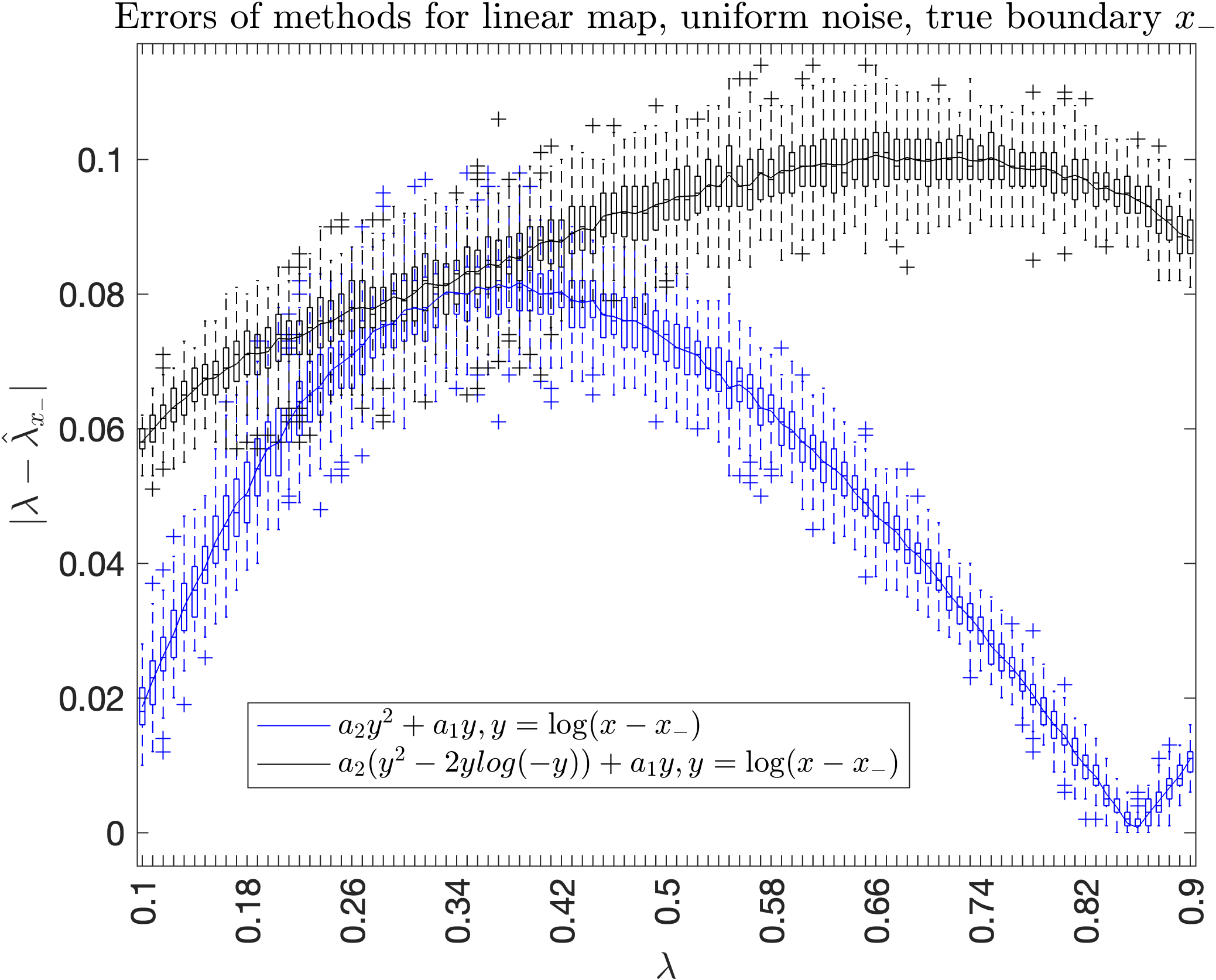}
    \put(0,0){
    (b)
    }
    \end{overpic}
    \captionsetup{width=\linewidth}\caption{Numerical approximations of $\lambda$ from the time series of the linear maps with uniform noise (\ref{eq:linear_noise}) for $0.1 \leq \lambda \leq 0.9$, using methods (\ref{eq:quad_method}) in blue and (\ref{eq:quad_method_strict}) in black, respectively, assuming the left boundary $x_-$ is known. The approximations are obtained following Algorithm~\ref{alg:tailfit} with $n= 10^5, b=200$, and $q = 0.3$. For each parameter value $\lambda$, the estimator $\hat{\lambda}_{x_-}$ (cf.~\eqref{eq:estimator}) is computed $100$ times, using independent time series generated with different noise realisations. Figure~(a) shows the box plots of the estimators $\hat{\lambda}_{x_-}$, while the corresponding errors are plotted in (b). The boxes represent the interquartile range (25$^{\text{th}}$ to 75$^{\text{th}}$ percentile), while the whiskers extend to all non-outlier data points, and the crosses mark the outliers. Sample means are indicated by solid lines.}
    \label{fig:lambda_actual_linear}
\end{figure}

\subsubsection*{Linear map with uniform noise}\label{sec:linear_actual}

We consider the linear map with uniform noise (\ref{eq:linear_noise}).
The resulting estimates $\hat{\lambda}_{x_-}$ are shown as box plots in Figure~\ref{fig:lambda_actual_linear}(a), while the corresponding errors $|\lambda - \hat{\lambda}_{x_-}|$ are shown in Figure~\ref{fig:lambda_actual_linear}(b). All errors remain small, in the order of $\mathcal{O}(10^{-1})$, for all $101$ parameter values $a$ uniformly sampled from the interval $[0.1,0.9]$. Overall, the leading-order method (\ref{eq:quad_method}) outperforms the higher-order method (\ref{eq:quad_method_strict}).

\begin{figure}[!b]
\begin{overpic}[width=.49\textwidth]{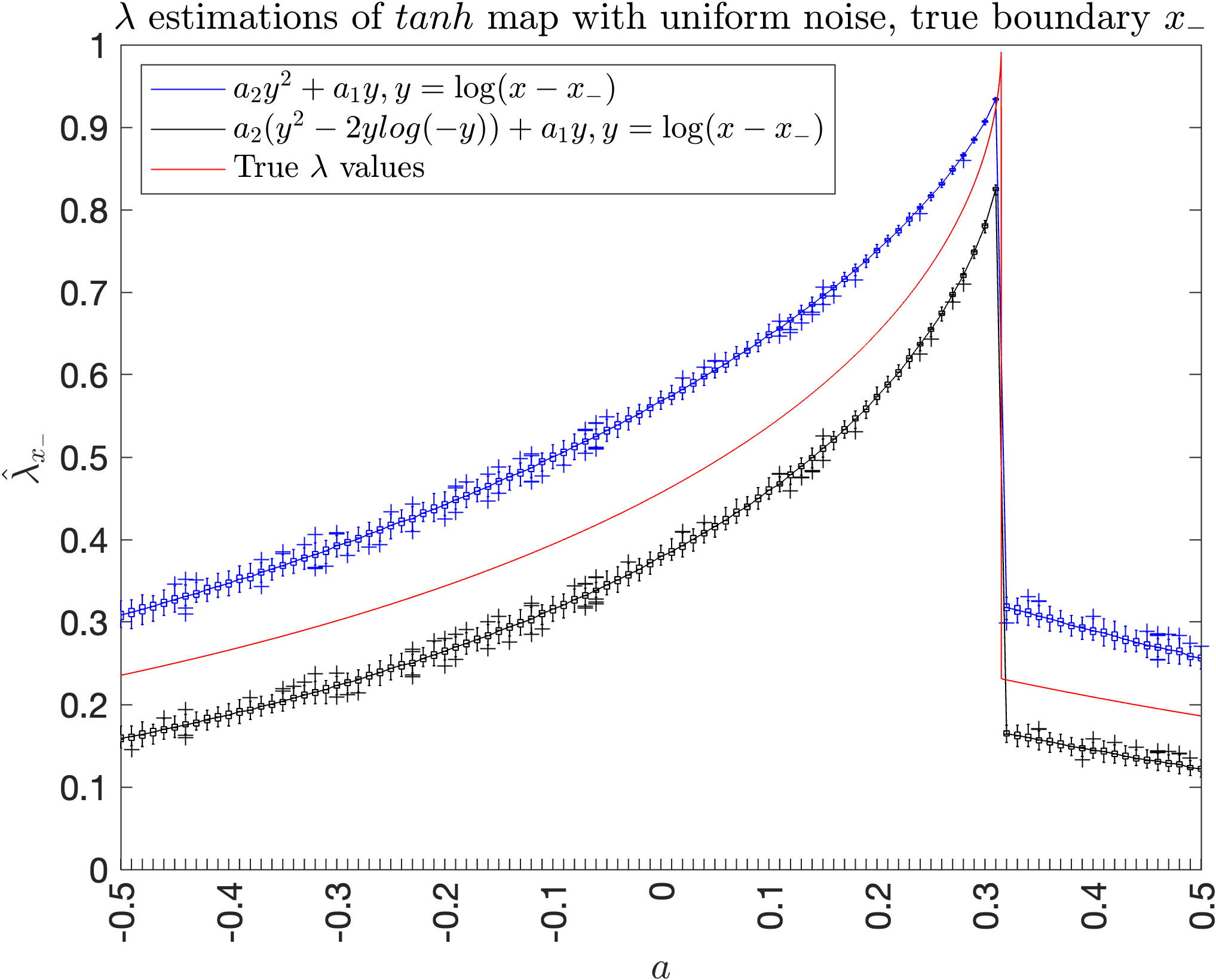}
    \put(0,0){
    (a)
    }
    \end{overpic}
    \begin{overpic}[width=.49\textwidth]{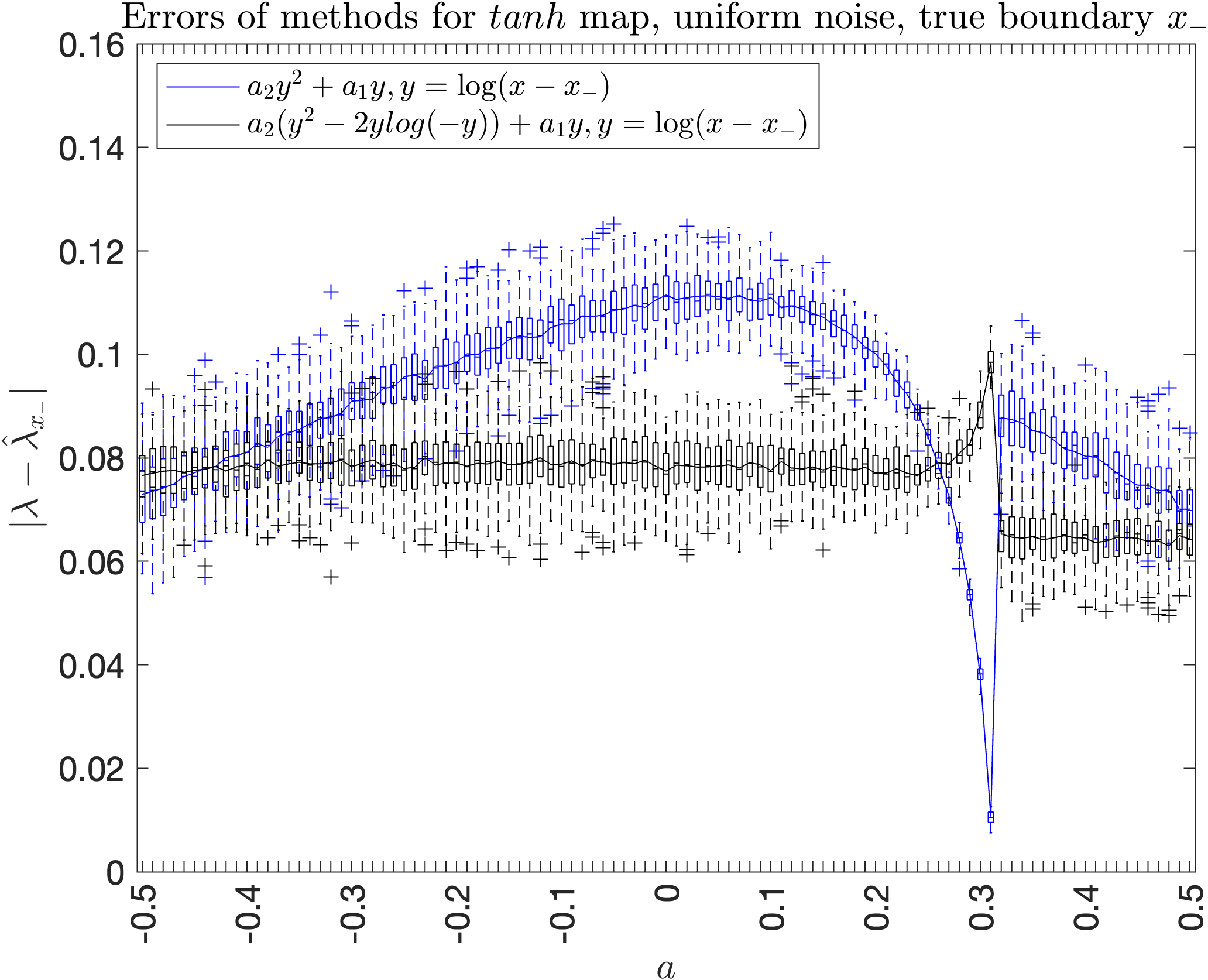}
    \put(0,0){
    (b)
    }
    \end{overpic}
    
    \captionsetup{width=\linewidth}\caption{Numerical approximation of $\lambda$ for the nonlinear map with uniform noise (\ref{eq:non-linear_eq}), with $-0.5 \leq a \leq 0.5$. We use both methods (\ref{eq:quad_method}) and (\ref{eq:quad_method_strict}), shown in blue and black, respectively, assuming the boundary $x_-$ is known. In (a), the estimator $\hat{\lambda}_{x_-}$ is plotted against the true value $\lambda$, while (b) shows the corresponding errors. The results are obtained using Algorithm~\ref{alg:tailfit} with $n= 10^5, b=200, q = 0.3$ and the initial value $y_0$ is chosen from the support of the stationary distribution.}
    \label{fig:lambda_actual_non_linear}
\end{figure}

\subsubsection*{Nonlinear map with bounded noise}

Similarly, we generate trajectories of (\ref{eq:non-linear_eq}) with uniform noise for $101$ parameter values $a$ uniformly sampled from $[-0.5,0.5]$,
and approximate $\lambda$ using both methods (\ref{eq:quad_method}) and (\ref{eq:quad_method_strict}). The numerical results
are shown in Figure~\ref{fig:lambda_actual_non_linear}. All errors remain small (less than $0.13$). We observe that the approximations using the higher-order asymptotics (\ref{eq:quad_method_strict}) generally outperform those by the leading-order method (\ref{eq:quad_method}), except for cases in which $\lambda$ is close to 1 or 0, that is $-0.5\lesssim a \lesssim -0.43$ and $0.25 \lesssim a \lesssim 0.3$ ($0.23 \lesssim \lambda(a)\lesssim 0.25$ and $0.83\lesssim \lambda(a)\lesssim 1$).

Additionally, we test the same methodology by replacing the uniform noise with a truncated normal distribution (\ref{eq:truncated_density}). The results for this experiment are shown in Figure~\ref{fig:quad_actual_non_linear_truncated}. As in the previous case, the approximations using the higher-order method (\ref{eq:quad_method_strict}) outperform the method in (\ref{eq:quad_method}), but in this case for all values of $-0.5\lesssim a \lesssim 0.5$ (meaning $0.19\lesssim \lambda(a) \lesssim 1$)
and the differences in approximation errors between the methods are more significant.

\begin{figure}[!t]
    \begin{overpic}[width=.49\textwidth]{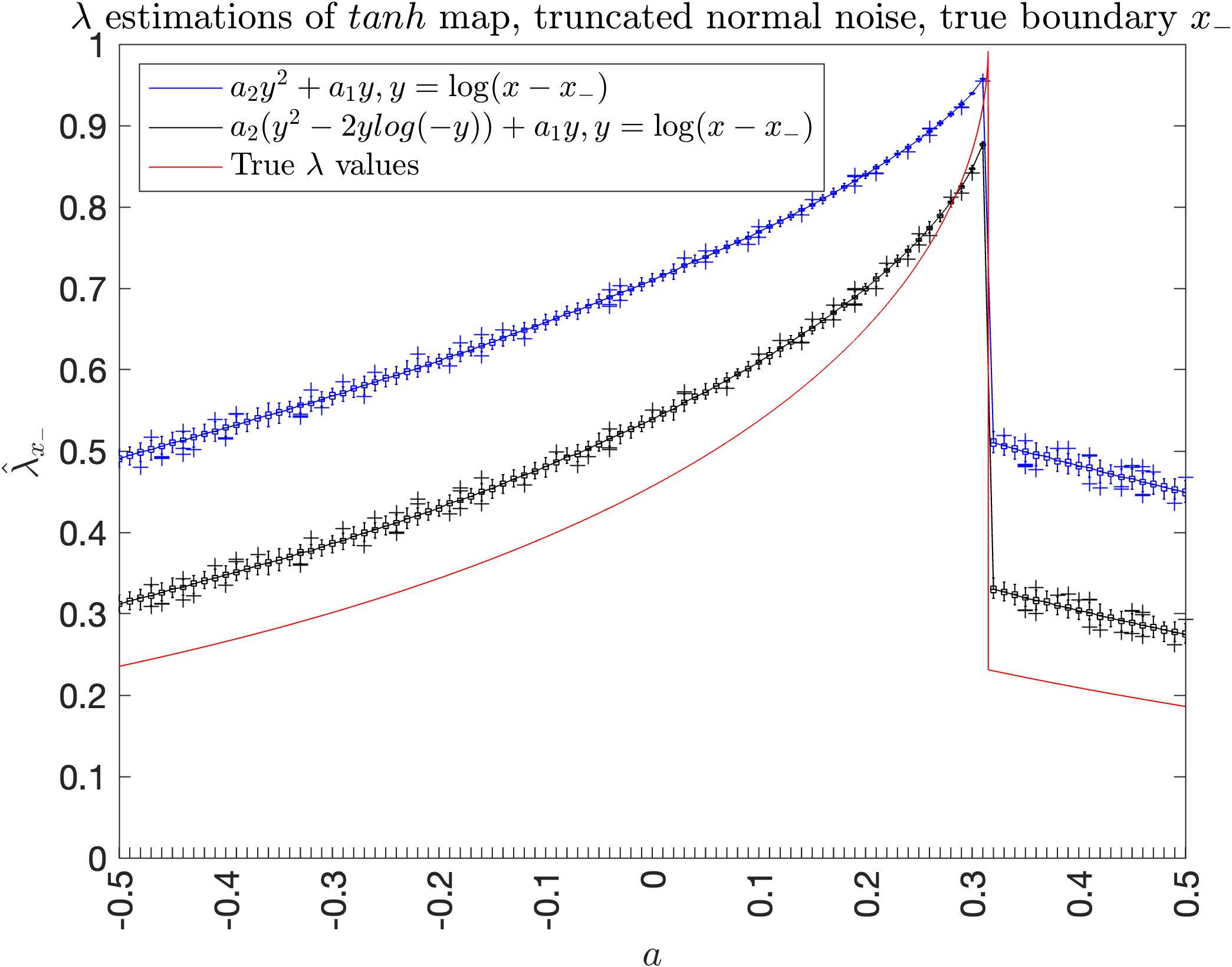}
    \put(0,0){
    (a)
    }
    \end{overpic}
    \begin{overpic}[width=.49\textwidth]{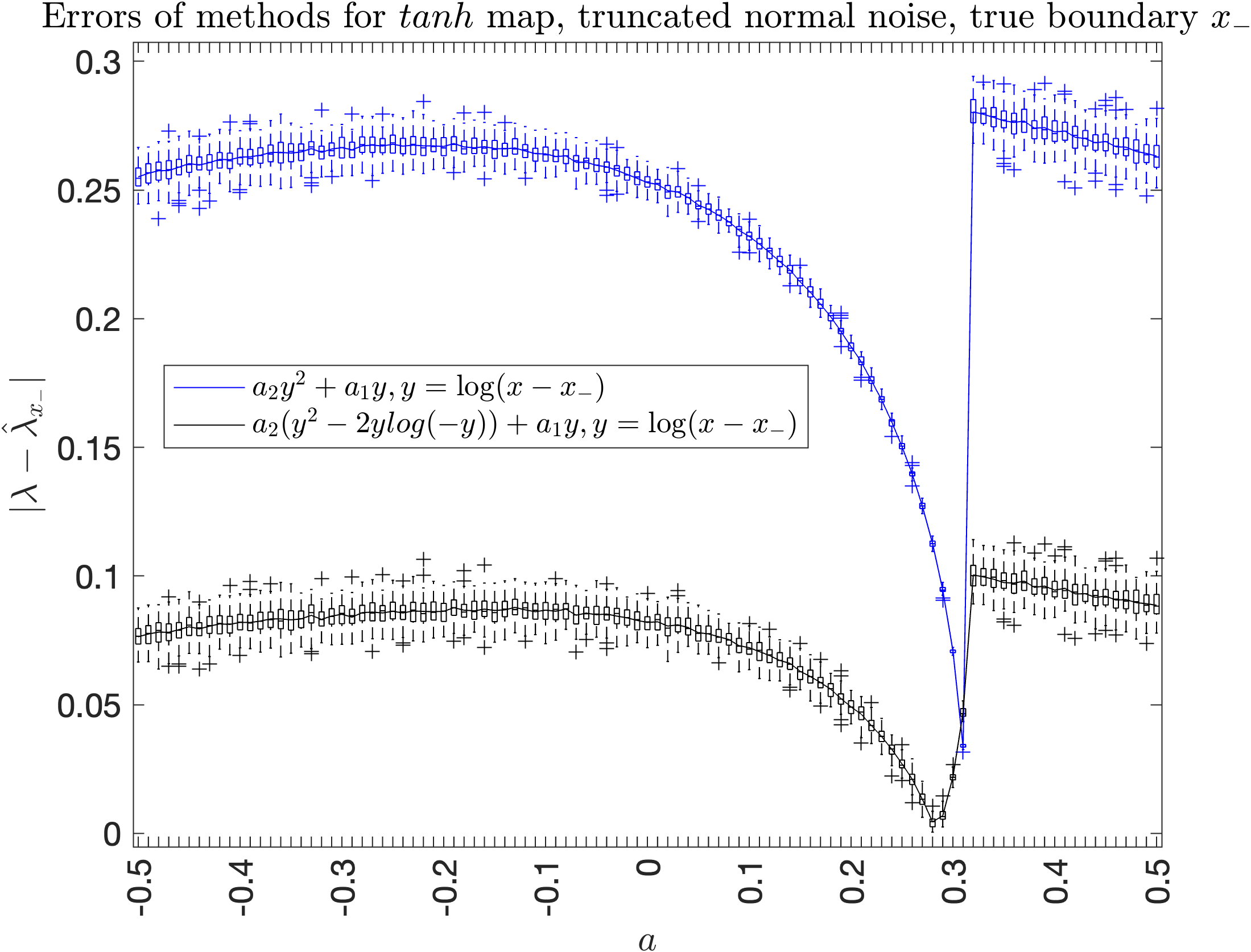}
    \put(0,0){
    (b)
    }
    \end{overpic}
    \captionsetup{width=\linewidth}\caption{Numerical approximation of $\lambda$ from the nonlinear map  (\ref{eq:non-linear_eq}) where the distribution of the noise $\xi_t$ is truncated normal (\ref{eq:truncated_density}) with $n= 10^5, b=200,$ and $q = 0.3$ using both (\ref{eq:quad_method}) and (\ref{eq:quad_method_strict}), where the boundary $x_-$ is the true boundary point, plotted in (a) while the errors in (b).}
    \label{fig:quad_actual_non_linear_truncated}
\end{figure}

\subsection{Estimated boundary from time series}\label{sec:unknown}

We now consider the case where the left boundary point $x_-$ is unknown and must be approximated from the time series. To this end, we incorporate an empty bin of zero height in the stationary density approximation $\hat{\phi}$, cf.~(\ref{eq:approx_density}). More precisely, we estimate the boundary as $\hat{x}_- = x_1 - \Delta z$ (see (\ref{eq:boundary_estimates})) where we recall that $\Delta z$ denotes the width of the histogram bins and $x_1$ is the midpoint of the left most bin. The corresponding estimator of $\lambda$ is denoted by $\hat{\lambda}_{\hat{x}_-}$. As in the previous subsection, we fix the hyperparameters to be $n=10^5$, $b=200$, and $q=0.3$. 

For comparison, we also include the \textit{interval method} proposed in \cite[Algorithm~11]{kuehn2018early}. In this approach, the derivative $\lambda = f'(x_-)$ is approximated by tracking one iteration of points $(y_{\gamma_i})_{1\leq i\leq k_1} \subset I_1$ and $(y_{\eta_i})_{1\leq i\leq k_1} \subset I_2$ in two disjoint intervals $I_1 = [a_1,b_1]$ and $I_2 = [a_2,b_2]$ near the boundary, with $b_1<a_2$. The approximation is then given by
\begin{equation}\label{eq:interval_method}
    \hat{\lambda} = \frac{\frac{1}{k_1}\sum_{i=1}^{k_1}y_{(\gamma_i+1)}- \frac{1}{k_2}\sum_{i=1}^{k_2}y_{(\eta_i+1)}}{b_2-a_1}.
\end{equation}

\subsubsection*{Linear map with uniform noise}

For the linear map with uniform noise, the numerical results for the approximation of $\lambda$ using the leading-order method (\ref{eq:quad_method}), the higher-order method (\ref{eq:quad_method_strict}) and the interval method are shown in Figure~\ref{fig:lambda_estimate_linear},  
It is evident that the higher-order asymptotics (\ref{eq:quad_method_strict}) perform significantly worse in this case. In contrast, the leading-order method (\ref{eq:quad_method}) generally outperforms the higher-order method (\ref{eq:quad_method_strict}) for $0.1\lesssim \lambda \lesssim 0.9$, and the interval method for $\lambda\lesssim 0.73$.

Comparing the estimations $\hat{\lambda}$ obtained with the true boundary $x_-$ in Figure~\ref{fig:lambda_actual_linear} to those with the estimated boundary in Figure~\ref{fig:lambda_estimate_linear}, we find that the latter is generally smaller, i.e. $\hat{\lambda}_{\hat{x}_-} < \hat{\lambda}_{x_-}$. This downward shift of the $\lambda$ estimates using the leading-order method (\ref{eq:first_order}), irrespective of the true value of $\lambda$, appears to result from the quadratic fitting procedure, see Appendix~\ref{APPENDIX:error_discussion} for further discussion.

\begin{figure}[!t]
    \centering
    \begin{overpic}[width=.49\textwidth]{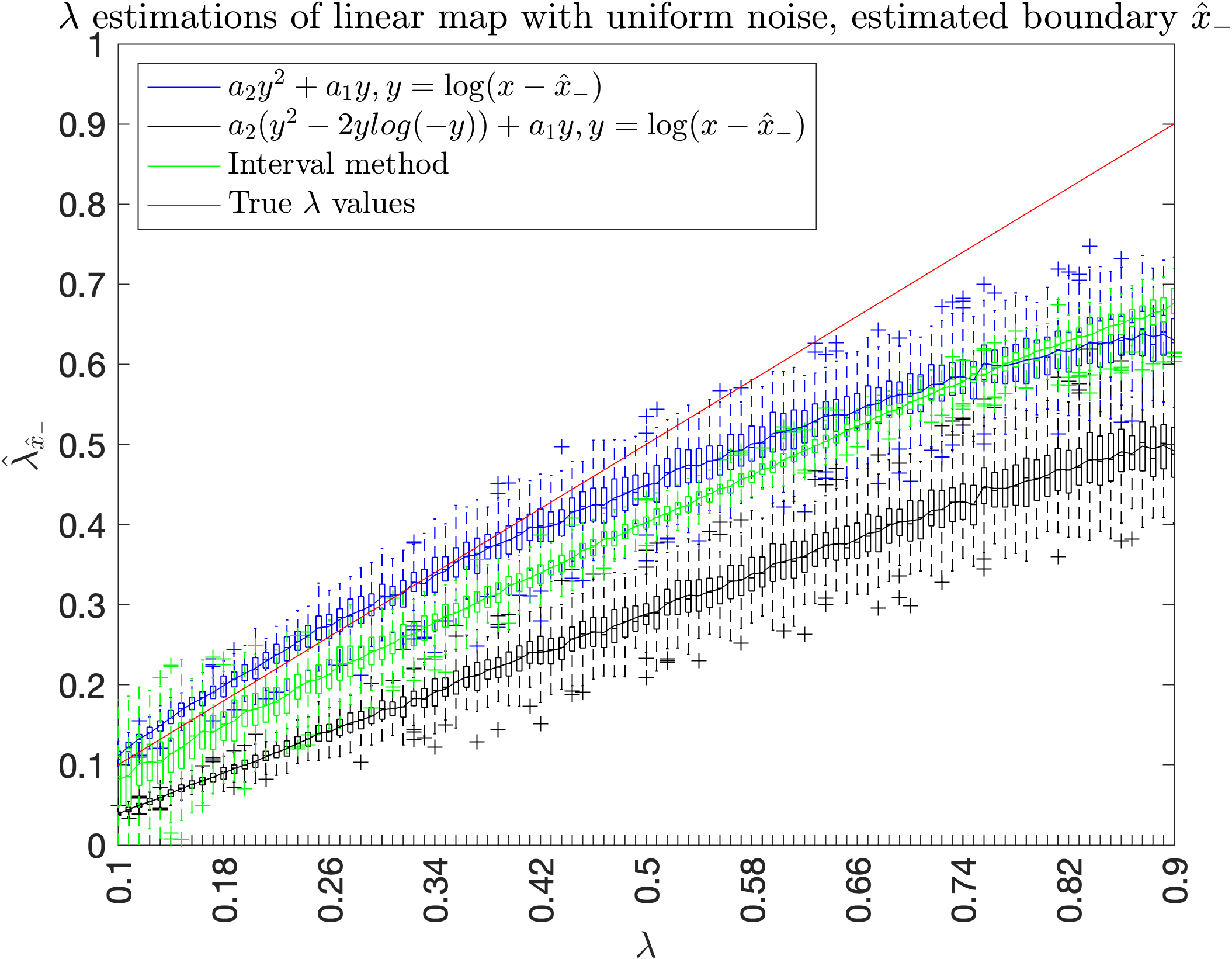}
    \put(0,0){
    (a)
    }
    \end{overpic}
    \begin{overpic}[width=.49\textwidth]{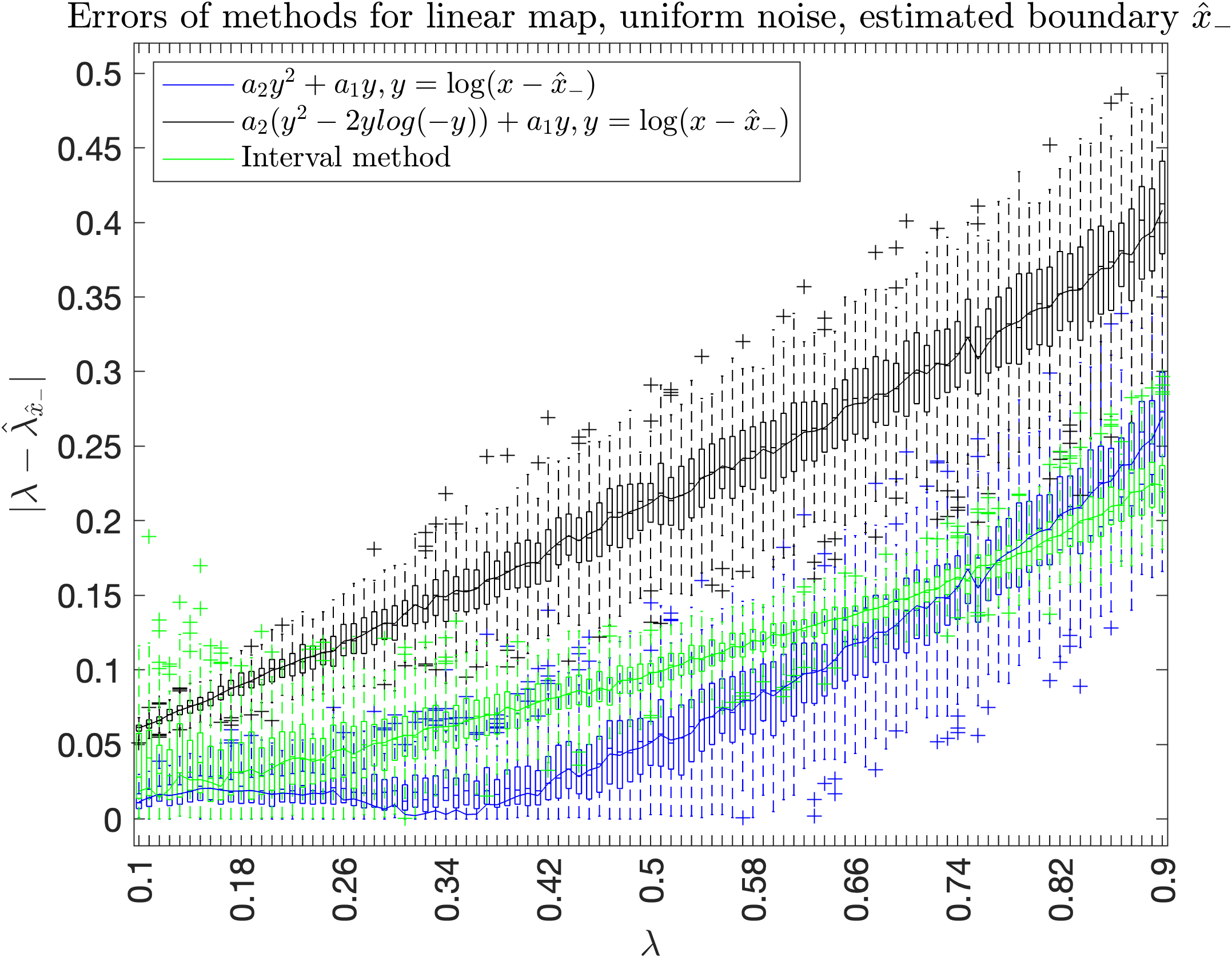}
    \put(0,0){
    (b)
    }
    \end{overpic}
    \captionsetup{width=\linewidth}\caption{Numerical approximation of $\lambda$ from the linear map with uniform noise (\ref{eq:linear_noise}) for $0.1 \leq \lambda \leq 0.9$, using (\ref{eq:quad_method}) depicted in blue and (\ref{eq:quad_method_strict}) in black, following Algorithm~\ref{alg:tailfit} with $n= 10^5, b=200,$ and $q = 0.3$. The boundary $x_-$ is estimated by adding an empty bin to the left of the histogram, as indicated in (\ref{eq:boundary_estimates}). The approximation using the interval method in \cite{kuehn2018early} is also plotted for comparison as green plot in (a), and the corresponding errors in (b). The interquartile ranges of the estimations are represented by the boxes with whiskers extending to all estimates, with outliers marked with crosses.}
    \label{fig:lambda_estimate_linear}
\end{figure}

\subsubsection*{Nonlinear map with bounded noise}\label{sec:tanh_estimated}
Similarly to the linear case, we present numerical results for the estimation of $\lambda$ over a range of parameter values $a$. Starting with $a=-0.5$, we compute the first 100 iterations of (\ref{eq:rand_diff}), with an initial condition $y_0 = 3$, and use the 100$^{\text{th}}$ iterate as the initial condition for Algorithm~\ref{alg:tailfit}. For each subsequent value of $a$ in $[-0.5,0.5]$ ($101$ uniformly spaced samples), 
the final iterate $y_{n}$ from the previous run is taken as the new initial condition, and Algorithm~\ref{alg:tailfit} is repeated.

We compare Algorithm~\ref{alg:tailfit} with the interval method from \cite{kuehn2018early} in Figure~\ref{fig:lambda_estimate_non_linear}. In Figure~\ref{fig:lambda_estimate_non_linear}(a), the topological bifurcation becomes evident at $a^* = \log(2-\sqrt{3}) + \sqrt{3} \approx 0.3151$ obtained from solving (\ref{eq:analytic_bif}). 
At this critical value, the left boundary point $x_-$ vanishes in a fold bifurcation of $f_-$, leaving the extremal map with only a single fixed point for $a>a^*$. Nevertheless, in our simulations with $n=10^5$ iterations, the transition to other regions of the state space is observed in the time series only after $a\approx 0.4$. This 
delay reflects the \textit{quasi-stationary} behaviour of bifurcating minimal invariant intervals near the bifurcation threshold. We refer the reader to \cite{olicon2021critical} for a theoretical approach to quasi-stationary distributions on bifurcating minimal invariant sets.

\begin{figure}[!t]
    \begin{overpic}[width=.49\textwidth]{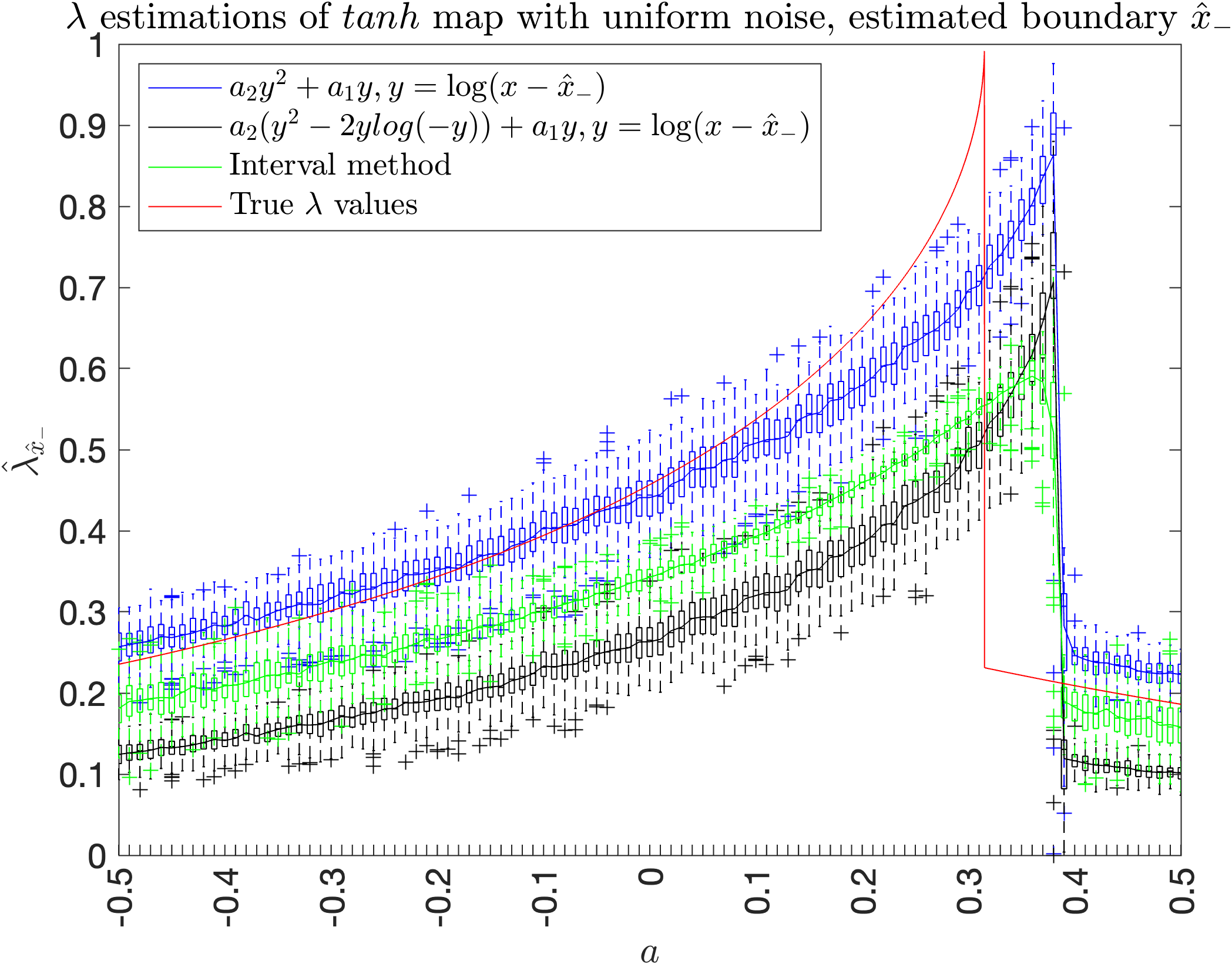}
    \put(0,0){
    (a)
    }
    \end{overpic}
    \begin{overpic}[width=.49\textwidth]{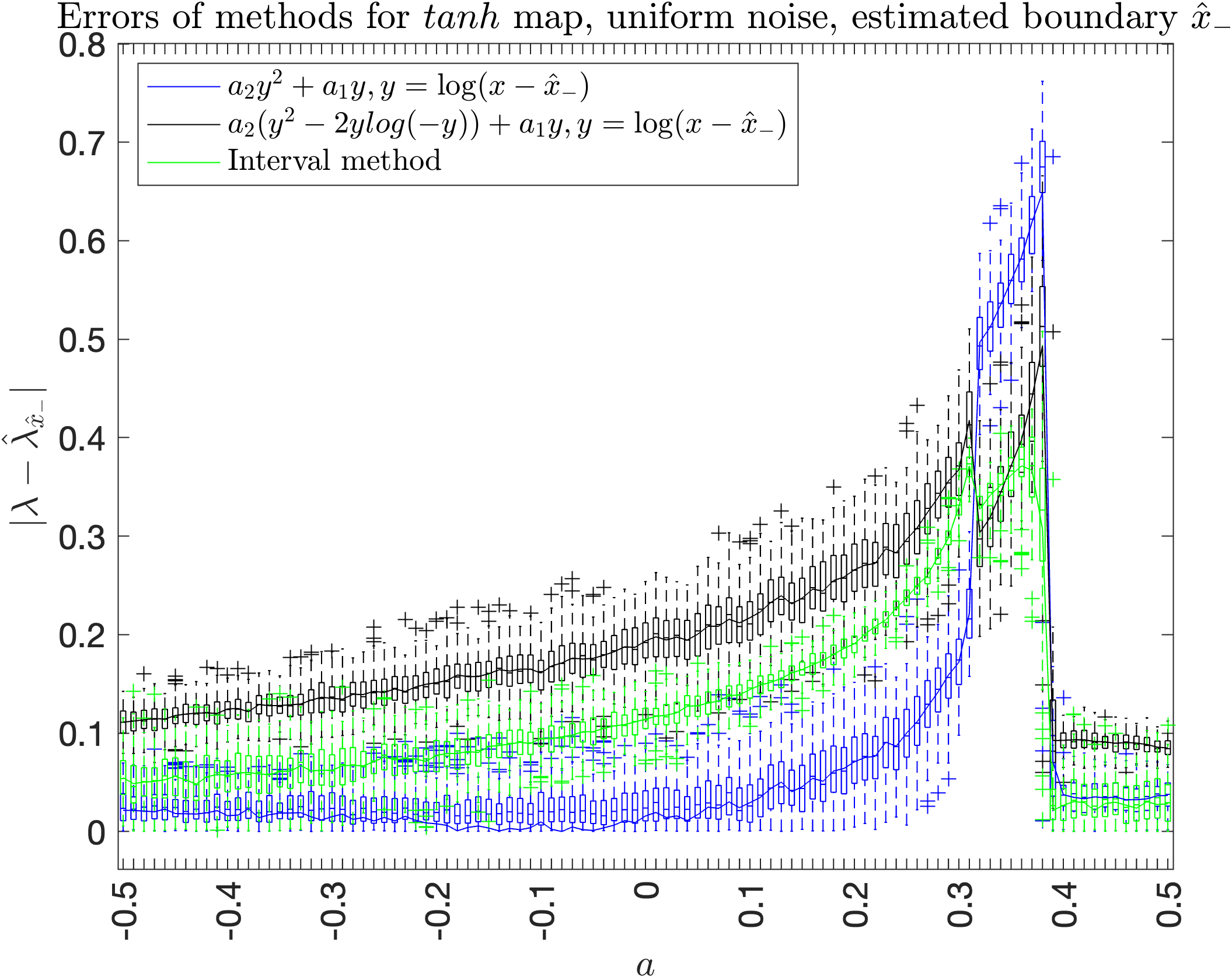}
    \put(0,0){
    (b)
    }
    \end{overpic}
    \captionsetup{width=\linewidth}\caption{Numerical approximation of $\lambda$ from the nonlinear map with uniform noise (\ref{eq:non-linear_eq}) for $-0.5 \leq a \leq 0.5$, using (\ref{eq:quad_method}) and (\ref{eq:quad_method_strict}) following Algorithm~\ref{alg:tailfit} with $n= 10^5, b=200,$ and $q = 0.3$. As parameter $a$ is increased, the initial value $y_0$ is the final iteration from the previous parameter. Together with the results using the interval method from \cite{kuehn2018early}, the approximations are plotted in (a) and their errors in (b).}
    \label{fig:lambda_estimate_non_linear}
\end{figure}

Similar to the linear case with uniform noise, the approximations obtained from the leading-order method (\ref{eq:quad_method}) generally outperform both the higher-order method (\ref{eq:quad_method_strict}) and the interval method\footnote{The interval method briefly outperforms the proposed method just after the topological bifurcation since the dynamics remain trapped in the quasi-stationary distribution.}. Errors in estimating the boundary point effectively shorten the tail of the stationary density, which in turn lowers the estimator $\hat{\lambda}$ (see Appendix~\ref{APPENDIX:error_discussion} for further discussion). Interestingly, this systematic underestimations align favourably with the leading-order method (\ref{eq:quad_method}) for small values of $\lambda$, as the method tends to overestimate $\lambda$ when the true boundary $x_-$ is known, as evidenced in Figure~\ref{fig:lambda_actual_non_linear}(a).

\begin{figure}[!t]
    \begin{overpic}[width=.49\textwidth]{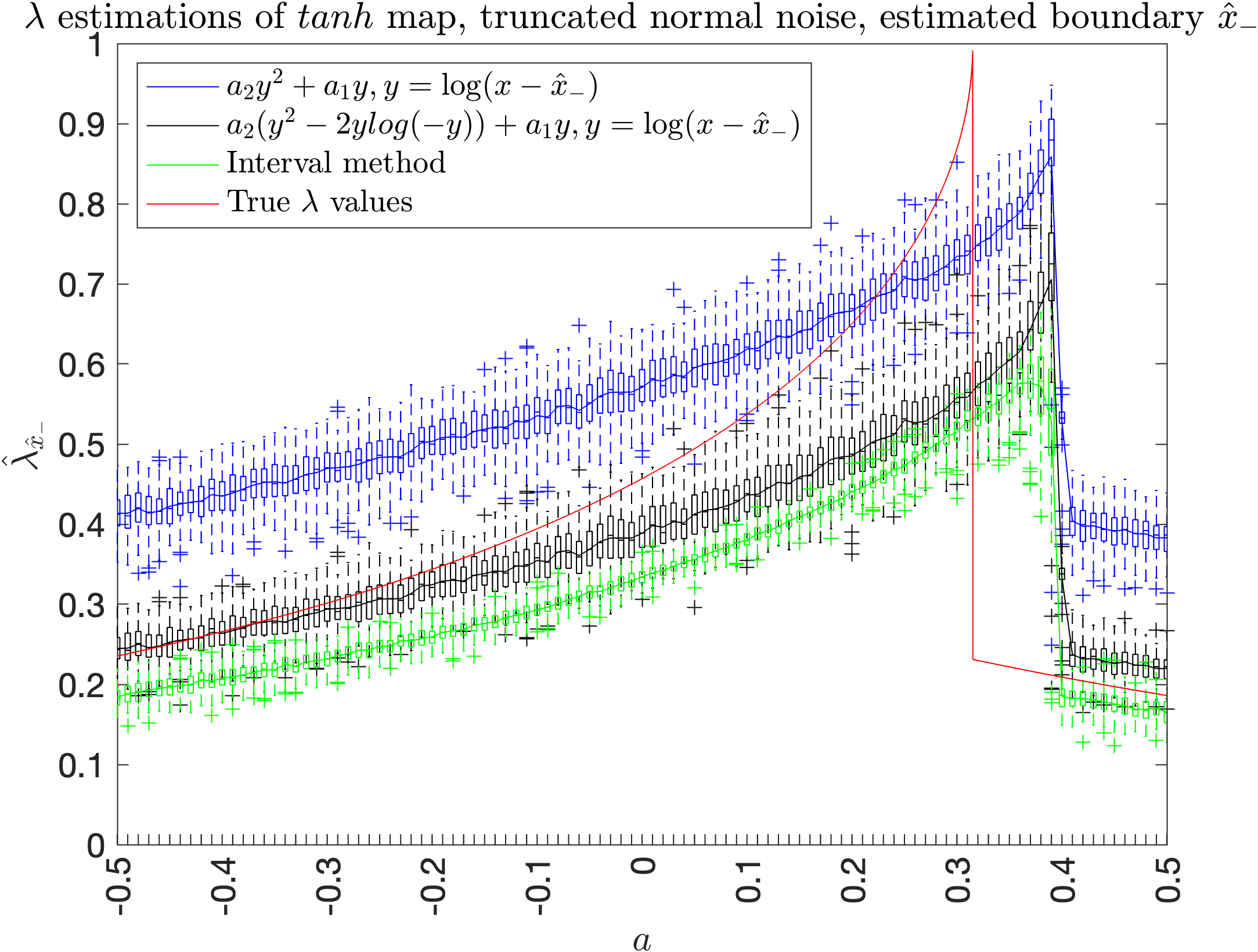}
    \put(0,0){
    (a)
    }
    \end{overpic}
    \begin{overpic}[width=.49\textwidth]{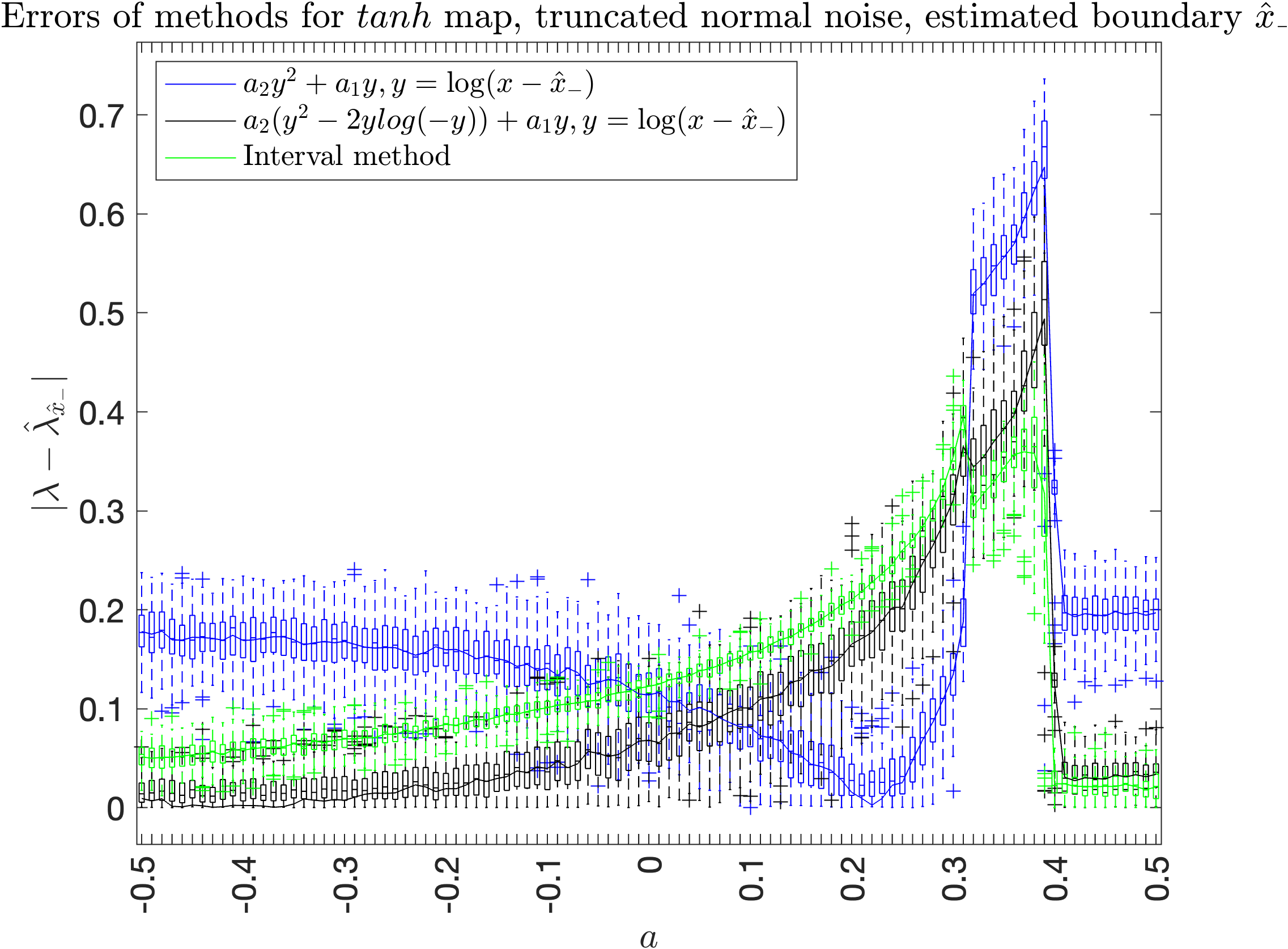}
    \put(0,0){
    (b)
    }
    \end{overpic}
    \captionsetup{width=\linewidth}\caption{Numerical approximation of $\lambda$ from the nonlinear map  (\ref{eq:non-linear_eq}) where the distribution of noise $\xi_t$ is truncated normal, using similar procedures to Figure~\ref{fig:lambda_estimate_non_linear} .} 
    \label{fig:quad_estimate_non_linear_truncated}

\end{figure}

In the case where the noise is sampled from a truncated normal distribution (\ref{eq:truncated_density}), the results of both methods (\ref{eq:quad_method}) and (\ref{eq:quad_method_strict}) are shown in Figure~\ref{fig:quad_estimate_non_linear_truncated}, together with those obtained from the interval method. In contrast to the uniform noise case, the higher-order method (\ref{eq:quad_method_strict}) outperforms both the leading-order method (\ref{eq:quad_method}) and the interval method for smaller values of $\lambda$, while the leading-order method performs best for $\lambda\approx 1$. 

As in the linear case, the estimator using the estimated boundary, $\hat{\lambda}_{x_-}$, is consistently less than the estimator using the true boundary, $\hat{\lambda}_{\hat{x}_-}$; that is $\hat{\lambda}_{\hat{x}_-} < \hat{\lambda}_{x_-}$ (cf. Figure~\ref{fig:lambda_actual_non_linear} and \ref{fig:quad_actual_non_linear_truncated} compared with Figure~\ref{fig:lambda_estimate_non_linear} and \ref{fig:quad_estimate_non_linear_truncated}). We refer the reader to Appendix~\ref{APPENDIX:error_discussion} for further discussion of how boundary estimation affects the approximation of $\lambda$.

For the truncated noise case, 
the leading-order method (\ref{eq:quad_method}) with the true boundary $x_-$ tends to substantially overestimate $\lambda$, particularly for small values of $\lambda$, as shown in Figure~\ref{fig:quad_actual_non_linear_truncated}. Consequently, even when the boundary is estimated from the histogram data, the corresponding estimator $\hat{\lambda}_{\hat{x}_-}$ still strongly overestimates $\lambda$ in this regime.
In contrast, for the higher-order method (\ref{eq:quad_method_strict}), the reduction in $\lambda$ estimates using estimated boundary $\hat{x}_-$ instead of $x_-$ actually improves accuracy, bringing the estimates closer to the true values of $\lambda$.

Overall, we conclude that when the boundary of the support of the stationary distribution is unknown, the leading-order method (\ref{eq:quad_method}) performs better for uniform noise for any $\lambda \in (0,1)$, while the higher-order method (\ref{eq:quad_method_strict}) generally performs better for truncated normal noise, provided that the approximation $\hat{\phi}$ in (\ref{eq:approx_density}) approximates well the stationary density rather than the quasi-stationary density shortly after the topological bifurcation.
In practical scenarios where only the time series is available, without prior knowledge of the underlying dynamical system or the distribution of the noise, both methods can be used to estimate $\lambda$, thereby providing an early warning signal to an onset of a bifurcation of the stochastic process.

\section{Conclusion}
In this paper, we introduced two numerical methods (\ref{eq:quad_method}) and (\ref{eq:quad_method_strict}), designed to estimate $\hat{\lambda}$ for the derivatives of extremal maps $\lambda = f_-'(x_-)$ or $ f_+'(x_+)$ at the boundary of the support  $[x_-,x_+]$ of a stationary distribution for random difference equations with bounded noise, as described in (\ref{eq:rand_diff}). These methods are based on asymptotic expansions of the stationary density tail, namely (\ref{eq:first_order}) (leading-order) and in (\ref{eq:higher_order}) (higher-order), derived from \cite{olicon24tail} and Theorem~\ref{THM:main}, respectively. Since $\lambda=1$ at a topological bifurcation of the support of the stationary density, the proximity of $\hat{\lambda}$ to $1$ provides an early warning signal of such a bifurcation. 

To validate these methods, we applied them to both linear (\ref{eq:linear_noise}) and nonlinear (\ref{eq:non-linear_eq}) maps with additive uniform noise, and additionally to the nonlinear map with truncated normal noise. 
When the boundary $x_-$ is known, both methods give good approximations of $\lambda$, with maximum errors $|\lambda-\hat{\lambda}| < 0.13$ in the uniform noise cases. For the nonlinear map with truncated normal noise, however, the higher-order method significantly outperforms the leading-order method. Further investigation is required to understand the role of the noise distribution in these approximations, which is left for future work. In particular, one should expect different results when the noise density $p:[-\varepsilon,\varepsilon]\rightarrow \mathbb{R}_+$ has a zero of order $r$ at $\pm\varepsilon$, since in this case, the asymptotic relation \eqref{eq:first_order} must be corrected by a factor depending on $r$, as shown in \cite{olicon24tail}.

In practical settings where both the boundary and the underlying system are unknown, the leading-order method (\ref{eq:quad_method}) outperforms both the higher-order method (\ref{eq:quad_method_strict}) and the interval method from \cite{kuehn2018early} for most values of $\lambda$. In contrast, in the case of truncated normal noise, the higher-order method (\ref{eq:quad_method_strict}) gives more accurate results for small $\lambda$, while the leading-order method performs better for large $\lambda$.

A key challenge highlighted throughout this work is the 
reliance of data points in the time series near 
the boundary of the support of the stationary distribution. 
This set, however, is by definition a \textit{rare event}, and those data points are scarce and seldom observed in the time series. The main advantage of the asymptotic expansions \eqref{eq:first_order} and \eqref{eq:higher_order} is that they allow 
the use of information farther from the boundary point. We conjecture that the $\mathcal{O}(\ln(x-x_-))$ terms in \eqref{eq:higher_order} play a crucial role in improving the accuracy of Algorithm~\ref{alg:tailfit}.

\section*{Acknowledgments}
 KA and WHT have been supported by the Project of Intelligent Mobility Society Design, Social Cooperation Program, UTokyo, JST Moonshot R \& D Grant Number JPMJMS2021. G.O.-M. has been supported by the Deutsche Forschungsgemeinschaft (DFG) through grant CRC 1114 ``Scaling Cascades in Complex Systems'', Project Number 235221301, Project A02 ``Multiscale data and asymptotic model assimilation for atmospheric flows '', and through Germany's Excellence Strategy -- The Berlin Mathematics Research Center MATH+ EXC-2046/1, project 390685689,  subproject AA1-8. JSWL have been supported by the EPSRC grants EP/W009455/1 and  EP/Y020669/1, and acknowledges support from the EPSRC Centre for Doctoral Training in Mathematics of Random Systems: Analysis, Modelling and Simulation (EP/S023925/1) and thanks IRCN (Tokyo) and GUST (Kuwait) for their research support. KA has been supported by Institute of AI and Beyond of UTokyo, the International Research Center for Neurointelligence (WPI-IRCN) at The University of Tokyo Institutes for Advanced Study (UTIAS), JSPS KAKENHI Grant Number JP20H05921, Cross-ministerial Strategic Innovation Promotion Program (SIP) and the 3rd period of SIP “ Smart energy management system ” Grant Number JPJ012207.

\addcontentsline{toc}{section}{References}
\bibliography{reference}

\appendix
\include{appendix_final_arxiv}

\end{document}

%% file: appendix_final_arxiv.tex
\section{Non-increasing variance example near tipping point}\label{appen:variance_decrease}

We present an example of a random  difference equation, where the variance of the time series does not provide a reliable early indicator of topological bifurcations, as illustrated in Figure~\ref{fig:variance} of the main text. The system is defined by
\begin{equation}
\label{eq:variance}
\begin{aligned}
    &y_{t+1} = f_a(y_t) + \xi_t,\\ 
    &\text{where } \;f_a(x) = f(\exp{(a)}x + g(a)) + h(a) + 0.5,\\
    & f(x) = 3\tanh{(x/2)},\\
    &g(x) = -0.72(x+0.8)^3 + 0.36,\\
    &h(x) = -0.2 (x+0.0011)^{1/3} + 0.021,  
    \end{aligned}
\end{equation}

\noindent with bounded noise $\|\xi_t\| \leq \varepsilon = 0.8$. The map $f_a$ is a modification of the nonlinear map (\ref{eq:non-linear_eq}) in the main text, constructed so that the attraction of the deterministic fixed point strengthens as $a$ increases, while the derivative $f'_a(x_-)$ at the left boundary $x_-$ of the support of the stationary distribution increases and approaches one.

\begin{figure}[!b]
    \begin{overpic}[width=.49\textwidth]{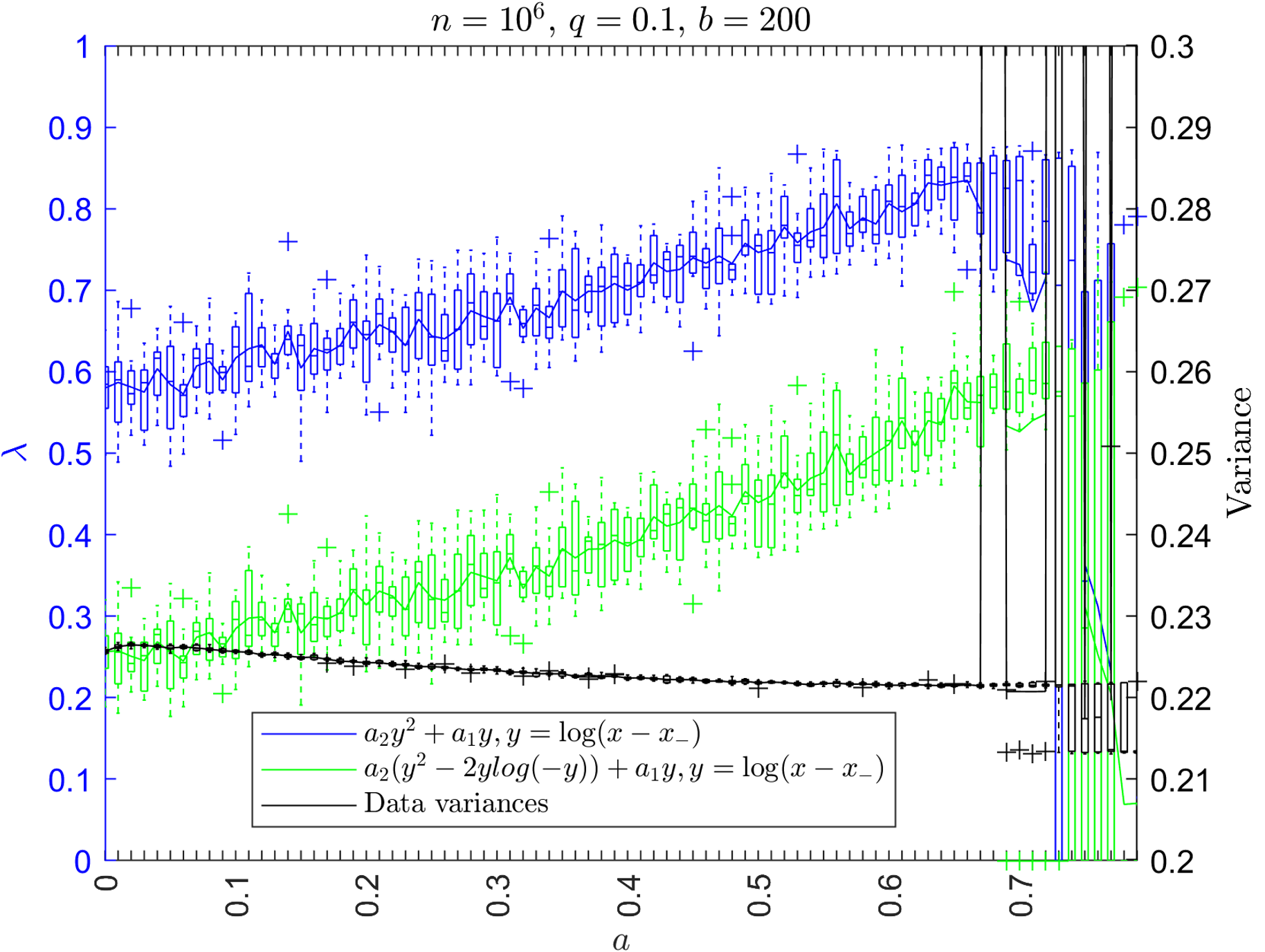}
    \put(0,0){
    (a)
    }
    \end{overpic}
    \begin{overpic}[width=.49\textwidth]{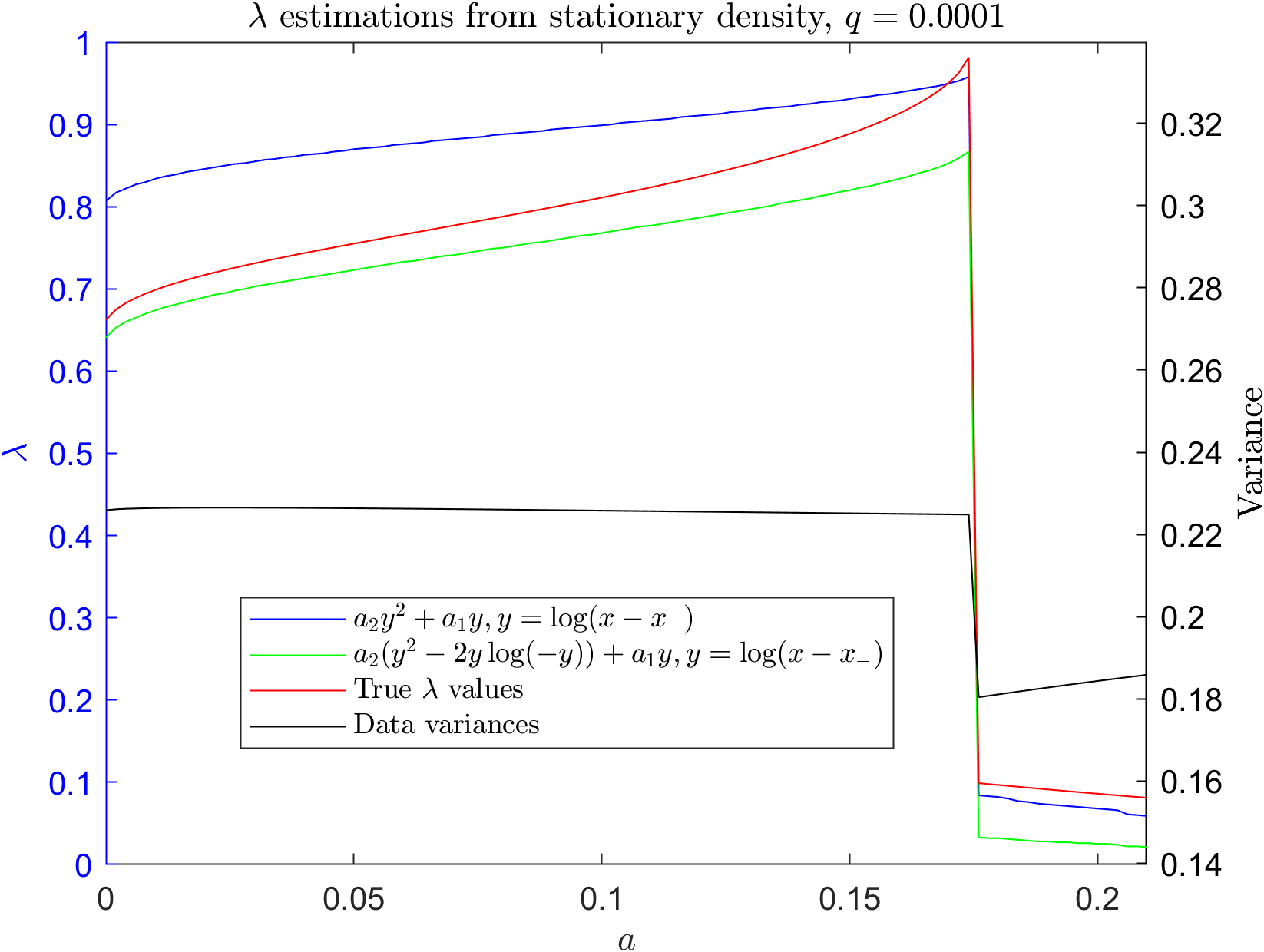}
    \put(0,0){
    (b)
    }
    \end{overpic}
\caption{Comparison between our proposed methods and the variance of data. Estimations of $\lambda$ using time series are shown in (a), while (b) estimates $\lambda$ from stationary density approximated using Ulam's method on the known support $[x_-,x_+]$ with $13$ subdivisions, thus $b = 2^{13}$, with $q = 0.0001$.}
\label{fig:variance_decrease_ex_result}
\end{figure}

We approximate the derivative of the lower extremal map $f_a(y)-\varepsilon$ at the left boundary $x_-$ using both the leading-order method (\ref{eq:quad_method}) and higher-order method (\ref{eq:quad_method_strict}). The results, computed with $n = 10^6, q = 0.1, n = 200$ are shown in Figure~\ref{fig:variance_decrease_ex_result}(a) below. They demonstrate that, although the variance of data decreases slightly (remaining almost constant), the estimators $\hat{\lambda}$ increase towards one as the system approaches the transition. As discussed in the main text, the actual transition of the data is observed only well after the topological bifurcation at $a\approx 0.175$, since trajectories remain in a meta-stable region. In Figure~\ref{fig:variance_decrease_ex_result}(b), the stationary density is approximated using Ulam's method under the assumption that the support boundaries are known. Consistent with the results presented in the main text, the higher-order method (\ref{eq:quad_method_strict}) provides more accurate estimates of $\lambda$ when the support boundaries are known. This also supports the expectation that, as $n \to \infty$, $b \to \infty$ and $q \to 0$, the estimator becomes increasingly accurate, i.e. $\hat{\lambda} \to \lambda$.

\section{Asymptotic expansion of the stationary density}
\label{APPENDIX:proof}
In this section we prove Theorem~\ref{THM:main}. We assume without loss of generality that $x_-=0$. Let $x_0$ be sufficiently close to $0$ such that \eqref{eq:main_ineqs} holds for all $x\in(0,x_0)$, where $n\equiv n_{x_0}^x$. We first prove the following auxiliary statement
\begin{proposition}
\label{PROP: expansion_factorial}
Let $n\equiv n_{x_0}^x$ as in \eqref{eq:hitting_time_det}. Then, as $x\rightarrow 0$,
\begin{equation}
    \label{eq:asymptotic_factorial}
\log n! = \left(\frac{1}{\log \lambda}\right) \log x \cdot \log\log \frac{1}{x} + \mathcal{O}(\log x).
\end{equation}
\end{proposition}
\begin{proof}
From \cite[Lemma 4.1]{olicon24tail}, the scaling law for $n$ as $x\rightarrow0$ is expressed as 
\begin{equation}
    \label{eq:scalinglaw_n}
    n=\log x\left[ \frac{1}{\log\lambda} + \mathcal{O}\left( \frac{1}{\log x}\right)\right] = \log\left(\frac{1}{x}\right)\left[  \frac{1}{\log \frac{1}{\lambda}} + \mathcal{O}\left( \frac{1}{\log x}\right)\right],
\end{equation}
so that by taking $\log$ on both sides one gets
\begin{equation}
    \label{eq:log_n_asymptotic}
    \log n = \log\log \frac{1}{x} + \mathcal{O}(1).
\end{equation}
We expand $\log n!$ as a sum, so that
\[
    \log n! =\sum_{j=1}^n \log j \leq \int_2^{n+1}\log s\;ds = (n+1)\log(n+1)-(n+1)-2\log 2+2.
\]
Using \eqref{eq:scalinglaw_n} and \eqref{eq:log_n_asymptotic}, it yields
\[
    \log n! \leq \frac{\log x \log\log\frac{1}{x}}{\log \lambda} + \mathcal{O}(\log x).
\]

Analogously, for the lower bound,
\[
    \log n! \geq \int_1^n{\log s\;ds} = n\log n-n+1 = \frac{\log x \log\log\frac{1}{x}}{\log \lambda} + \mathcal{O}(\log x),
\]
which implies that \eqref{eq:asymptotic_factorial} holds.
\end{proof}

As an auxiliary result, we show that the stationary density is nondecreasing near its left boundary $x_-=0$. We start by noticing that continuity of $p$ and $p(-\varepsilon)=p(x-f(y))\vert_{y=f_-^{-1}(x)}>0$ implies that there exist $C_1,C_2>0$ and $x_0$ sufficiently close to $0$, such that  the stationary density satisfies that
 \begin{equation}
 \label{eq:iterate1}
    C_1\int_0^{f_-^{-1}(x)} \phi(y)dy
 \leq \phi(x) \leq C_2\int_0^{f_-^{-1}(x)} \phi(y)dy,
  \end{equation}
for any $x\in[0,x_0]$. We use the estimates above in the following statements.
\begin{proposition}
    \label{PROP:density_decreasing}
The stationary density $\phi$ is nondecreasing near $0$.
 \end{proposition}
\begin{proof}

We show that for all $x$ sufficiently close to $0$, $\phi'(x)>0$. By Leibniz integral rule, 
\begin{equation}
    \label{eq:derivative_phi}
\phi'(x)= \frac{p(-\varepsilon) \phi(f_-^{-1}(x))}{\varepsilon f_-'(f_-^{-1}(x))} + \frac{1}{\varepsilon}\int_0^{f_-^{-1}(x)}p'(x-f(y))\phi(y)dy.
\end{equation}

In order to bound $\phi'(x)$ from below, let
\begin{equation}
    B_1(x):= \frac{p(-\varepsilon)}{\varepsilon f_-'(f_-^{-1}(x))}.
\end{equation}
We multiply and divide by $p(x-f(y))$ in the integrand in \eqref{eq:derivative_phi}, so that
\begin{align*}
    p'(x-f(y)) & \geq -\Vert p'\Vert_\infty  = -\frac{p(x-f(y))\cdot \Vert p' \Vert_\infty}{p(x-f(y))} \geq \\
    & \geq -\frac{ \Vert p'\Vert_\infty }{\inf_{y\in [0,f_-^{-1}(x)]}p(x-f(y))} \; p(x-f(y)) \equiv -B_2(x) p(x-f(y)).
\end{align*}
Using the above inequality in \eqref{eq:derivative_phi} yields
\[   
\phi'(x) \geq B_1(x) \phi(f_-^{-1}(x)) - \frac{B_2(x)}{\varepsilon} \int_0^{f_-^{-1}(x)}p(x-f(y))\phi(y) dy = B_1(x) \phi(f_-^{-1}(x)) - B_2(x)\phi(x),
\]
where the last equality, follows from \eqref{eq:phi_near_0}.
Hence, a sufficient condition for $\phi'(x)$ to be positive is that
\begin{equation}
      \label{eq:cond_positive_derivative}
\frac{\phi(f_-^{-1}(x))}{\phi(x)}  > \frac{B_2(x)}{B_1(x)}.
\end{equation}

We show this is the case for $x$ sufficiently close to $0$, by showing that $\phi(f_-^{-1}(x))/\phi(x) \rightarrow \infty$ as $x\rightarrow 0$, since 
\[
\lim_{x\to 0}\frac{B_2(x)}{B_1(x)}= \frac{\varepsilon \lambda\Vert p'\Vert_\infty}{p(-\varepsilon)^2}<\infty.\] Indeed, assuming that \eqref{eq:iterate1} holds for any $x\in[0,x_0]$, if $f_-^{-1}(x)< x_0$, then we can iterate \eqref{eq:iterate1} to obtain, as an upper bound,
\[
    \phi(x)\leq C_2^2 \int_0^{f_-^{-1}(x)}\int_0^{f_-^{-1}(y)} \phi(z) dz dy.
\]
Exchanging the order of integration, one gets
\begin{equation}
\label{eq:step2}
    \phi(x)\leq C_2^2\int_0^{f_-^{-2}(x)} \int_{f_-(z)}^{f_-^{-1}(x)} \phi(z)dydz =
    C_2^2\int_0^{f_-^{-2}(x)} \left( f_-^{-1}(x)-f_-(z) \right) \phi(z)dz.
\end{equation}
From above, it follows that
\[
\phi(x) \leq \frac{C_2^2 f_-^{-1}(x)}{B_3(x)} \int_0^{f_-^{-2}(x)}p(f_-^{-1}(x)-f(y))\phi(y) dy = \frac{\varepsilon C_2^2 f_-^{-1}(x)}{B_3(x)} \phi(f_-^{-1}(x)),
\]
where 
\[
B_3(x) = \sup_{y\in[0,f_-^{-2}(x)]}
p(f^{-1}_-(x)-f(y)) \rightarrow p(-\varepsilon),\] 
as $x\rightarrow 0$. Therefore,
\[
    \frac{\phi(f_-^{-1}(x))}{\phi(x)} \geq \frac{B_3(x)}{\varepsilon C_2^2 f_-^{-1}(x)} \rightarrow \infty,
\]
as $x\rightarrow 0$. Hence, for sufficiently small $x$, \eqref{eq:cond_positive_derivative} holds, and the result follows.
\end{proof}

In order to prove Theorem~\ref{THM:main}, we assume for now that $f_-'$ is increasing near $0$, for which we can prove Lemma~\ref{LEMMA:main}. Otherwise, one can assume that $f_-'$ is either decreasing or constant around $0$ and formulate an analogous statement, and proceed with similar arguments.

\begin{proof}[Proof of Lemma~\ref{LEMMA:main}]
The rest of the proof follows closely \cite[Lemma 3.1(a)]{olicon24tail}.
Let $x_0$ be sufficiently small, and $x\in[0,x_0]$. Analogously as \eqref{eq:step2}, if $f_-^{-1}(x)< x_0$, we can iterate \eqref{eq:iterate1} to obtain, as a lower bound,
\[
    \phi(x)\geq C_1^2\int_0^{f_-^{-2}(x)} \int_{f_-(z)}^{f_-^{-1}(x)} \phi(z)dydz =
    C_1^2\int_0^{f_-^{-2}(x)} \left( f_-^{-1}(x)-f_-(z) \right) \phi(z)dz.
\]
By performing the change of variables $y=f_-(z)$, due to the monotonicity of $f_-'$, it yields
\begin{align*}
    \phi(x) & \geq C_1^2 \int_0^{f_-^{-1}(x)} \frac{f_-^{-1}(x)-y}{f_-'(f_-^{-1}(y))}\cdot \phi(f_-^{-1}(y))dy \\
    & \geq C_1^2 \left(\frac{1}{f_-'(f_-^{-2}(x))}\right) \int_0^{f_-^{-1}(x)} \left(f_-^{-1}(x)-y\right) \phi(f_-^{-1}(y))dy.
\end{align*}
If $f_-^{-2}(x)<x_0$, we can plug \eqref{eq:iterate1} in the integrand to obtain
\[
    \phi(x)\geq C_1^3 \left(\frac{1}{f_-'(f_-^{-2}(x))}\right)
    \int_0^{f_-^{-1}(x)}\int_0^{f_-^{-2}(y)}(f_-^{-1}(x)-y)\phi(z)dzdy.
\]
Again, by exchanging the order of integration, 
\begin{align*}
\phi(x) & \geq \frac{C_1^3}{1} \left(\frac{1}{f_-'(f_-^{-2}(x))}\right)
    \int_0^{f_-^{-3}(x)}\int_{f_-^2(z)}^{f_-^{-1}(x)}(f_-^{-1}(x)-y)\phi(z)dydz \\
    & \geq \frac{C_1^3}{1\cdot 2} \left(\frac{1}{f_-'(f_-^{-2}(x))}\right)
    \int_0^{f_-^{-3}(x)}(f_-^{-1}(x)-f_-^2(z))^2\phi(z)dz \\
    & \geq \frac{C_1^3}{1\cdot 2} \left(\frac{1}{f_-'(f_-^{-2}(x))}\right)
    \left(\frac{1}{f_-'(f_-^{-2}(x)) f_-'(f_-^{-3}(x))}\right)
    \int_0^{f_-^{-1}(x)}(f_-^{-1}(x)-y)^2\phi(f_-^{-2}(y))dy.
\end{align*}

Iterating this process for $n-1$ times, where $n=n_{x_0}^x$, the lower bound in \eqref{eq:main_ineqs} is obtained. The upper bound can be deduced in a similar fashion.
\end{proof}

\begin{proof}[Proof of Theorem~\ref{THM:main}]
We proceed by obtaining upper and lower bounds for $\log\phi(x)$. Let $x_0>0$ be as in the proof of Lemma~\ref{LEMMA:main}, and assume in particular that the extremal map $f_-$, its inverse $f_-^{-1}$, and its derivative $f_-'$, satisfy
\begin{equation}
    \label{eq:estimates_f-}
    \begin{split}
      \lambda x + \alpha_1 x^2 \leq  f_-(x) \leq\lambda x + \alpha_2 x^2 \\
      \frac{1}{\lambda}x +\beta_1 x^2 \leq  f_-^{-1}(x)  \leq \frac{1}{\lambda}x +\beta_2 x^2 \\
    \lambda +2\alpha_1 x \leq  f_-'(x) \leq \lambda + 2\alpha_2 x,
    \end{split}
\end{equation}
for all $x\in[0,x_0]$, where $\alpha_1,\alpha_2,\beta_1,\beta_2\in\mathbb{R}$.

\noindent\textit{Upper bound.} We start the proof by taking $\log$ in \eqref{eq:main_ineqs} so that
\begin{equation}
    \label{eq:main_ineq_ln}
    \log\phi(x)\leq n\log C_2-\log n! -\frac{n(n-1)}{2}\log\lambda+ n\log f_-^{-1}(x) + \log\phi_0,
\end{equation}
where $\phi_0:=\phi(x_0)$. Since $f_-^{-1}(x)=\frac{x}{\lambda}+o(x)$, due to \eqref{eq:scalinglaw_n}, it yields that
\begin{align*}
    \log\phi(x)\leq & -\log n!+ \frac{\log^2 x}{2\log \lambda} + \mathcal{O}(\log x).
\end{align*}
Hence, due to \eqref{eq:asymptotic_factorial},
\begin{equation}
\label{eq:upper_bound}
    \log\phi(x) \leq a_2\log^2x -2a_2 \log x\cdot \log\log\frac{1}{x}+\mathcal{O}(\log x). 
\end{equation}

\noindent\textit{Lower bound.}  We bound from below the integral in \eqref{eq:main_ineqs} by integrating on the subinterval $[x,f_-^{-1}(x)]$. Hence,
\begin{align*}
     \int_{0}^{f_-^{-1}(x)} (f_-^{-1}(x)-y)^{n-1}\phi(f_-^{-n+1}(y))dy \geq & \phi(f_-^{-n+1}(x)) \cdot\int_{x}^{f_-^{-1}(x)}(f_-^{-1}(x)-y)^{n-1}dy   \\
     \geq & \frac{(f_-^{-1}(x)-x)^n \cdot \phi(f_-(x_0))}{n},
\end{align*}
where in both inequalities the monotonicity of $\phi$ near $0$ was used. Therefore, we obtain from \eqref{eq:main_ineqs} the lower bound
\begin{equation}
    \label{eq:first_step}
    \phi(x) \geq \frac{\hat{C}_1^n x^n}{n!} \left[ \prod_{m=2}^n \prod_{j=2}^m \frac{1}{f_-'(f_-^{-j}(x))} \right]  \phi(f_-(x_0)),
\end{equation}
where $\hat{C}_1:= C_1\left( \frac{1-\lambda}{\lambda}-\delta\right)$, where $\delta>0$ is a sufficiently small constant. Taking $\log$ in \eqref{eq:first_step}, it follows that
\begin{equation}
    \label{eq:estimate_dummy}
    \begin{split}
    \log\phi(x) & \geq n\log x -\log n! -\sum_{m=2}^n\sum_{j=2}^m \log f_-'(f_-^{-j}(x))  + \mathcal{O}(\log x) \\
     & = \frac{\log^2 x}{\log \lambda} - \frac{\log x \cdot \log\log \frac{1}{x}}{\log \lambda} -\sum_{m=2}^n\sum_{j=2}^m \log f_-'(f_-^{-j}(x))  + \mathcal{O}(\log x),
    \end{split}
\end{equation}
 where the last equality follows from \eqref{eq:asymptotic_factorial} and\eqref{eq:scalinglaw_n}.

 From \eqref{eq:estimates_f-}, for all $j=2,\ldots, m$, the following estimate holds: 
 \[
    \log f_-'(f_-^{-j}(x))\leq \log\lambda + \log\left(1+ \frac{2\alpha_2}{\lambda} f_-^{-j}(x)\right).
 \]
This implies that
\begin{align*}
   \sum_{m=2}^n\sum_{j=2}^m \log f_-'(f_-^{-j}(x)) & \leq \frac{n(n-1)}{2}\log\lambda  + \sum_{m=2}^n \sum_{j=2}^m \log\left(1+ \frac{2\alpha_2}{\lambda} f_-^{-j}(x)\right)
   \\
    & \leq \frac{\log^2 x}{2\log \lambda} + \sum_{m=2}^n \sum_{j=2}^m \log\left(1+ \frac{2\vert \alpha_2\vert}{\lambda} f_-^{-j}(x)\right) + \mathcal{O}(\log x),
\end{align*}
since $n = \log x / \log \lambda + \mathcal{O}(\log x)$ in (\ref{eq:scalinglaw_n}). We show that the double sum above is $\mathcal{O}(\log x)$. 
For simplicity, we denote $\tilde{\alpha}:=2\vert\alpha_2 \vert/\lambda$, and $\Delta f_-^{-j}\equiv f_-^{-j-1}(x)-f_-^{-j}(x)$. Then, 
\begin{align*}
\sum_{j=2}^{m} \log \left(1+ \tilde{\alpha} f_-^{-j}(x)\right) & = \sum_{j=2}^{m} \frac{\log \left(1+ \tilde{\alpha} f_-^{-j}(x)\right)}{f_-^{-1}(f_-^{-j}(x)) -f_-^{-j}(x)} \Delta f_-^{-j} \leq 
\sum_{j=2}^{m} \frac{\log \left(1+ \tilde{\alpha} f_-^{-j}(x)\right)}{(\frac{1}{\lambda}-1)f_-^{-j}(x) + \beta_1 f_-^{-j}(x)^2} \Delta f_-^{-j}\\ 
& \leq 
\tilde{\alpha}\lambda \sum_{j=2}^m \frac{1}{1-\lambda- \lambda\vert \beta_1 \vert f_-^{-j}(x)} \Delta f_-^{-j},
\end{align*}
where we used in the last inequality that $\log(1+z) \leq z$, and the denominator is positive whenever $x_0$ is sufficiently small. By estimating the sum with a Riemann integral, we obtain that 
\begin{equation}
\label{eq:logestimate_step}
\begin{split}
\sum_{j=2}^{m} \log \left(1+ \tilde{\alpha} f_-^{-j}(x)\right) & \leq 
\tilde{\alpha} \lambda\int_{f_-^{-2}(x)}^{f_-^{-m-1}(x)} \frac{ds}{1-\lambda -\lambda\vert \beta_1\vert s} \\
& = \frac{\tilde{\alpha}}{\vert\beta_1 \vert} \left[ \log \left(1-\frac{\lambda \vert \beta_1\vert}{1-\lambda} f_-^{-2}(x)\right) - \log \left(1-\frac{\lambda \vert \beta_1\vert}{1-\lambda} f_-^{-m-1}(x) \right) \right].
\end{split}
\end{equation}

Therefore, 
\[
\sum_{m=2}^n\sum_{j=2}^{m} \log \left(1+ \tilde{\alpha} f_-^{-j}(x)\right) \leq A \left[  (n-1)\log \left(  1 - B f_-^{-1}(x)\right)- \sum_{m=2}^n \log \left( 1 - Bf_-^{-m-1}(x) \right) \right],
\]
where $A=\tilde{\alpha}/\vert\beta_1\vert$ and $B=\lambda\vert\beta_1\vert/(1-\lambda)$. Since $\log(1-Bf_-^{-m}(x))\geq \log(1- B f_-^2(x_0))$ for all $m\leq n= n_{x_0}^x$, then
\begin{align*}
\sum_{m=2}^n\sum_{j=2}^{m} \log \left(1+ \tilde{\alpha} f_-^{-j}(x)\right) & \leq A(n-1)\left[ \log \left(  1-B f_-^{-1}(x)\right) - \log(1- B f_-(x_0))\right] = \mathcal{O}(\log x),
\end{align*}
since $\log(1+Bf_-^{-1}(x)) = \mathcal{O}(x)$, and $n=\mathcal O(\log x)$.

Finally, by using the last inequality in \eqref{eq:estimate_dummy}, we obtain
\begin{equation}
    \label{eq:lower_bound}
    \log\phi(x) \geq \frac{\log^2 x}{2\log \lambda} - \frac{\log x \cdot \log \log\frac{1}{x}}{\log \lambda} + \mathcal{O}(\log x),
\end{equation}
and the result follows as the combination of \eqref{eq:upper_bound} and \eqref{eq:lower_bound}.
\end{proof}

\section{Optimum hyperparameters}\label{appen:optimum}

When approximating $\lambda$ using Algorithm~\ref{alg:tailfit}, three hyperparameters come into play, namely
\begin{enumerate}
    \item $n\in\mathbb{N}$, the total number of data points,
    \item $b< n$, the number of bins used in the histogram, and
    \item $q\in(0,1)$, the {quantile} used in the estimation of $\lambda$.
\end{enumerate}
A hyperparameter is said to be \emph{optimal} if it minimises the \emph{root mean square error (RMSE)},

\[\textup{RMSE}=\sqrt{\frac{1}{m}\sum_{i = 1}^m(\lambda_i-\hat{\lambda}_i)^2},\] 
over $m$ different parameter values $(\lambda_i)_{1\leq i\leq m}$, while holding the other hyperparameters constant.
Determining the optimal values of $b$ and $q$ for a given sample size $n$ is crucial for applying the proposed method effectively. To illustrate this, we focus on the case where the left boundary point $x_-$ is unknown and is estimated as described in Subsection~\ref{sec:unknown}. 

\subsection{Exploring different values of $b$ and $q$}
\label{sec:optimum}

We perform a case study on the nonlinear map (\ref{eq:non-linear_eq}) with uniform noise in $[-0.1,0.1]$ for $-0.5\leq a \leq 0.31$, i.e. before the bifurcation point. Following Algorithm~\ref{alg:tailfit}, we apply both the leading-order method (\ref{eq:quad_method}) and the higher-order method (\ref{eq:quad_method_strict}). Our first objective is to investigate the optimal number of histogram bins $b\in [20,500]$ for estimating $\lambda$ with $n = 10^5$ sample points generated from (\ref{eq:non-linear_eq}).
We consider three fixed quantiles, $q = 0.1, 0.35$, and $0.6$. For each pair $(b,q)$,
we compute $\hat{\lambda}$ and evaluate the resulting RMSE as a function of $b$ (see Figure~\ref{fig:optimum_finding}(a)). The dependence on $b$ differs greatly between the two methods:
for the higher-order method (\ref{eq:quad_method_strict}), small $b$ values yield the lowest RMSE, wheareas for the leading-order method (\ref{eq:quad_method}) the RMSE remains low across a wide range $b \in [80,500]$.

To explore the dependence on $q$, we fix $b=100, 200, 300$ and compute the RMSE of $\hat{\lambda}$ for $q \in [0.01, 0.8]$ (see Figure~\ref{fig:optimum_finding}(b)). For the leading-order method (\ref{eq:quad_method}), the RMSE is consistently small for $0.5 < q < 0.8$, while for the higher-order method (\ref{eq:quad_method_strict}) there is a minimum at $q \approx 0.2$. The resulting optimal choices are $(b,q) \approx (150,0.7)$ for the leading-order method and $(b,q) \approx (20,0.2)$ for the higher-order method.

We next consider the case where the noise distribution is truncated normal and repeat the numerical experiments. The results are shown in Figure~\ref{fig:optimum_finding_truncated}. In this setting, the RMSE as a function of $b$ in Figure~\ref{fig:optimum_finding_truncated}(a) exhibits a trend similar to the uniform noise case in Figure~\ref{fig:optimum_finding}(a). Likewise, the dependence of RMSE on $q$ for the higher-order method (\ref{eq:quad_method_strict}) in Figure~\ref{fig:optimum_finding_truncated}(b) follows the trend of the uniform noise results in Figure~\ref{fig:optimum_finding}(b), with a minimum near $q \approx 0.4$. In contrast, for the leading-order method (\ref{eq:quad_method}), the RMSE reaches its minimum around $q \approx 0.01$ in Figure~\ref{fig:optimum_finding_truncated}(b), which differs significantly from the uniform noise, where the optimum occurs at large $q$ in Figure~\ref{fig:optimum_finding}(b). The corresponding optimal hyperparameters  are $(b,q) \approx (>500,0.01)$ for the leading-order method, and $(b,q) \approx (20, 0.4)$ for the higher-order method.

We also examine the case where the true boundary $x_-$ is known. The RMSE across different hyperparameter values $b$ and $q$ is presented in Figure~\ref{fig:optimum_finding_true} for uniform noise and Figure~\ref{fig:optimum_finding_truncated_true} for truncated normal noise. 

\begin{figure}[!ht]
\begin{overpic}[width=.49\textwidth]{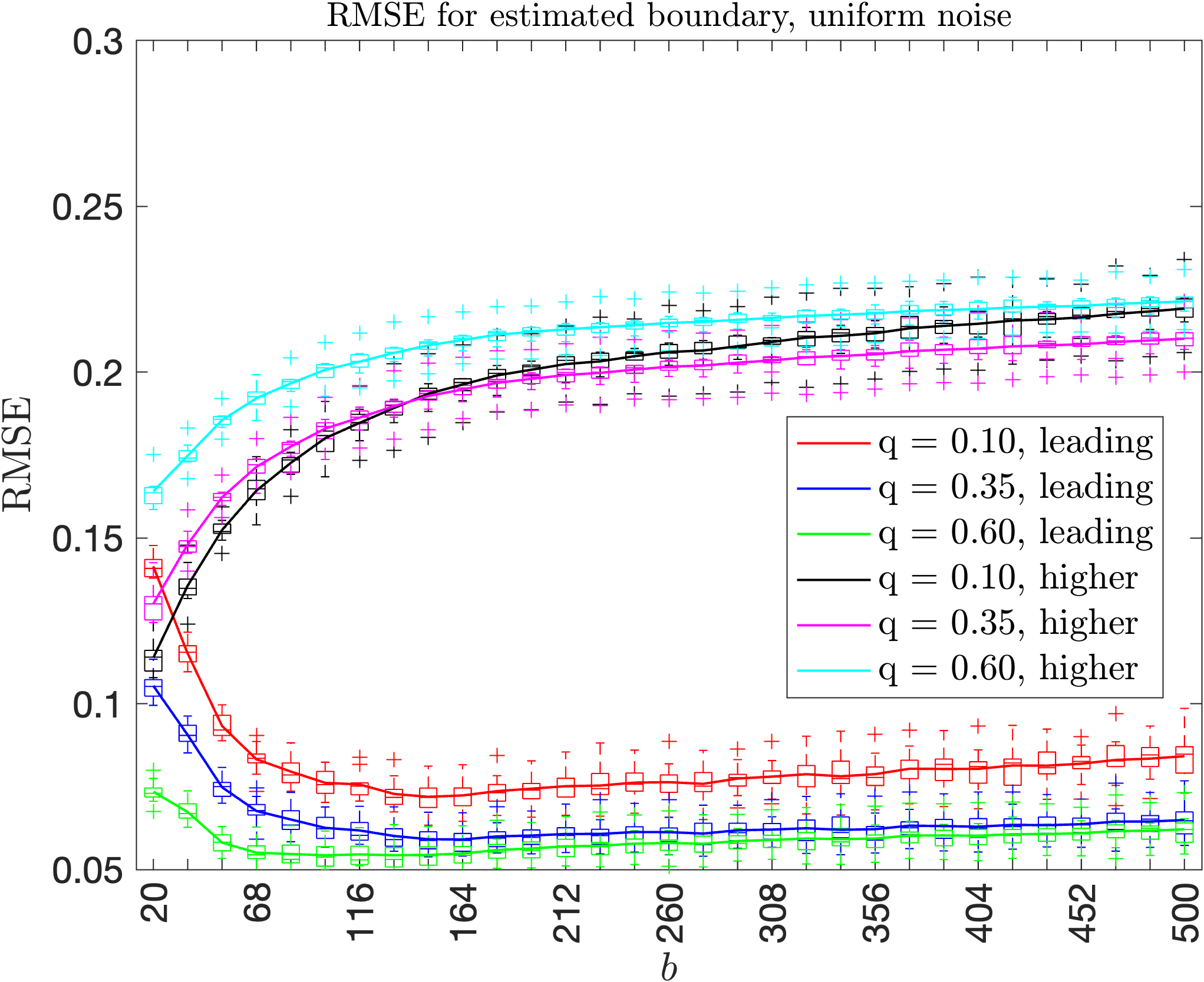}
    \put(0,0){
    (a)
    }
    \end{overpic}
    \begin{overpic}[width=.49\textwidth]{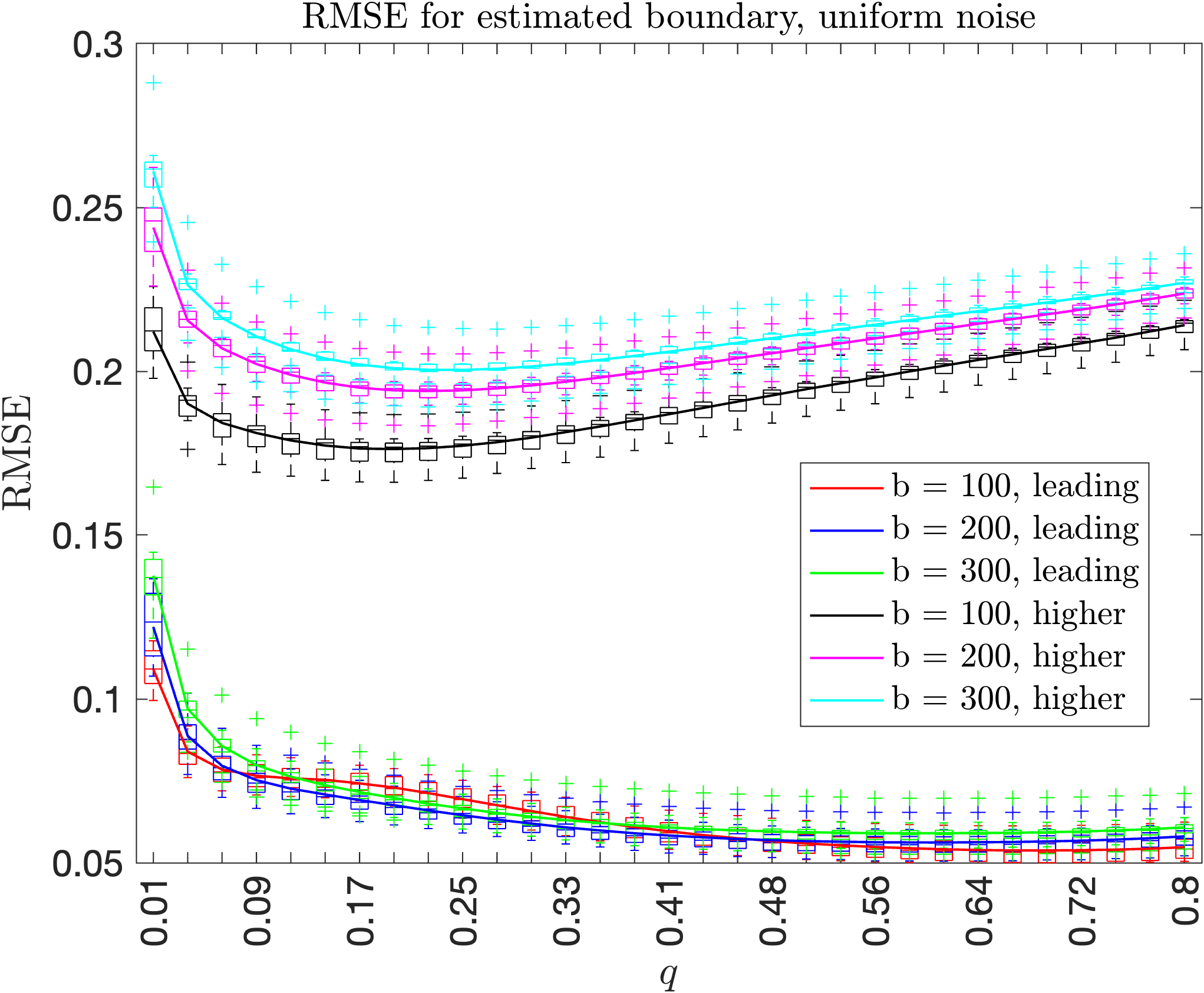}
    \put(0,0){
    (b)
    }
    \end{overpic}
\caption{Computation of root mean square errors (RMSE) of $\lambda$ approximations with $n = 10^5$ of the nonlinear map (\ref{eq:non-linear_eq}) with uniform noise $\xi_t \in [-0.1,0.1]$ for 10 distinct noise realisations, using both methods (\ref{eq:quad_method}) and (\ref{eq:quad_method_strict}) to find 1) the optimum number of histogram bins $b$ while fixing $q=0.1, 0.35$ or $0.6$ quantile in Figure~\ref{fig:optimum_finding}(a) and 2) the optimal $q$-quantile of data while fixing $b = 100, 200$ or $300$ in (b).}
\label{fig:optimum_finding}

\begin{overpic}[width=.49\textwidth]{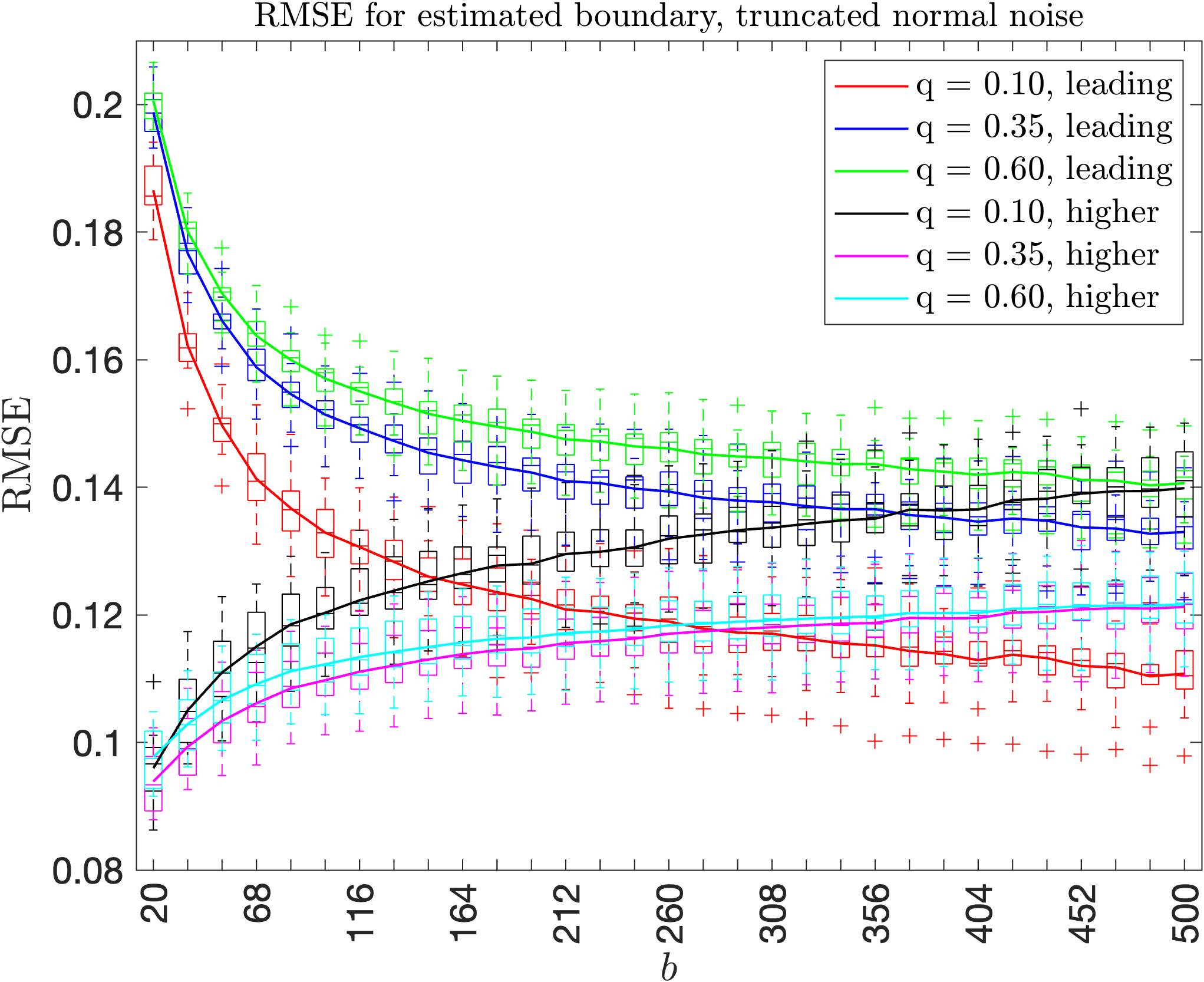}
    \put(0,0){
    (a)
    }
    \end{overpic}
    \begin{overpic}[width=.49\textwidth]{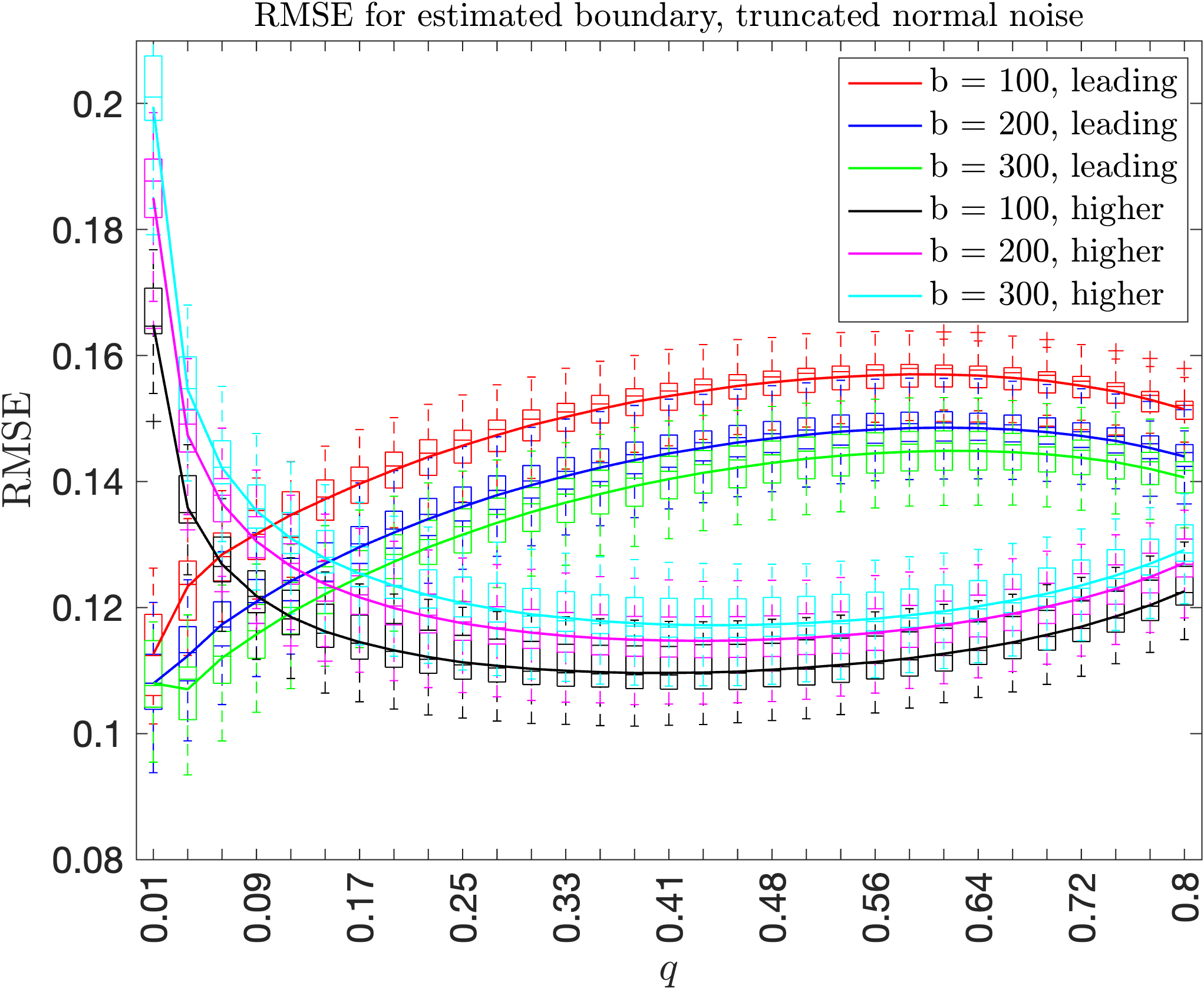}
    \put(0,0){
    (b)
    }
    \end{overpic}
\caption{RMSE of $\lambda$ approximations with $n = 10^5$ of the nonlinear map (\ref{eq:non-linear_eq}) with truncated normal noise $\xi_t \in [-0.1,0.1]$ following similar procedures to Figure~\ref{fig:optimum_finding}.}
\label{fig:optimum_finding_truncated}
\end{figure}

\begin{figure}[!h]
\begin{overpic}[width=.49\textwidth]{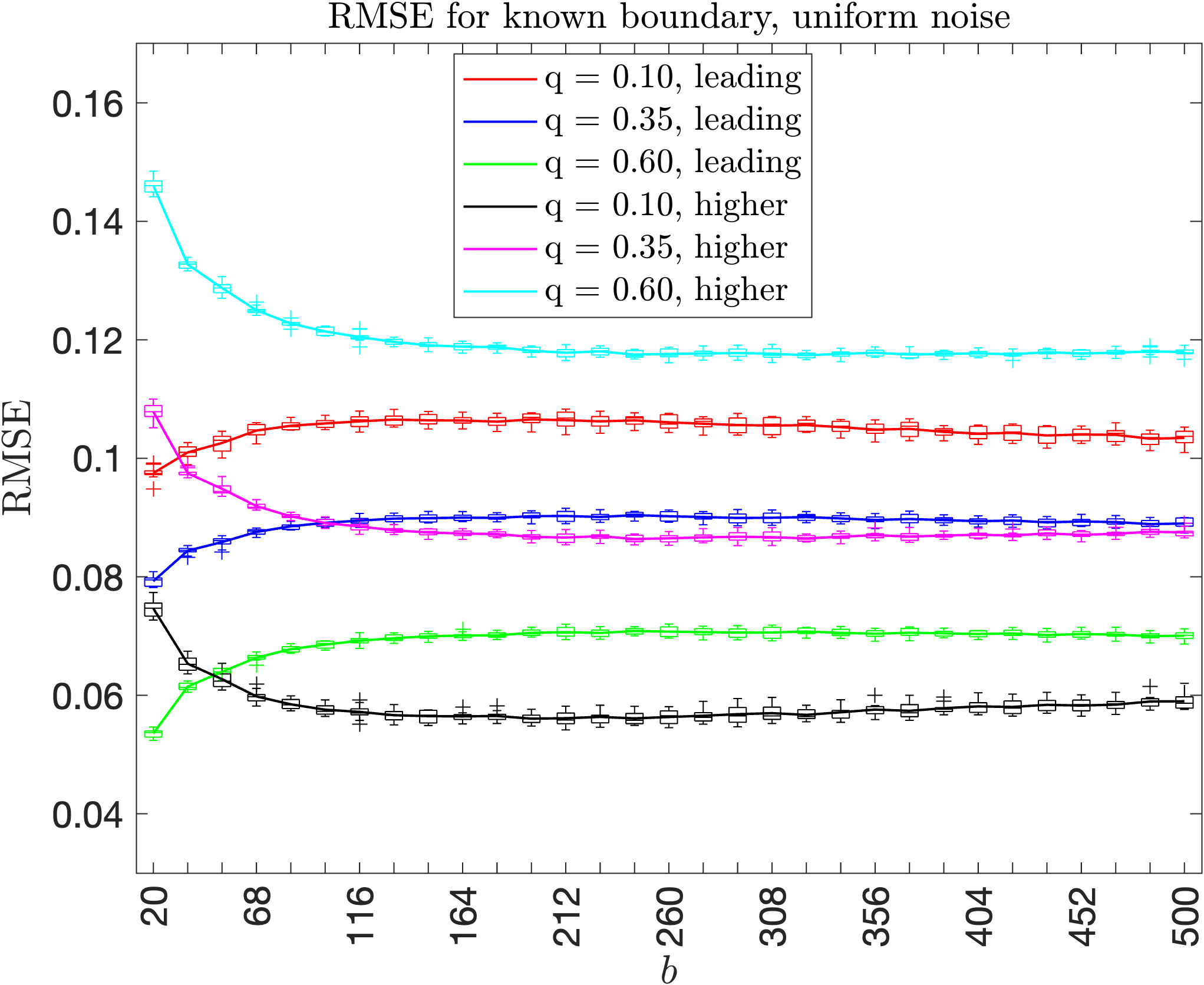}
\put(0,0){
(a)
}
\end{overpic}
\begin{overpic}[width=.49\textwidth]{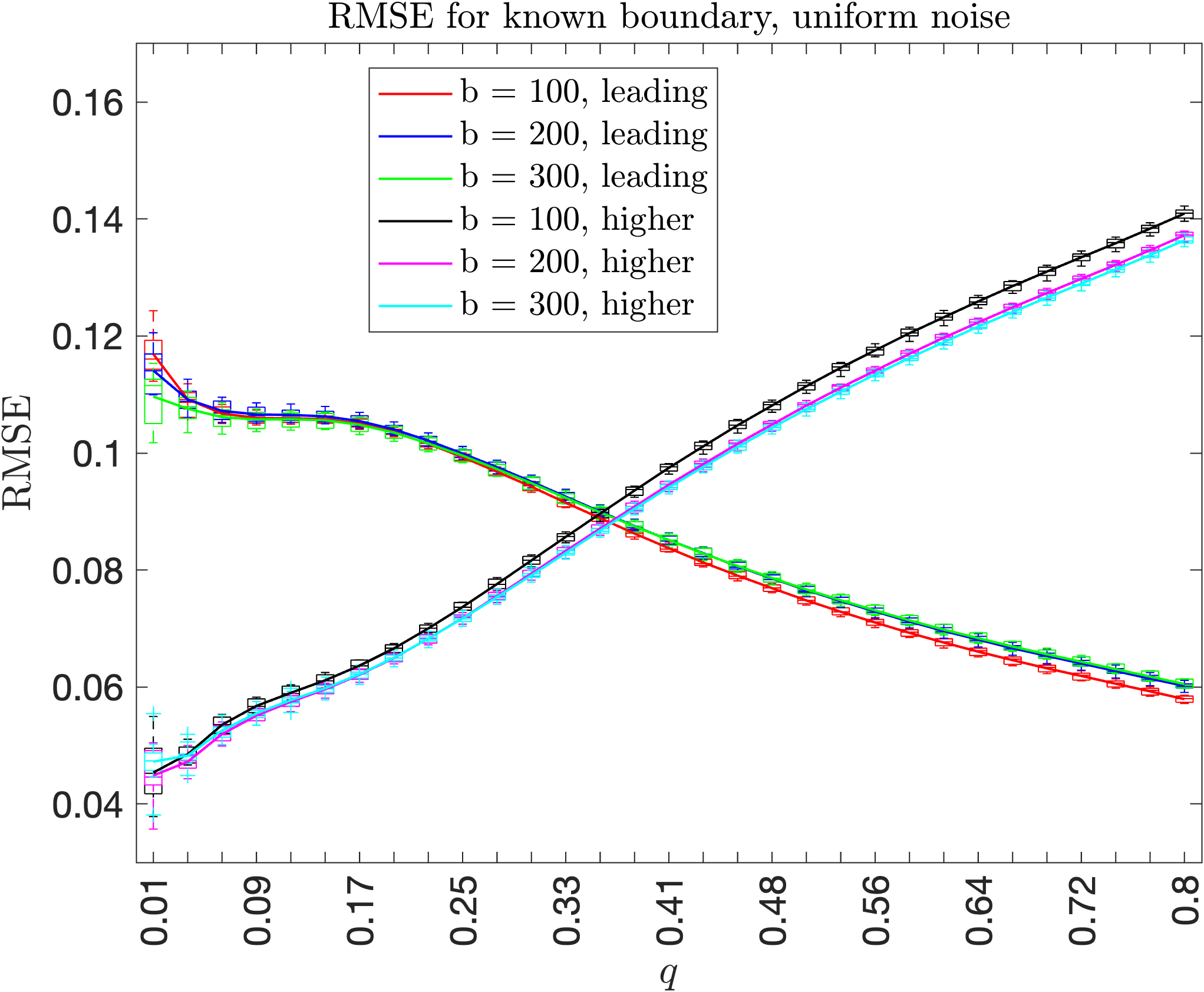}
\put(0,0){
(b)
}
\end{overpic}
\caption{RMSE of $\lambda$ approximations with $n = 10^5$ of the nonlinear map (\ref{eq:non-linear_eq}) with uniform noise $\xi_t \in [-0.1,0.1]$ for 10 distinct noise realisations, using both methods (\ref{eq:quad_method}) and (\ref{eq:quad_method_strict}) for different values of $b$ and $q$, where true boundary is used, i.e. $\hat{x}_- = x_-$. }
\label{fig:optimum_finding_true}

\begin{overpic}[width=.49\textwidth]{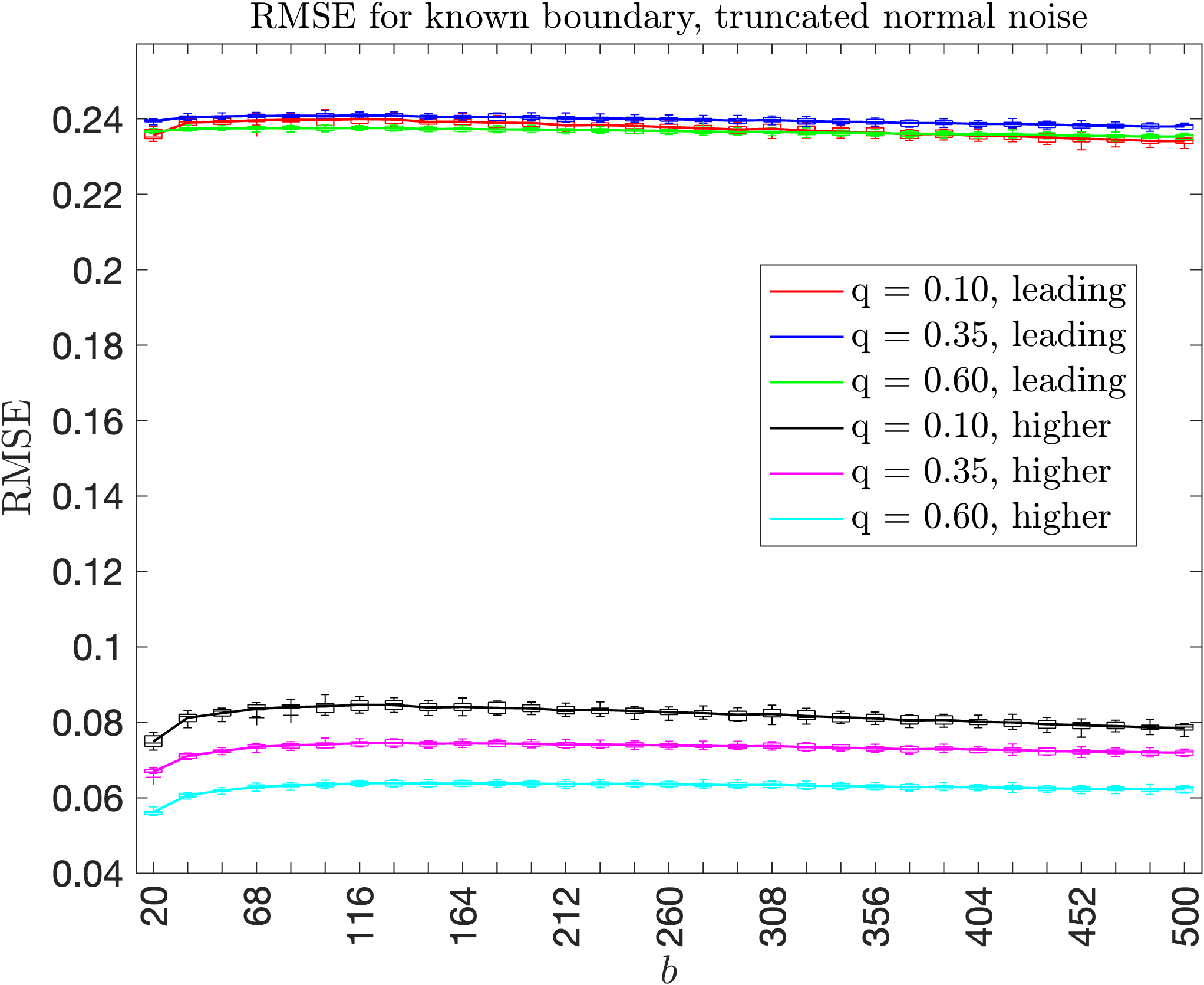}
\put(0,0){
(a)
}
\end{overpic}
\begin{overpic}[width=.49\textwidth]{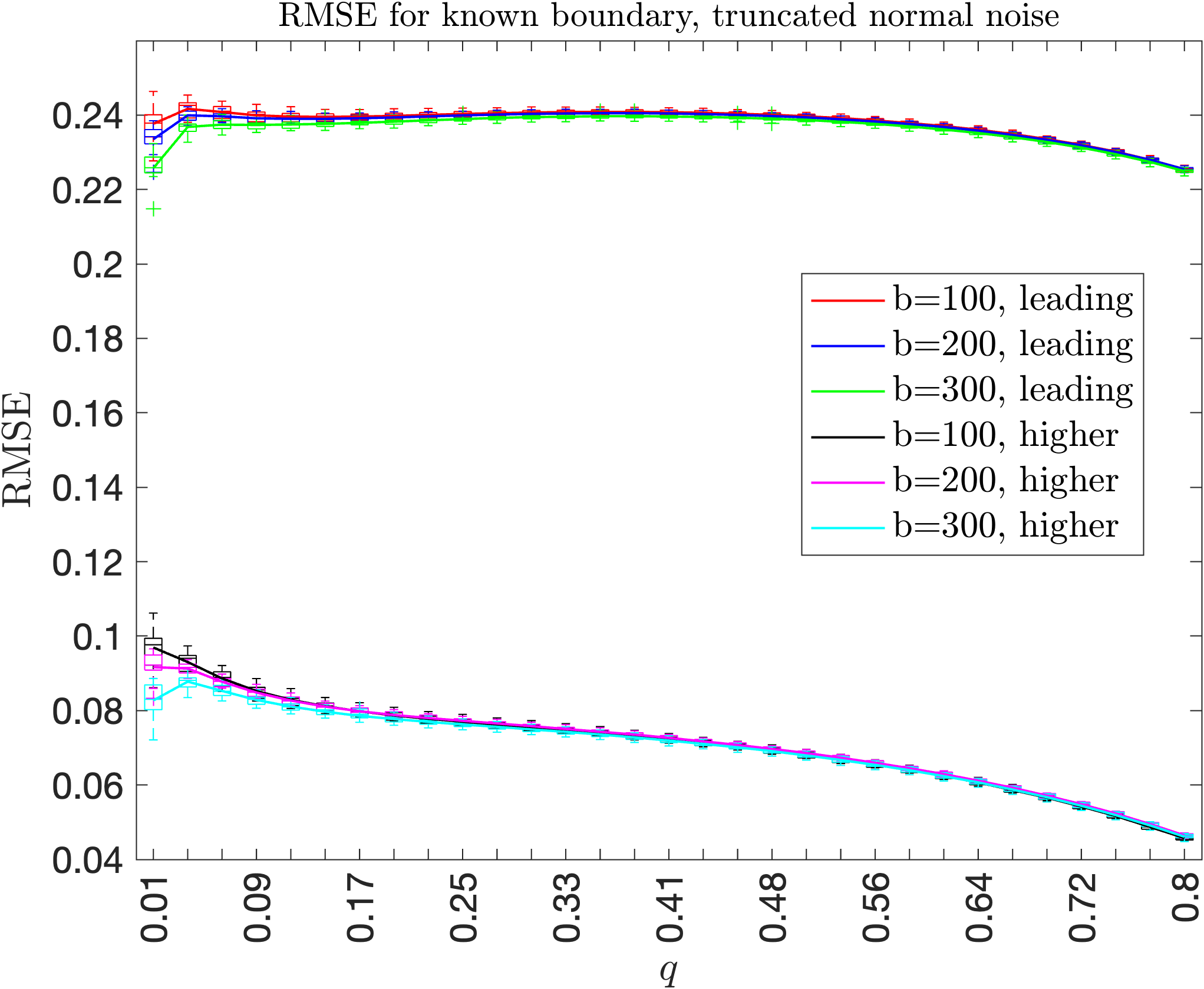}
\put(0,0){
(b)
}
\end{overpic}
\caption{RMSE of $\lambda$ approximations with $n = 10^5$ of the nonlinear map (\ref{eq:non-linear_eq}) with truncated normal noise $\xi_t \in [-0.1,0.1]$ following similar procedures to Figure~\ref{fig:optimum_finding_true}.}
\label{fig:optimum_finding_truncated_true}
\end{figure}

\begin{figure}[b!]
\begin{overpic}[width=.49\textwidth]{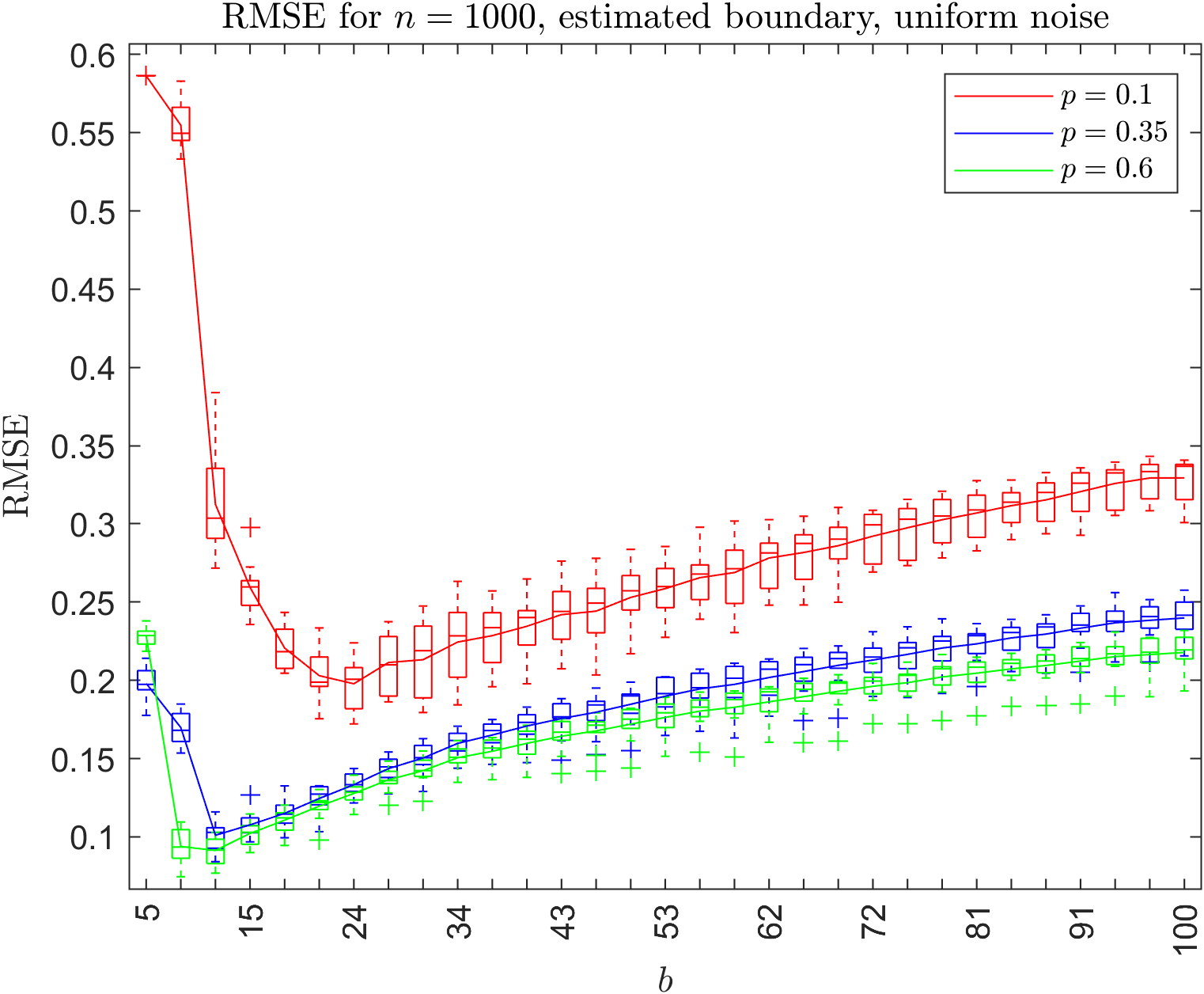}
\put(0,0){
(a)
}
\end{overpic}
\begin{overpic}[width=.49\textwidth]{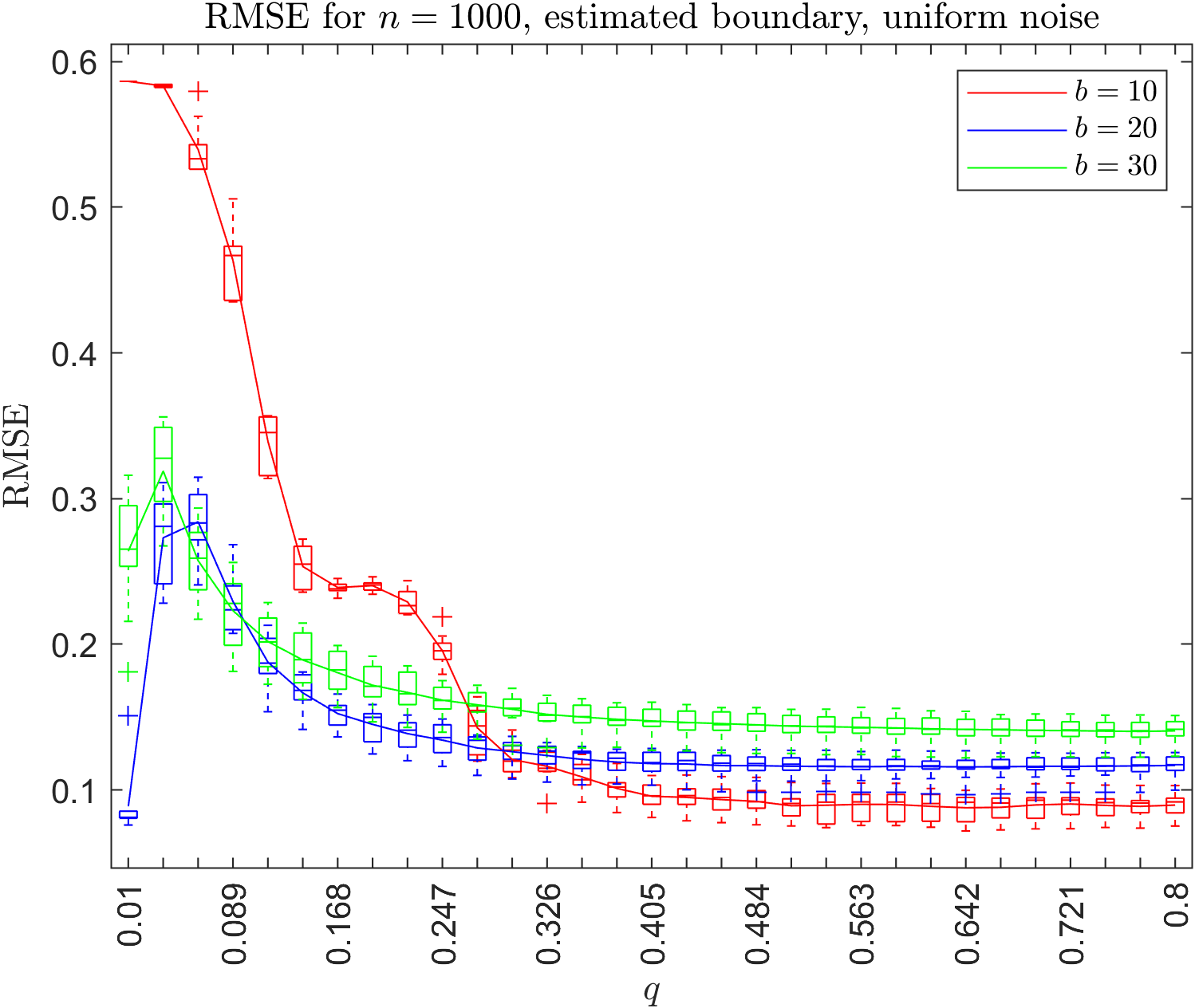}
\put(0,0){
(b)
}
\end{overpic}
\caption{RMSE computations for $n = 1000$ data points for the nonlinear map with uniform noise in (\ref{eq:non-linear_eq}), for various hyperparameter values $b\in [5,100]$ and $q\in[0.01,0.8]$, using the leading-order method (\ref{eq:quad_method}).}
\label{fig:optimum_1000}
\end{figure}

For uniform noise (Figure~\ref{fig:optimum_finding_true}(a)), small $b$ values yield smaller RMSE for the leading-order method, while the higher-order method achieves its minimum at $b \approx 250$. This pattern is reversed when the boundary is estimated, as in Figure~\ref{fig:optimum_finding}(a). Regarding the dependence of $q$ in Figure~\ref{fig:optimum_finding_true}(b), the leading-order method improves as $q$ increases, while the higher-order method worsens. For $q<0.35$, the higher-order method performs better, while for $q>0.35$ the leading-order method dominates. These results highlight the strong influence of $q$ on accuracy compared to $b$.

For truncated normal noise with known boundary, the choice of $b$ has little impact on the accuracy
as evident from Figure~\ref{fig:optimum_finding_truncated_true}(a). In contrast, increasing $q$ improves performance for both methods
as illustrated in Figure~\ref{fig:optimum_finding_truncated_true}(b). From both figures, it is evident that the higher-order method consistently outperforms the leading-order method across $b \in [20,500]$ and $q\in [0.01,0.8]$.

Overall, the results demonstrate that the optimal hyperparameters depend not only on the noise distributions but also on the chosen method and whether the true boundary is known. In practical settings where only the time series is available and the underlying system is unknown, we recommend choosing moderate values, such as $b = 200$ and $q = 0.3$ (used in the main text), provided that sufficiently large samples are available (e.g. $n = 10^5$). We discuss next on the performance of the proposed method when $n$ is small.

\begin{figure}[b!]
\begin{overpic}[width=.49\textwidth]{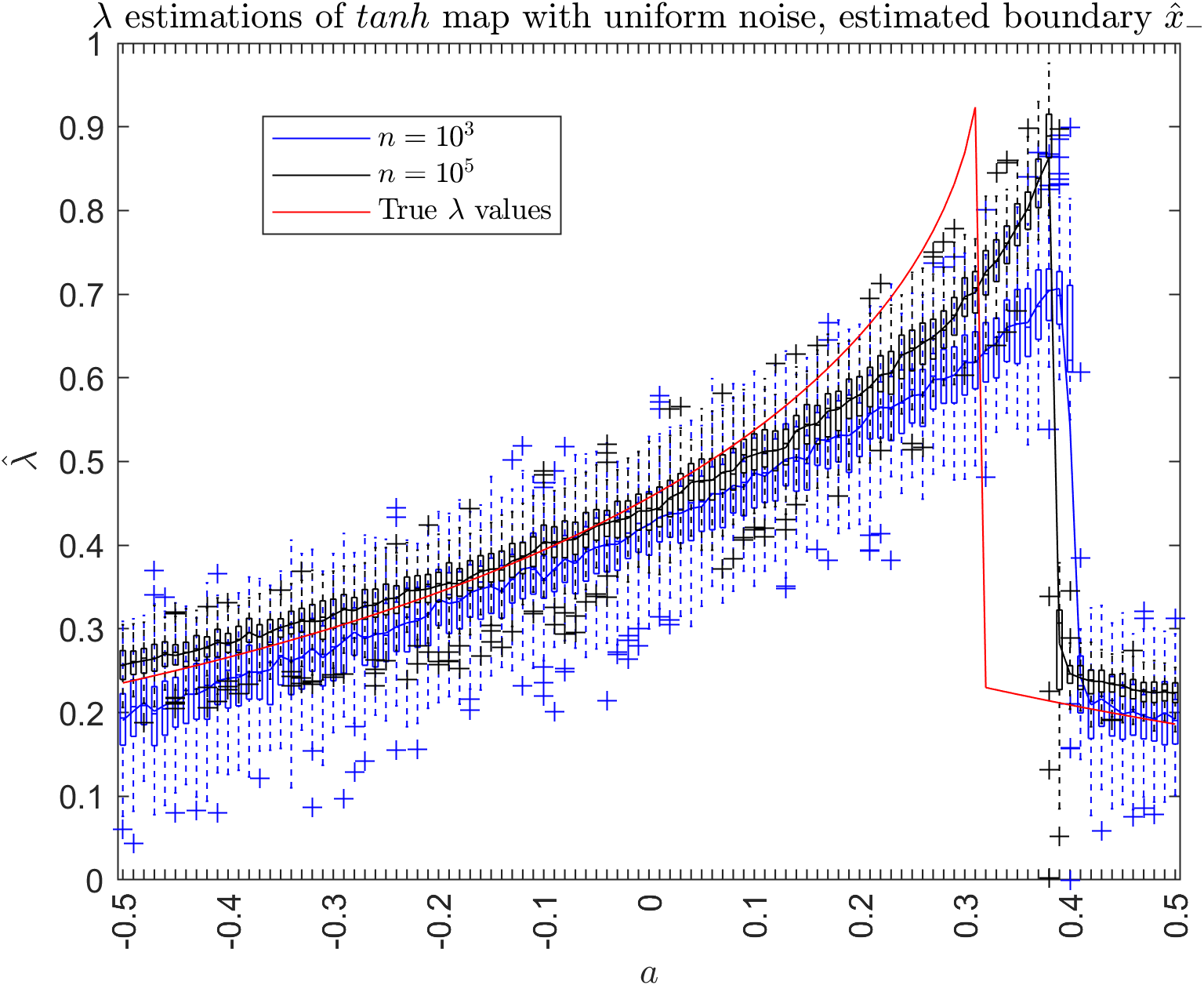}
\put(0,0){
(a)
}
\end{overpic}
\begin{overpic}[width=.49\textwidth]{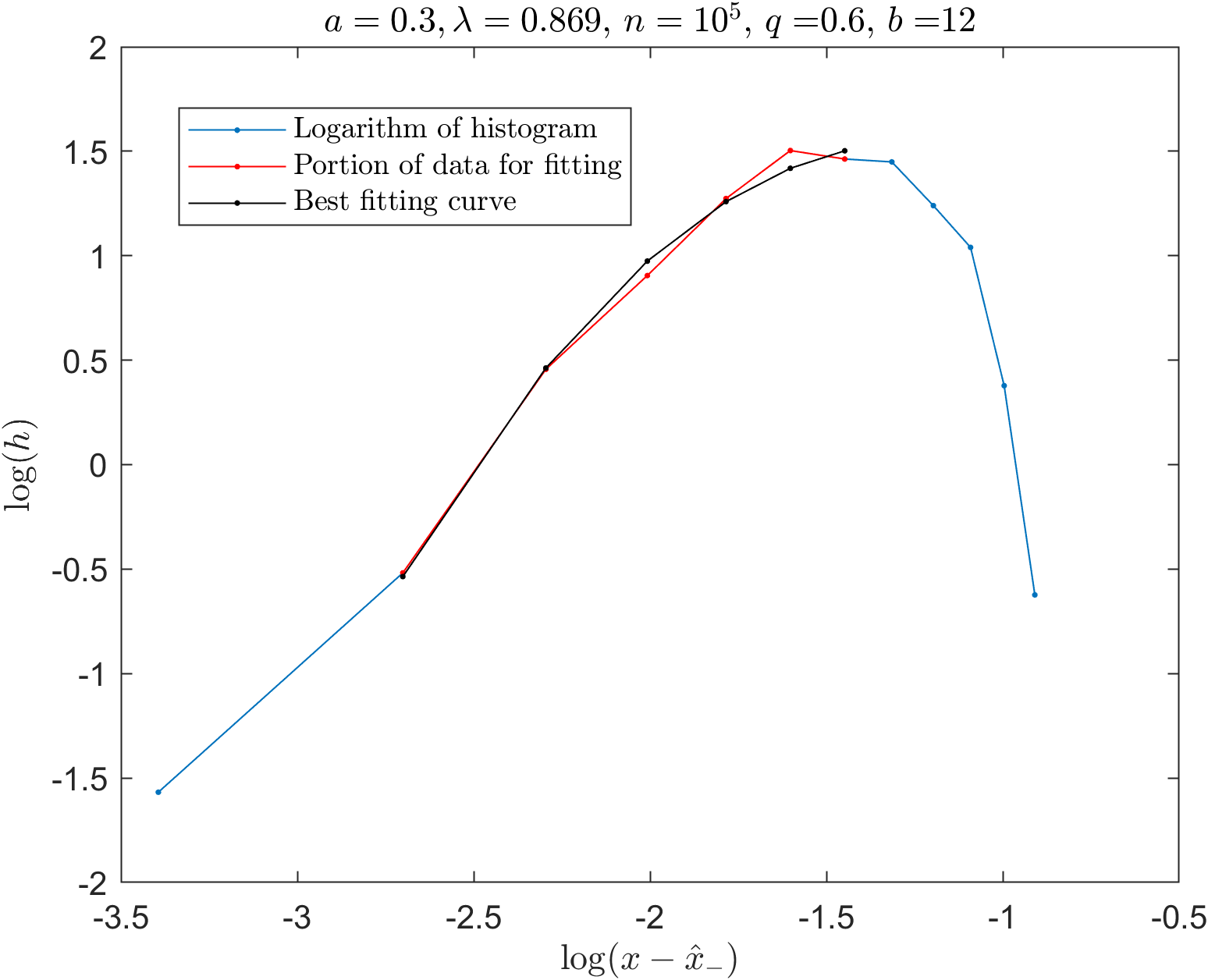}
\put(0,0){
(b)
}
\end{overpic}
\caption{Estimations of $\lambda$ using leading-order method (\ref{eq:quad_method}) following Algorithm~\ref{alg:tailfit} for $n = 1000$ using $b = 12,$ and $q = 0.6$ illustrated as blue box plot in (a), together with the results for $n = 10^5$ with $b = 200,$ and $q = 0.3$ as black box plot. The optimisation method (\ref{eq:quad_method}) for $n=1000$ at $a = 0.3$ only uses 5 histogram data points out of the $12$ total histogram bins as illustrated in (b) plotted in the logarithm scale.}
\label{fig:data1000}
\end{figure}

\subsection{Lower number of data points}\label{sec:low_number}

The observation of large optimal values of $q$ suggests that $\lambda$ can still be well approximated even when a large quantile of the data is chosen. To assess robustness, we validate that the approximations obtained with Algorithm~\ref{alg:tailfit} remain accurate even when a limited number of data $n$ are available and explore the limit of $n$ when this fails. 

We reduce the number of data points to $n = 10^3$ for the nonlinear system (\ref{eq:non-linear_eq}) with uniform noise. The optimal hyperparameters for the leading-order method (\ref{eq:quad_method}) are found to be $b = 12$ and $q = 0.6$ (see Figure~\ref{fig:optimum_1000}). Using these hyperparamters, the resulting $\lambda$ approximations are shown in Figure~\ref{fig:data1000}(a). They closely resemble those obtained for $n =10^5$, which are superimposed for comparison.

Because of the limited sample size,
the number of histogram bins $b$ must also be kept low. In this case, we choose $b = 12$, so only 12 histogram data points $(x_i,h_i)_{1\leq i\leq 12}$ are available. For example, when $a = 0.3$, only 5 histogram data points are used in the optimisation using (\ref{eq:quad_method}), as illustrated in Figure~\ref{fig:data1000}(b).

\begin{figure}[b!]
\begin{overpic}[width=.49\textwidth]{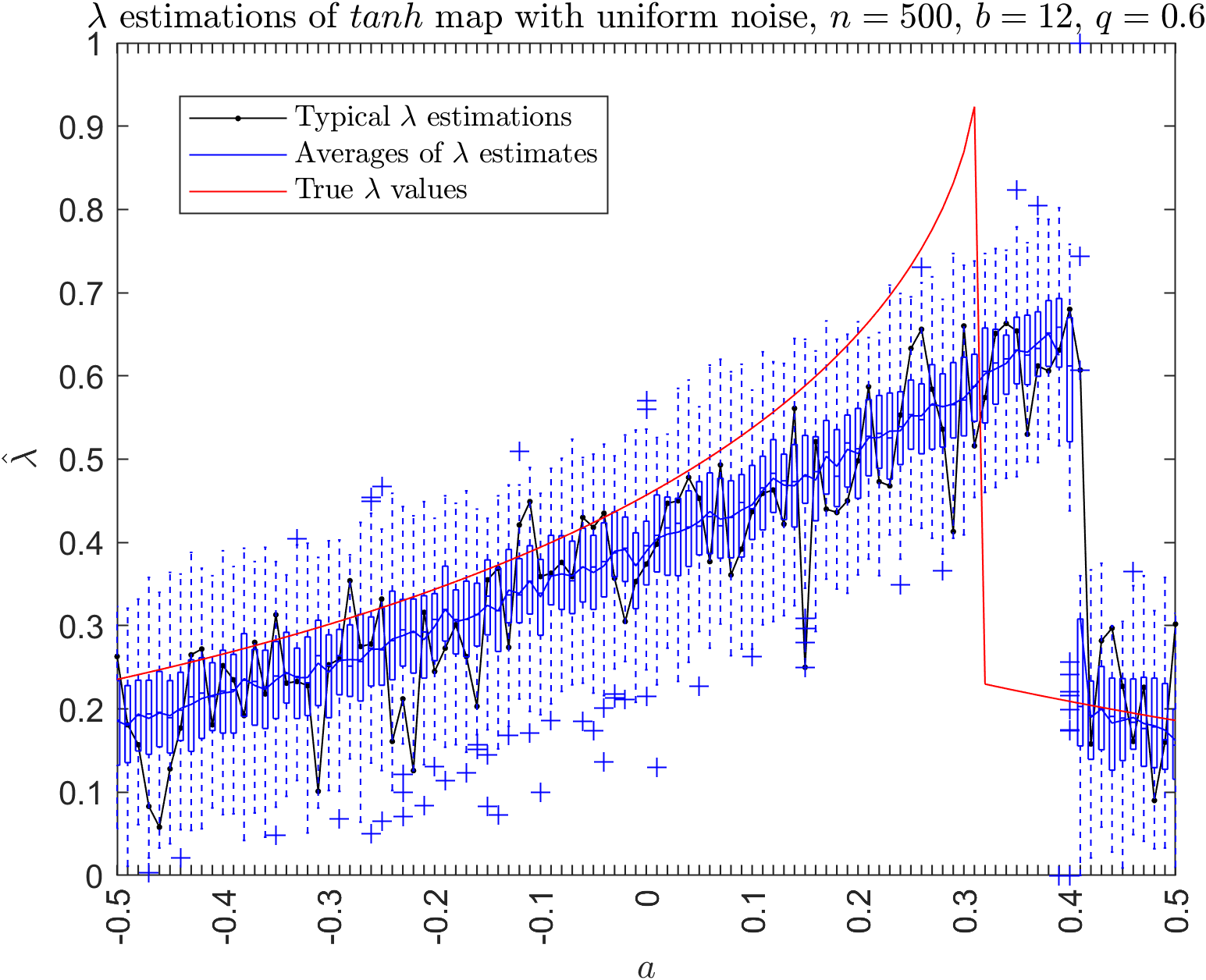}
\put(0,0){
(a)
}
\end{overpic}
\begin{overpic}[width=.49\textwidth]{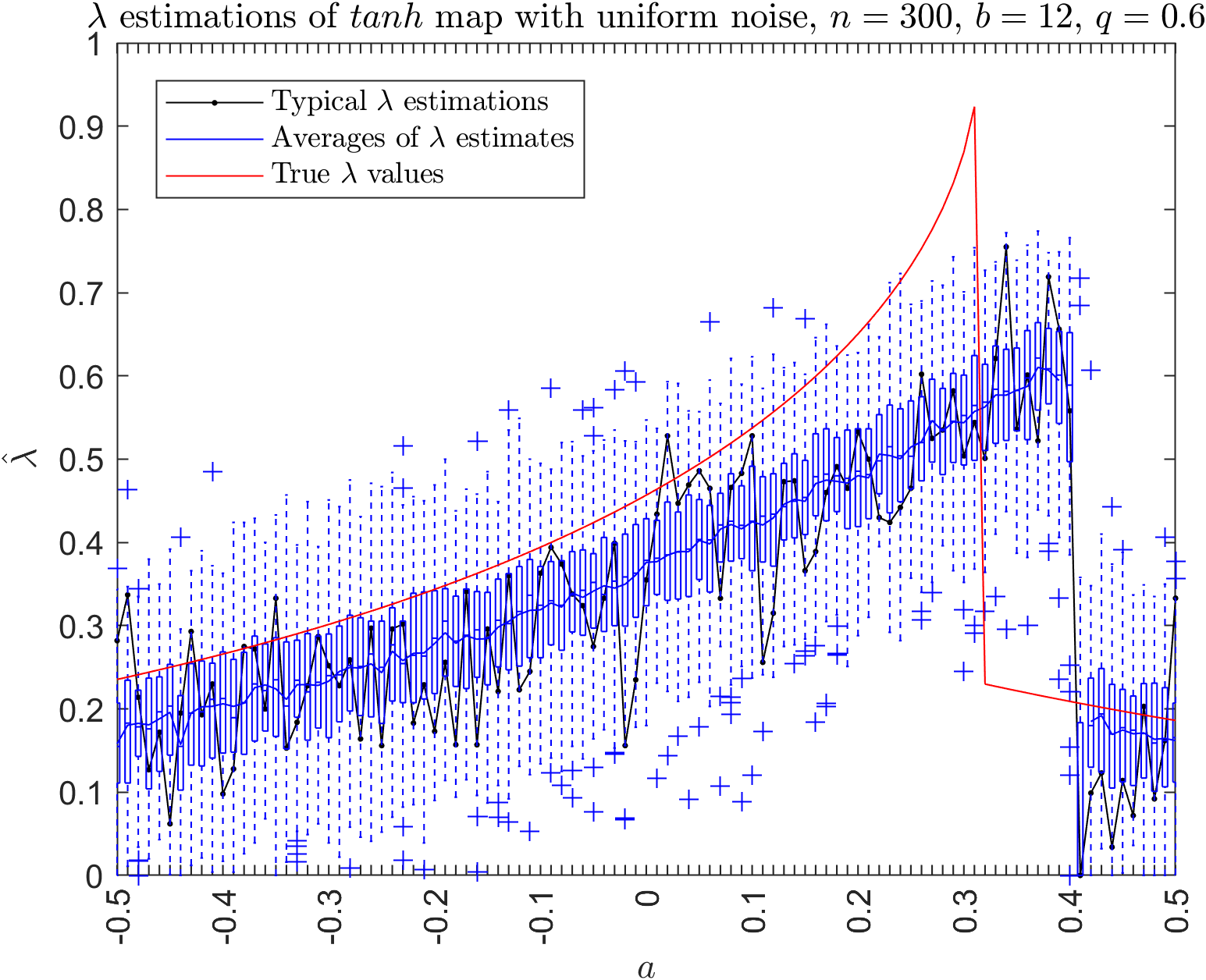}
\put(0,0){
(b)
}
\end{overpic}
\begin{minipage}{.007\textwidth}
    \text{}
\end{minipage}
\vspace{.3em}

\begin{overpic}[width=.49\textwidth]{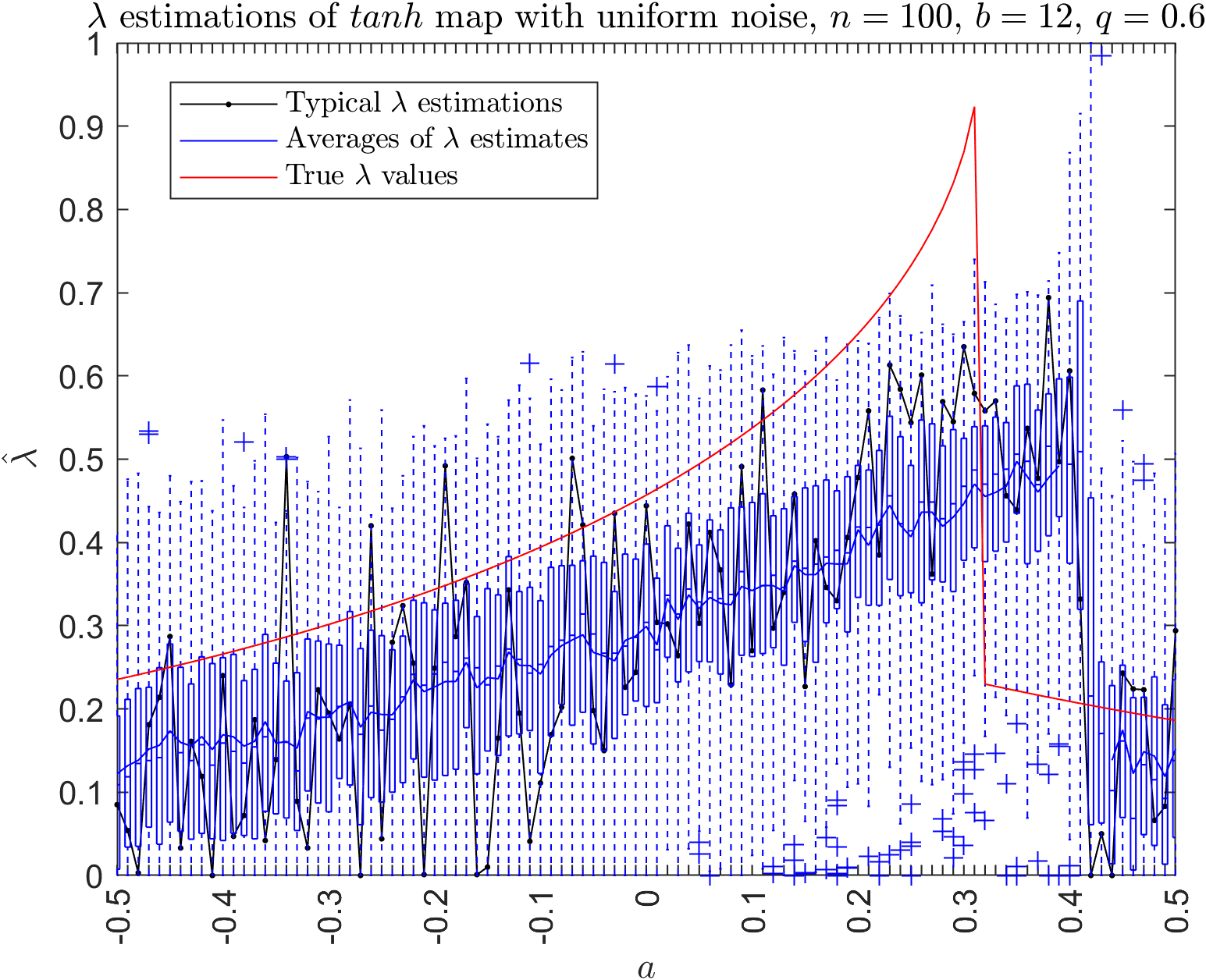}
\put(0,0){
(c)
}
\end{overpic}
\begin{overpic}[width=.49\textwidth]{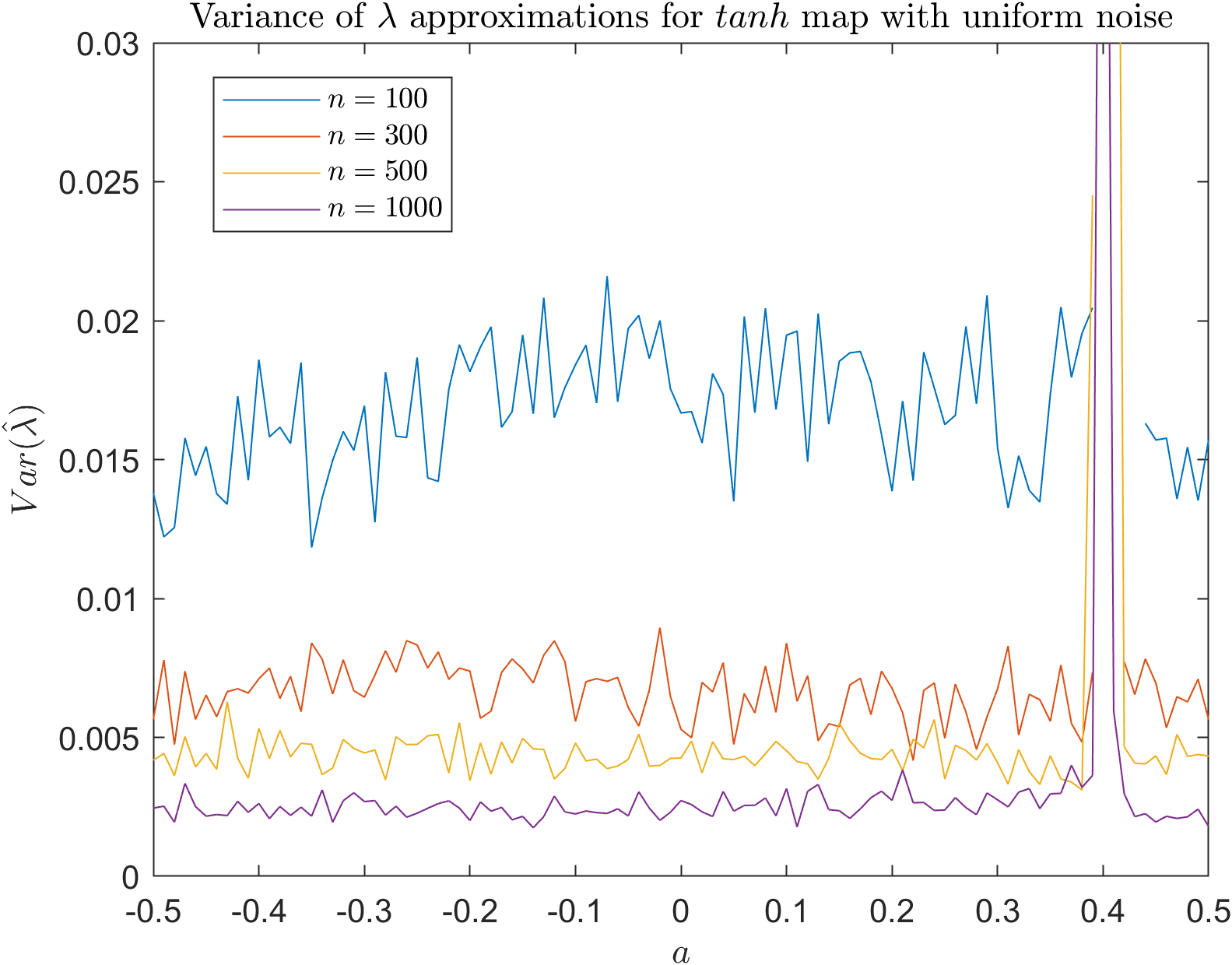}
\put(0,0){
(d)
}
\end{overpic}

\caption{Estimation of $\lambda$ from $n = 500$, $n = 300$ or $n = 100$ data points respectively. All estimations utilise $b = 12$ histogram bins and $q = 0.6$ quantile of data points. The results for 100 different sets of data points are plotted in (a), (b) and (c) respectively. The accuracy of the approximations decreases for a smaller number of data $n$ while the variance of $\lambda$ approximations for different sample data increases significantly as shown in (d).}
\label{fig:data100_300_500}
\end{figure}

To further test the reliability of our approach with small data sizes,
we repeat the experiments with lower values of $n$, namely $n = 500, n = 300$ and $n = 100$, using the same hyperparameters $b = 12$ and $q = 0.6$. The corresponding results are shown in Figures~\ref{fig:data100_300_500}(a), \ref{fig:data100_300_500}(b), and \ref{fig:data100_300_500}(c), where we include a typical $\lambda$ approximation for a specific noise realisation in each case. 
As expected, the accuracy of $\lambda$ estimations decreases with smaller $n$, and the variance of $\hat{\lambda}$ across different noise realisations increases significantly with decreasing $n$, as illustrated in Figure~\ref{fig:data100_300_500}(d).

Nevertheless, even with these low data points, the upward trend in the average estimations, represented by the blue line plot in Figure~\ref{fig:data100_300_500}(c), can still be observed and may provide an early warning signal for a bifurcation.

\section{Discussion on approximation errors}\label{APPENDIX:error_discussion}
\subsection{True boundary versus estimated boundary}
The quadratic function used for fitting in (\ref{eq:quad_method}) has the form $\psi(y)=a_2y^2 + a_1y$ for $y = \log(x-x_-)$, which admits two roots: $0$ and $r:=-a_1/a_2$. The points $x$, from the histogram tail for solving the optimisation (\ref{eq:quad_method}), must lie sufficiently close to the boundary $x_-$ so that 
\begin{equation}\label{eq:opt_cond}
    \log(x-\hat{x}_-) < \log(x-x_-) < 0.
\end{equation}
The coefficients $a_1$ and $a_2$ are determined by the root $r$ and the gradient of $\psi$ at $r$, denoted by $g=f'(r)=-a_1$. On the one hand, $a_2=g/r$, and on the other hand $a_2 = \tfrac{1}{2\log \lambda}$. Combining these gives the estimator $\hat{\lambda}_{\hat{x}_-} = \exp\left(\frac{r}{2g}\right)$.

\begin{figure}[!b]
    \centering
    \includegraphics[width=.5\textwidth]{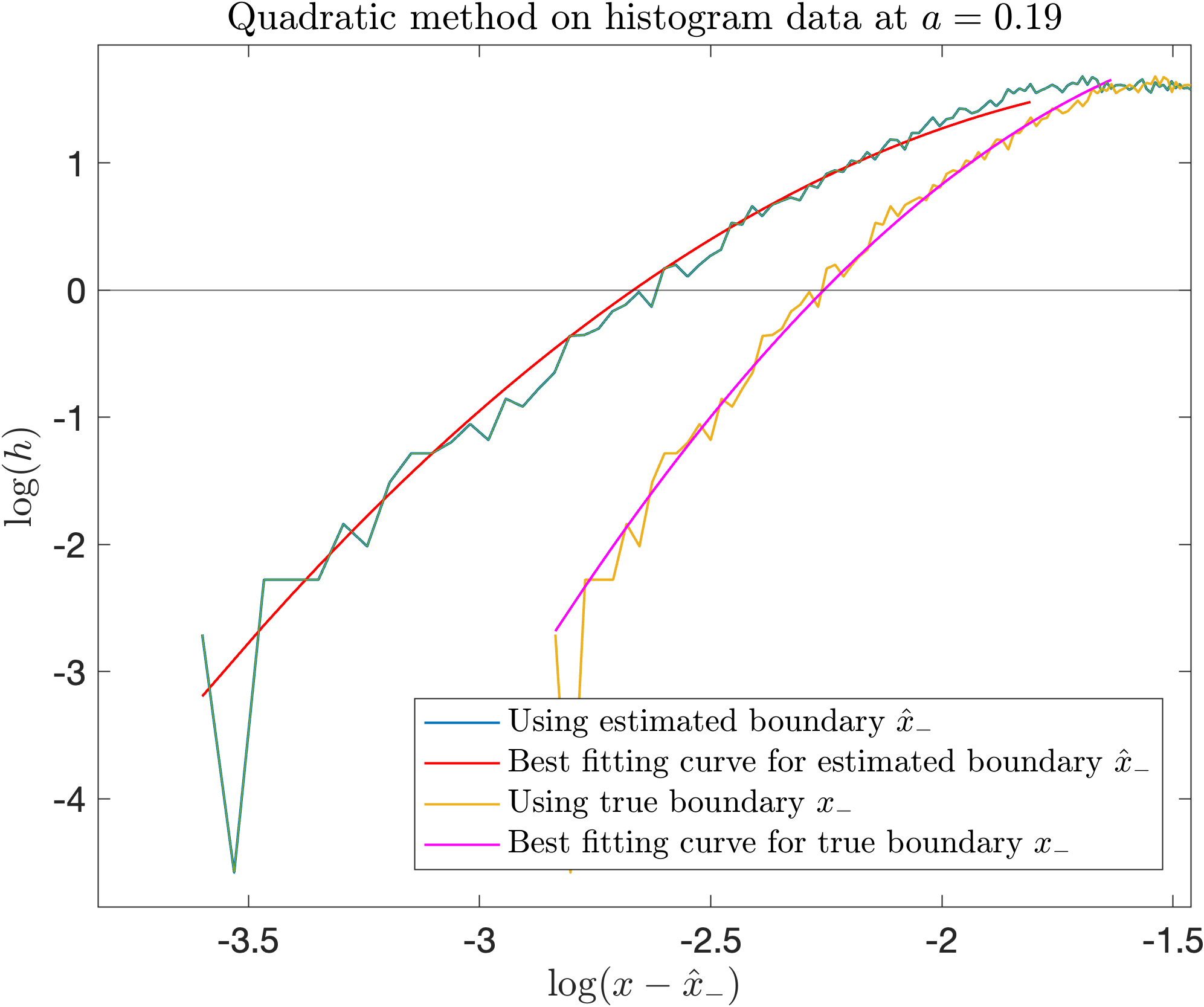}
    \caption{Comparing the best fitting curve where $x_-$ is the true boundary and estimated boundary from histogram data respectively, for $a = 0.19$ in the nonlinear example (\ref{eq:non-linear_eq}) with uniform noise. The negative root of the best-fitting quadratic curve and the gradient at the root both decrease when the boundary is estimated.}
    \label{fig:compare_known_unknown}
\end{figure}

We assume that the best fitting quadratic $\psi$ has a negative root, which holds whenever (\ref{eq:opt_cond}) is satisfied. Consider the same histogram data $(x_i)_{1\leq i\leq b_l}$ used for fitting for the leading-order method (\ref{eq:quad_method}). We distinguish between two cases as in the main text: (i) the true boundary $\hat{x}_- = x_-$ is known, and (ii) an estimated boundary $\hat{x}_- > x_-$. Denote $r(x_-), g(x_-)$ as the negative root and gradient at the root when using the true boundary, and $r(\hat{x}_-), g(\hat{x}_-)$ for the estimated boundary. From Figure~\ref{fig:compare_known_unknown}, we observe that the value of root decreases when we use the estimated boundary $\hat{x}_-$ compared to the true boundary $x_-$, and the gradient decreases i.e.

\begin{displaymath}r(\hat{x}_-) < r(x_-)<0, \text{and}\; 0<g(\hat{x}_-) < g(x_-).
\end{displaymath}

Consequently, the estimated $\lambda$ decreases, i.e. $\hat{\lambda}_{\hat{x}_-} < \hat{\lambda}_{x_-}$. We give an intuitive justification for this observation. Since the histogram data $(h_i,x_i)_{1\leq i \leq b_l}$ are the same but $\log(x_i-\hat{x}_-)<\log(x_i-x_-)$ for the same value of $\log(h_i)$ for all $1\leq i \leq b_l$, where the corresponding graph is shifted leftwards when we switch from true boundary $x_-$ to estimated boundary $\hat{x}_-$. This explains why the root decreases i.e. $r(\hat{x}_-) < r(x_-)$. 

Moreover, the shift is smaller for larger values of $x_i$, as we are in the logarithm scale, i.e. 
\begin{displaymath}
    0< \log(x_i-x_-)- \log(x_i-\hat{x}_-) < \log(x_{i+1}-x_-)- \log(x_{i+1}-\hat{x}_-).
\end{displaymath}
Hence, the data points above the horizontal line ($\log(h) = 0$) shift less than those below it. This leads to a flatter fitted curve and hence a smaller gradient: $g(\hat{x}_-) < g(x_-)$. The arguments for the higher-order method are not as clear, since the candidate function is not quadratic, but numerical experiments suggest a qualitatively similar behaviour.

\begin{Remark}
    The reasoning above is heuristic rather than rigorous. It provides intuition for the observation that the estimation $\hat{\lambda}$ decreases when the estimated boundary $\hat{x}_-$ is used instead of the true boundary $x_-$. A formal proof would require a detailed study of the optimisation problem in (\ref{eq:quad_method}), which can be tedious, and counterexamples may exist in pathological cases.
\end{Remark}

\subsection{Changes of $\hat{\lambda}$ estimations versus different boundary estimates}
We investigate how inaccuracies in boundary estimation affect the $\lambda$ estimations. Recall that since the boundary $x_-$ is unknown, we estimate it by taking the midpoint of an empty bin on the left end of the histogram, denoted as $\hat{x}_-$. For a general estimated boundary $\hat{x}'_- \in [x_-, \hat{x}_-]$, we denote the corresponding estimator by $\hat{\lambda}_{\hat{x}'_-}$. 
We compute the difference of estimators, $\hat{\lambda}_{x_-} - \hat{\lambda}_{\hat{x}'_-}$ for different boundary estimations, illustrated in Figure~\ref{fig:boundary_linear_estimate} as a function of the boundary errors $\hat{x}'_--x_-$ on a $\log-\log$ scale, for both methods (\ref{eq:quad_method}) and (\ref{eq:quad_method_strict}). We observe that the estimates $\hat{\lambda}_{\hat{x}'_-}$ decrease (i.e. the values of $\hat{\lambda}_{x_-} - \hat{\lambda}_{\hat{x}'_-}$ increase), as the errors of the boundary estimation $\hat{x}'_- - x_-$ increase.

More preceisely, Figure~\ref{fig:boundary_linear_estimate}, exhibits a linear relationship between between $\log(\hat{\lambda}_{x_-}-\hat{\lambda}_{\hat{x}'_-})$ and $\log(\hat{x}'_- -x_-)$, providing thus an indirect almost-affine relationship between $\hat{\lambda}_{x_-}-\hat{\lambda}_{\hat{x}'_-}$ and $x_--\hat{x}'_-$. Indeed, since $\log(\hat{\lambda}_{x_-}-\hat{\lambda}_{\hat{x}'_-})\approx \alpha + m\log(x_--\hat{x}'_-)$, for some constant $\alpha$ and $m\approx 1$, this implies that $\hat{\lambda}_{x_-}-\hat{\lambda}_{\hat{x}'_-} \approx A(x_--\hat{x}'_-)^m$, for some $A = \exp(\alpha)>0$. By adding and subtracting $\lambda$ on the left-hand side, we obtain that
\[
    \hat{\lambda}_{\hat{x}'_-}-\lambda \approx (\hat{\lambda}_{x_-}-\lambda) - A(x_--\hat{x}'_-)^m.
\]
This highlights that the constant $A$ governs the sensitivity of the estimator to boundary errors. While $A$ clearly depends on $\lambda$, it may also be influenced by higher-order terms or system-specific characteristics.

We repeat this procedure for the leading-order method on the nonlinear map with uniform noise as shown in Figure~\ref{fig:boundary_non_linear_estimate}(a) and the higher-order method on the nonlinear map with truncated normal noise illustrated in Figure~\ref{fig:boundary_non_linear_estimate}(b). Both cases reproduce the linear trend already seen in the linear example.

\begin{Remark}
    In addition to the boundary estimations, another factor that can affect the accuracy of the $\hat{\lambda}$ estimation is the regression error arising from the optimisation problems in \eqref{eq:quad_method} and \eqref{eq:quad_method_strict}. Understanding the role of these factors in greater depth is of practical interest, but beyond the scope of this paper.
\end{Remark}

\begin{figure}[!htbp]
    \centering
    \begin{overpic}[width=.49\textwidth]{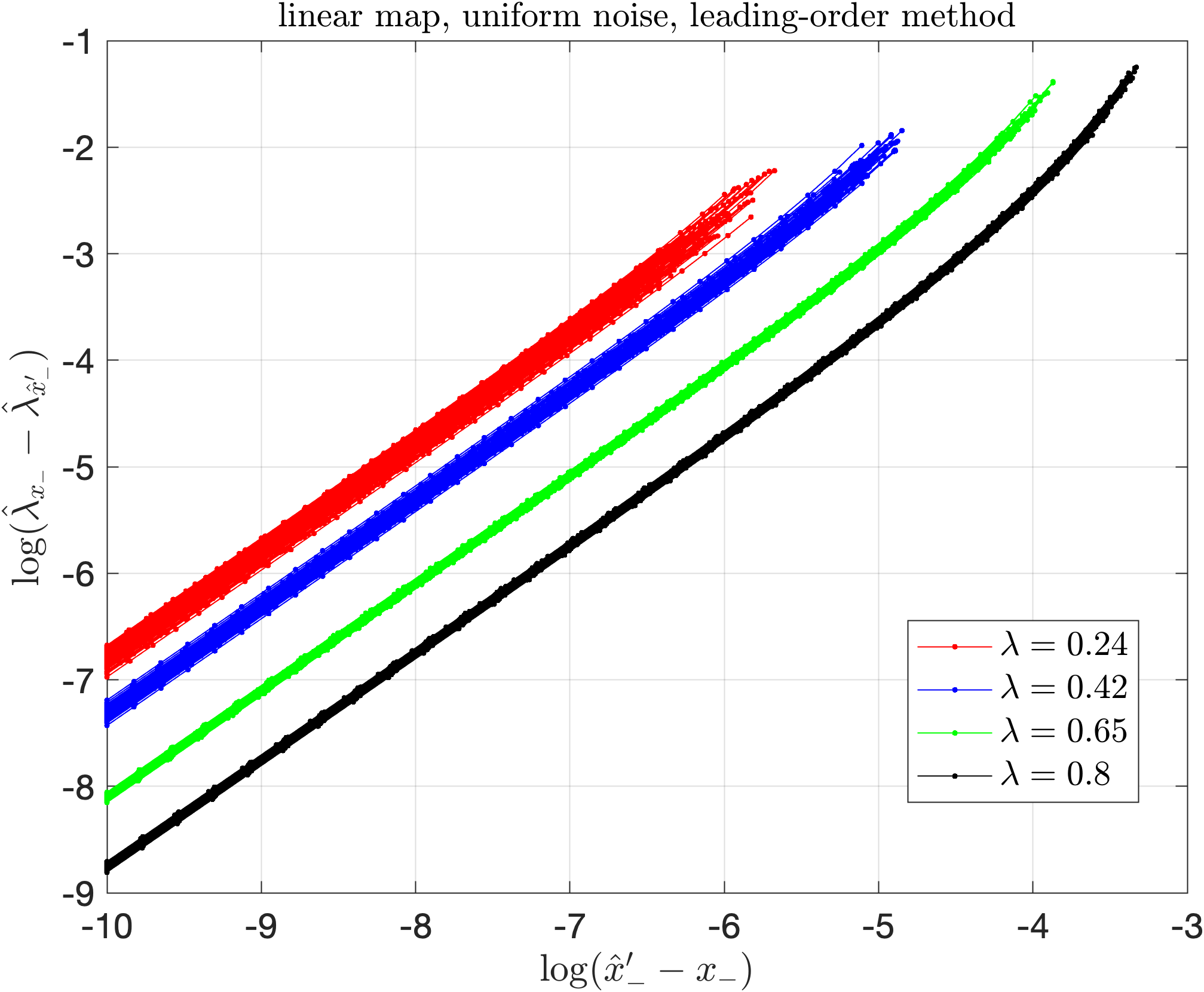}
    \put(0,0){
    (a)
    }
    \end{overpic}
    \begin{overpic}[width=.49\textwidth]{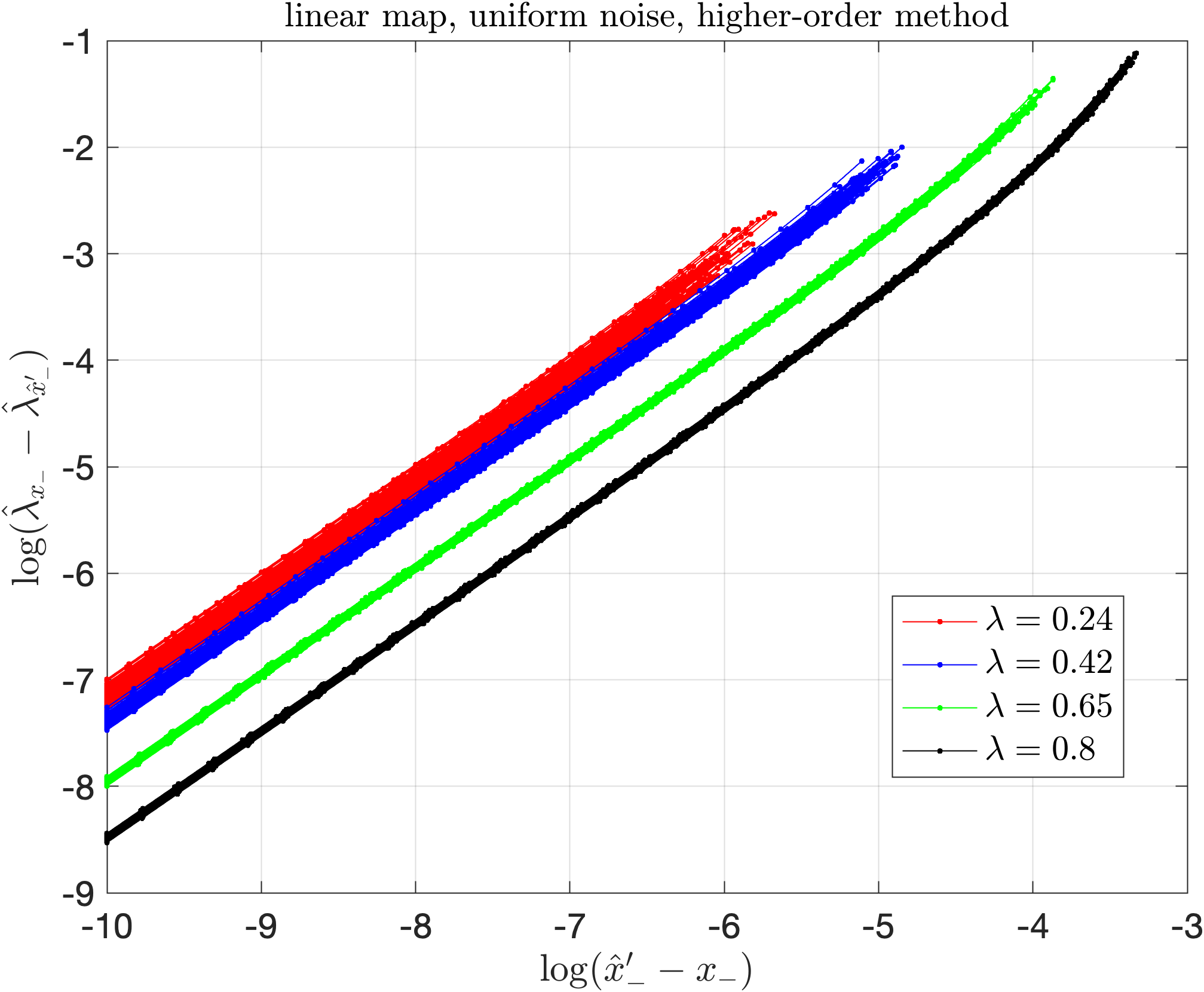}
    \put(0,0){
    (b)
    }
    \end{overpic}
    \caption{The differences of estimations, $\hat{\lambda}_{x_-} - \hat{\lambda}_{\hat{x}'_-}$, where $\hat{\lambda}_{x_-}$ is the estimation using the true boundary $x_-$, against the distance of estimated boundary $\hat{x}'_-$ and $x_-$ for 100 distinct noise realisations on the logarithm scale. They correspond to four distinct $\lambda$ values, using the leading-order method (\ref{eq:quad_method}) in (a) and the higher-order method (\ref{eq:quad_method_strict}) in (b). The average slopes for (a) are 1.03, 1.04, 1.06, 1.10, while for (b), the slopes are 1.02, 1.03, 1.05, 1.08 for $\lambda = 0.24, 0.42, 0.65, 0.8$ respectively. This indicates a linear relation with slope one.}
    \label{fig:boundary_linear_estimate}
\begin{overpic}[width=.49\textwidth]{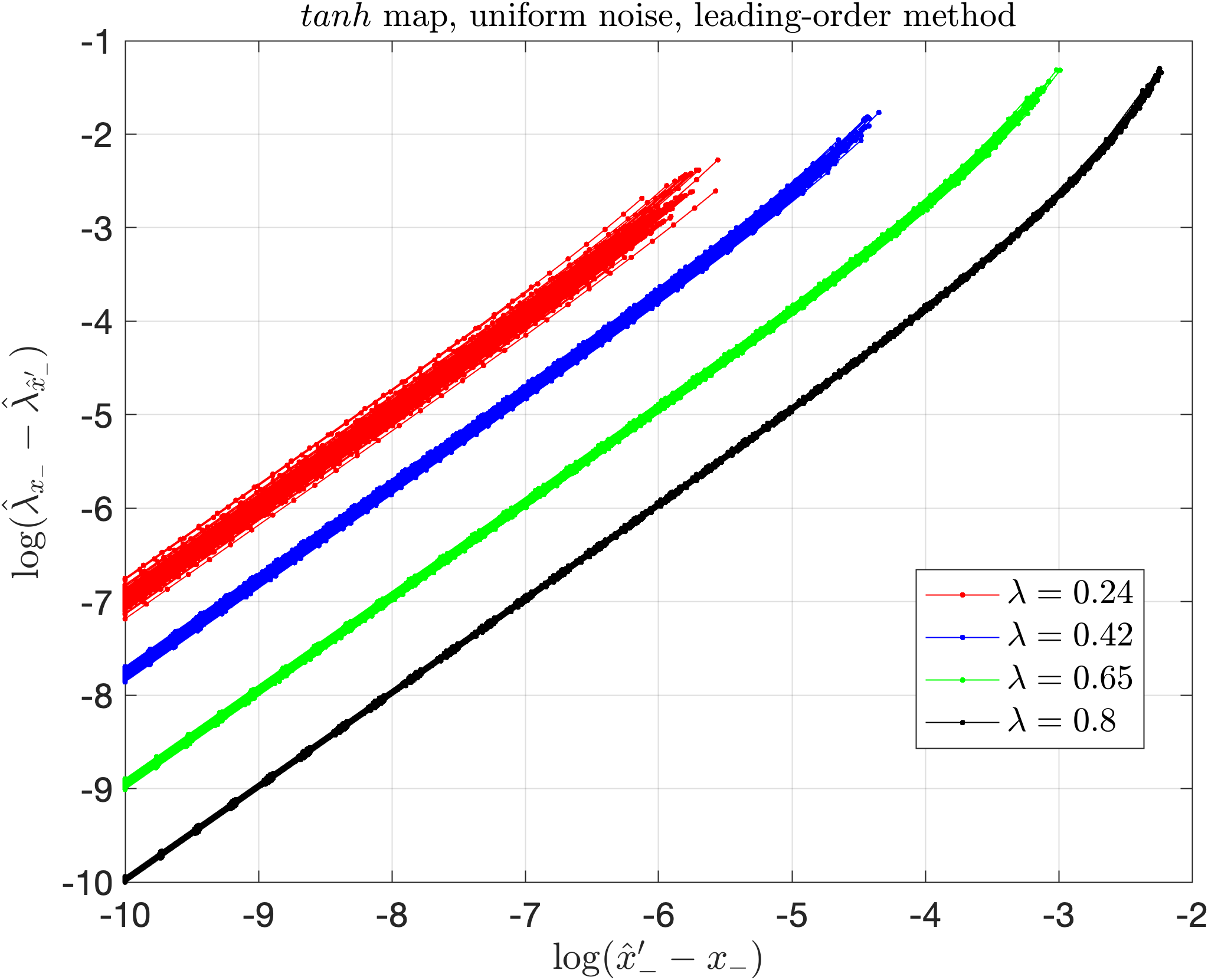}
    \put(0,0){
    (a)
    }
    \end{overpic}
    \begin{overpic}[width=.49\textwidth]{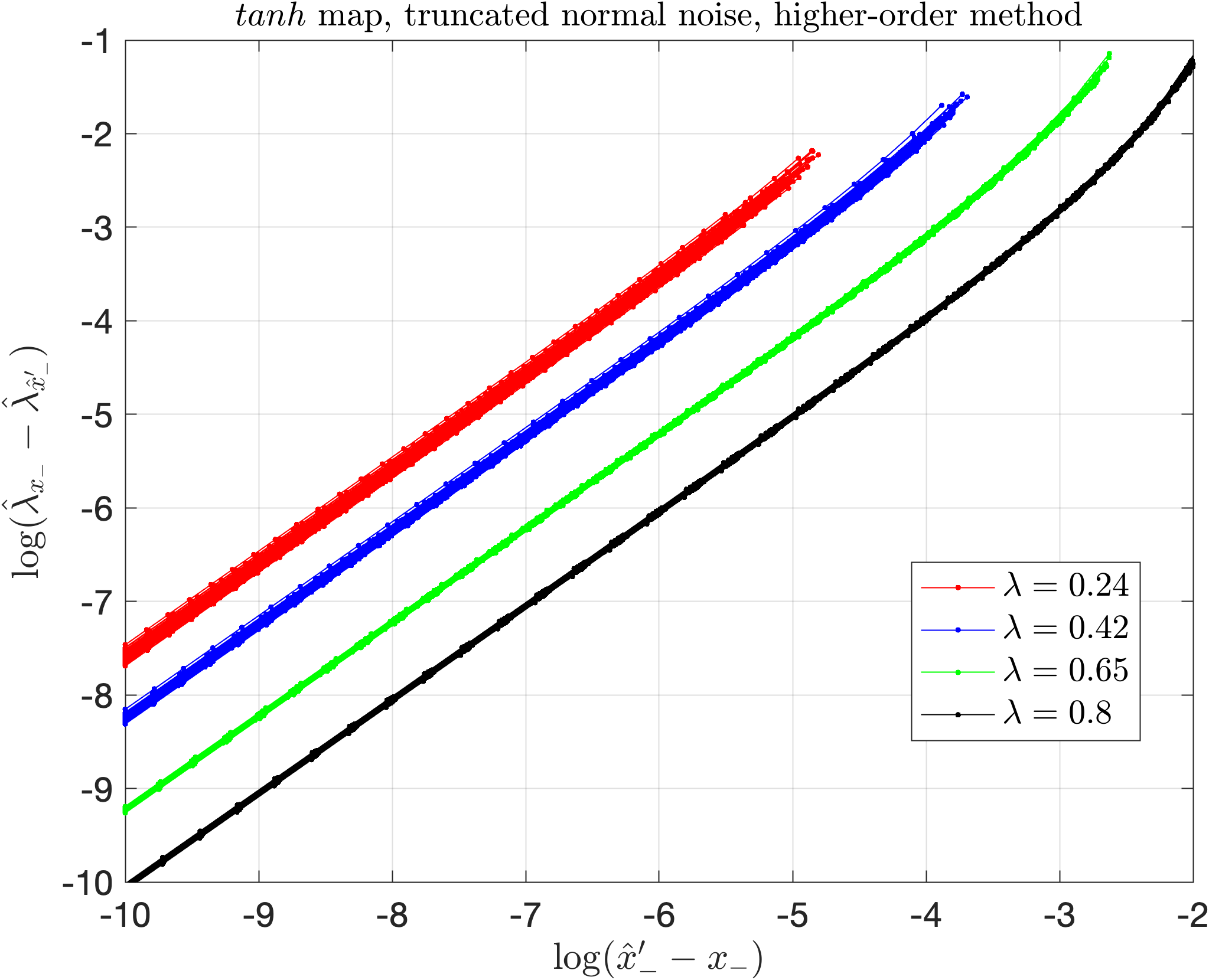}
    \put(0,0){
    (b)
    }
    \end{overpic}
    \caption{Illustration of the differences of estimations $\hat{\lambda}-\hat{\lambda}_{x_-}$, against the boundary discrepancy $\hat{x}_- - x_-$, in the logarithm scale for 100 different noise realisations, for different values of $a$ where $\lambda = 0.24, 0.42, 0.65, 0.8$, in the nonlinear map with bounded noise (\ref{eq:non-linear_eq}). In (a), the noise is uniform and we apply the leading-order method (\ref{eq:quad_method}), while in (b), truncated normal noise is taken and we apply the higher-order method (\ref{eq:quad_method_strict}). The average slopes for (a) are 1.03, 1.04, 1.07, 1.09, while for (b), the slopes are 1.02, 1.04, 1.07, 1.10 for $\lambda = 0.24, 0.42, 0.65, 0.8$ respectively. Similar to the linear case, this indicates a linear relation with slope one.}
    \label{fig:boundary_non_linear_estimate}
\end{figure}

%% file: reference.bib
@misc{Tey2024,
  author = {Tey, Wei Hao},
  title = {Early indicator for bounded noise model},
  year = {2024},
  publisher = {GitHub},
  journal = {GitHub repository},
  howpublished = {\url{https://github.com/weihaotey/early_indicator_bounded_noise}},
  review="\url{https://github.com/weihaotey/early_indicator_bounded_noise}"
}

@phdthesis{olicon2021critical,
  title={Critical behaviour of random diffeomorphisms: quasi-stationary measures and escape times},
  author={Olicon Mendez, Guillermo},
  year={2021},
  institute={Imperial College London},
  doi = {10.25560/90180},
  review="\href{	https://doi.org/10.25560/90180}{$\;$spiral.imperial.ac.uk}"
}

@unpublished{lambhenon,
  title={Bifurcation of the {H}\'{e}non map with additive bounded noise},
  author={Lamb, Jeroen S. W. and Rasmussen, Martin and Tey, Wei Hao},
  institution={\href{https://arxiv.org/abs/2504.02776}{arxiv.org 2504.02776}},
  year={2025},
  doi={10.48550/arXiv.2504.02776}
}

@unpublished{olicon24tail,
  title={Tail behaviour of stationary densities for one-dimensional random diffeomorphisms},
  author={Lamb, Jeroen S. W. and Olic{\'o}n-M{\'e}ndez, Guillermo and Rasmussen, Martin},
  institution={\href{https://arxiv.org/abs/2410.18554}{arxiv.org 2410.18554}},
  year={2024},
  doi={10.48550/arXiv.2410.18554}}

@article{aihara2022dynamical,
  title={Dynamical network biomarkers: Theory and applications},
  author={Aihara, Kazuyuki and Liu, Rui and Koizumi, Keiichi and Liu, Xiaoping and Chen, Luonan},
  journal={Gene},
  volume={808},
  pages={145997},
  year={2022},
  publisher={Elsevier},
  doi={10.1016/j.gene.2021.145997},
  review="\href{	https://doi.org/10.1016/j.gene.2021.145997}{$\;$https://doi.org/10.1016/j.gene.2021.145997}"
}

@article{araujo2000attractors,
  title={Attractors and time averages for random maps},
  author={Ara{\'u}jo, V{\i}tor},
  journal={Annales de l'Institut Henri Poincar{\'e} C, Analyse non lin{\'e}aire},
  volume = {17},
  number = {3},
  pages = {307-369},
  year = {2000},
  issn = {0294-1449},
  doi = {10.1016/S0294-1449(00)00112-8},
  review = "\href{https://doi.org/10.1016/S0294-1449(00)00112-8}{$\;$https://doi.org/10.1016/S0294-1449(00)00112-8}"
}

@article{Ashwinetal12,
  title={Tipping points in open systems: bifurcation, noise-induced and rate-dependent examples in the climate system},
  author={Ashwin, P. and Wieczorek, S. and Vitolo, R. and Cox, P.},
  journal={Philosophical Transactions of the Royal Society A},
  volume = {370},
  pages = {1166–-1184},
  year = {2012},
  doi = {10.1098/rsta.2011.0306},
  review = "\href{https://doi.org/10.1098/rsta.2011.0306}{$\;$https://doi.org/10.1098/rsta.2011.0306}"
}

@article{kuehn13,
    author = {C. Kuehn},
    title = {A mathematical framework for critical transitions: Normal forms, variance, and applications},
    journal = {Journal of Nonlinear Science},
    volume = {23},
    pages = {457--510},
    year = {2013},
    doi = {10.1007/s00332-012-9158-x},
    review = "\href{https://doi.org/10.1007/s00332-012-9158-x}{$\;$https://doi.org/10.1007/s00332-012-9158-x}"
}

@book{Kifer86,
  title={Ergodic theory of random transformations},
  author={Kifer, Y.},
  volume={10},
  year={1986},
  publisher={Birkhauser},
  doi={10.1007/978-1-4684-9175-3},
  review = "\href{https://doi.org/10.1007/978-1-4684-9175-3}{$\;$https://doi.org/10.1007/978-1-4684-9175-3}"
}

@article{zmarrou2007bifurcations,
  title={Bifurcations of stationary measures of random diffeomorphisms},
  author={Zmarrou, Hicham and Homburg, Ale Jan},
  journal={Ergodic Theory and Dynamical Systems},
  volume={27},
  number={5},
  pages={1651--1692},
  year={2007},
  publisher={Cambridge University Press},
  doi = {10.1017/S0143385707000077},
  review="\href{http://doi.org/10.1017/S0143385707000077}{$\;$http://doi.org/10.1017/S0143385707000077}"
}

@article{lamb2015topological,
  title={Topological bifurcations of minimal invariant sets for set-valued dynamical systems},
  author={Lamb, Jeroen S. W. and Rasmussen, Martin and Rodrigues, Christian},
  journal={Proceedings of the American Mathematical Society},
  volume={143},
  number={9},
  pages={3927--3937},
  year={2015},
  doi = {10.1090/S0002-9939-2015-12544-0},
  review="\href{http://doi.org/10.1090/S0002-9939-2015-12544-0}{$\;$http://doi.org/10.1090/S0002-9939-2015-12544-0}"
}

@article{kuehn2018early,
  title={Early-warning signals for bifurcations in random dynamical systems with bounded noise},
  author={Kuehn, Christian and Malavolta, Giuseppe and Rasmussen, Martin},
  journal={Journal of Mathematical Analysis and Applications},
  volume={464},
  number={1},
  pages={58--77},
  year={2018},
  publisher={Elsevier},
  doi={10.1016/j.jmaa.2018.03.066},
  review="\href{http://doi.org/10.1016/j.jmaa.2018.03.066}{$\;$http://doi.org/10.1016/j.jmaa.2018.03.066}"
}

@article{chen2012detecting,
  title={Detecting early-warning signals for sudden deterioration of complex diseases by dynamical network biomarkers},
  author={Chen, Luonan and Liu, Rui and Liu, Zhi-Ping and Li, Meiyi and Aihara, Kazuyuki},
  journal={Scientific reports},
  volume={2},
  number={1},
  pages={342},
  year={2012},
  publisher={Nature Publishing Group UK London},
  doi={10.1038/srep00342},
  review="\href{http://doi.org/10.1038/srep00342}{$\;$http://doi.org/10.1038/srep00342}"
}

@book{d2013bounded,
  title={Bounded noises in physics, biology, and engineering},
  author={d'Onofrio (ed.), Alberto },
  year={2013},
  publisher={Birkh\"{a}user New York},
  doi={10.1007/978-1-4614-7385-5},
  review={\href{https://doi.org/10.1007/978-1-4614-7385-5}{$\;$https://doi.org/10.1007/978-1-4614-7385-5}}
}

@article{scheffer2001catastrophic,
  title={Catastrophic shifts in ecosystems},
  author={Scheffer, Marten and Carpenter, Steve and Foley, Jonathan A and Folke, Carl and Walker, Brian},
  journal={Nature},
  volume={413},
  number={6856},
  pages={591--596},
  year={2001},
  publisher={Nature Publishing Group UK London},
  doi={10.1038/35098000},
  review="\href{https://doi.org/10.1038/35098000}{$\;$https://doi.org/10.1038/35098000}"
}

@article{scheffer2009early,
  title={Early-warning signals for critical transitions},
  author={Scheffer, Marten and Bascompte, Jordi and Brock, William A and Brovkin, Victor and Carpenter, Stephen R and Dakos, Vasilis and Held, Hermann and Van Nes, Egbert H and Rietkerk, Max and Sugihara, George},
  journal={Nature},
  volume={461},
  number={7260},
  pages={53--59},
  year={2009},
  publisher={Nature Publishing Group},
  doi={10.1038/nature08227}
}

@unpublished{kourliouros2023persistence,
  title={Persistence of Minimal Invariant Sets for Certain Set-Valued Dynamical Systems: A Boundary Mapping Approach},
  author={Kourliouros, Konstantinos and Lamb, Jeroen S. W. and Rasmussen, Martin and Tey, Wei Hao and Timperi, Kalle and Turaev, Dmitry},
  doi={10.48550/arXiv.2303.01895},
  institution={\href{https://arxiv.org/abs/2303.01895}{arxiv.org 2303.01895}},
  year={2023}
}

@article{held2004detection,
  title={Detection of climate system bifurcations by degenerate fingerprinting},
  author={Held, Hermann and Kleinen, Thomas},
  journal={Geophysical Research Letters},
  volume={31},
  number={23},
  year={2004},
  publisher={Wiley Online Library},
  doi={10.1029/2004GL020972},
  review="\href{https://doi.org/10.1029/2004GL020972}{$\;$https://doi.org/10.1029/2004GL020972}"
}

@article{dakos2008slowing,
  title={Slowing down as an early warning signal for abrupt climate change},
  author={Dakos, Vasilis and Scheffer, Marten and Van Nes, Egbert H and Brovkin, Victor and Petoukhov, Vladimir and Held, Hermann},
  journal={Proceedings of the National Academy of Sciences},
  volume={105},
  number={38},
  pages={14308--14312},
  year={2008},
  publisher={National Acad Sciences},
  doi={10.1073/pnas.0802430105},
  review="\href{
https://doi.org/10.1073/pnas.0802430105}{$\;$https://doi.org/10.1073/pnas.0802430105}"
}

@article{carpenter2006rising,
  title={Rising variance: a leading indicator of ecological transition},
  author={Carpenter, Stephen R and Brock, William A},
  journal={Ecology letters},
  volume={9},
  number={3},
  pages={311--318},
  year={2006},
  publisher={Wiley Online Library},
  doi={10.1111/j.1461-0248.2005.00877.x},
  review="\href{https://doi.org/10.1111/j.1461-0248.2005.00877.x}{$\;$https://doi.org/10.1111/j.1461-0248.2005.00877.x}"
}

@article{biggs2009turning,
  title={Turning back from the brink: detecting an impending regime shift in time to avert it},
  author={Biggs, Reinette and Carpenter, Stephen R and Brock, William A},
  journal={Proceedings of the National academy of Sciences},
  volume={106},
  number={3},
  pages={826--831},
  year={2009},
  publisher={National Acad Sciences},
  doi={10.1073/pnas.0811729106},
  review="\href{https://doi.org/10.1073/pnas.0811729106}{$\;$https://doi.org/10.1073/pnas.0811729106}"
}

@article{guttal2008changing,
  title={Changing skewness: an early warning signal of regime shifts in ecosystems},
  author={Guttal, Vishwesha and Jayaprakash, Ciriyam},
  journal={Ecology letters},
  volume={11},
  number={5},
  pages={450--460},
  year={2008},
  publisher={Wiley Online Library},
  doi={10.1111/j.1461-0248.2008.01160.x},
  review="\href{https://doi.org/10.1111/j.1461-0248.2008.01160.x}{$\;$https://doi.org/10.1111/j.1461-0248.2008.01160.x}"
}

@article{feng2024early,
  title={Early warning indicators via latent stochastic dynamical systems},
  author={Feng, Lingyu and Gao, Ting and Xiao, Wang and Duan, Jinqiao},
  journal={Chaos: An Interdisciplinary Journal of Nonlinear Science},
  volume={34},
  number={3},
  year={2024},
  publisher={AIP Publishing}, 
  doi={10.1063/5.0195042},
  review="\href{https://doi.org/10.1063/5.0195042}{$\;$https://doi.org/10.1063/5.0195042}"
}

@incollection{homburg2013bifurcations,
  title={Bifurcations of random differential equations with bounded noise},
  author={Homburg, Ale Jan and Young, Todd R and Gharaei, Masoumeh},
  booktitle={Bounded Noises in Physics, Biology, and Engineering},
  pages={133--149},
  year={2013},
  publisher={Birkh\"{a}user New York},
  doi={10.1007/978-1-4614-7385-5_9},
  review="\url{https://doi.org/10.1007/978-1-4614-7385-5_9}"
}

@phdthesis{tey2022minimal,
  title={On minimal invariant sets of certain set-valued dynamical systems: a boundary map approach},
  author={Tey, Wei Hao},
  year={2022},
  school={Imperial College London},
  doi = {10.25560/100849},
  review="\href{http://doi.org/10.25560/100849}{$\;$spiral.imperial.ac.uk}"
}

@article{drijfhout2025shutdown,
  title={Shutdown of northern Atlantic overturning after 2100 following deep mixing collapse in CMIP6 projections},
  author={Drijfhout, Sybren and Angevaare, Joran R and Mecking, Jennifer and van Westen, Ren{\'e} M and Rahmstorf, Stefan},
  journal={Environmental Research Letters},
  volume={20},
  number={9},
  pages={094062},
  year={2025},
  publisher={IOP Publishing},
  doi={10.1088/1748-9326/adfa3b},
  review="\href{http://doi.org/10.1088/1748-9326/adfa3b}{$\;$http://doi.org/10.1088/1748-9326/adfa3b}"
}
